%% file: DensiteRepQuant.tex
\documentclass[12pt]{amsart}
\usepackage[T1]{fontenc}
\usepackage{amsmath,amssymb,amsthm,graphicx, url, verbatim, float, ulem,bm}
\usepackage[hidelinks]{hyperref}
\usepackage[capitalize]{cleveref}
\usepackage{enumerate}
\usepackage{xcolor, relsize}
\usepackage{tikzit}
\input{tikzstyles.tikzstyles}
\usepackage{tikz-cd}

\newtheorem{lemma}{Lemma}[section]
\newtheorem{proposition}[lemma]{Proposition}
\newtheorem{theorem}[lemma]{Theorem}
\newtheorem*{theorem*}{Theorem}
\newtheorem{corollary}[lemma]{Corollary}
\newtheorem*{corollary*}{Corollary}
\newtheorem{question}[lemma]{Question}

\newtheorem*{claim}{Claim}

\theoremstyle{definition}

\newtheorem{definition}[lemma]{Definition}
\newtheorem{remark}[lemma]{Remark}

\newtheorem{notation}[lemma]{Notation}

\newtheorem*{namedtheorem}{\theoremname}
\newcommand{\theoremname}{testing}

\newcommand{\ul}{\underline{\lambda}}
\newcommand{\ra}{\rightarrow}
\newcommand{\lra}{\longrightarrow}
\newcommand{\PGL}{\mathrm{PGL}}
\newcommand{\PU}{\mathrm{PU}}
\newcommand{\SU}{\mathrm{SU}}
\newcommand{\SL}{\mathrm{SL}}
\newcommand{\SO}{\mathrm{SO}}
\newcommand{\PSO}{\mathrm{PSO}}
\newcommand{\PSU}{\mathrm{PSU}}
\newcommand{\PSL}{\mathrm{PSL}}
\newcommand{\GL}{\mathrm{GL}}
\newcommand{\scc}{simple closed curve}

\newcommand{\Gal}{\mathrm{Gal}}
\newcommand{\Or}{\mathcal{O}}

\newcommand{\Zzp}{\mathbb{Z}[\zeta_p]}
\newcommand{\Zrp}{\mathbb{Z}[\zeta_p+\zeta_p^{-1}]}

\newcommand{\bfG}{\mathbf{G}}
\newcommand{\bfV}{\mathbf{V}}
\newcommand{\bfT}{\mathbf{T}}
\newcommand{\oV}{\overline{V}}

\DeclareMathOperator{\Mod}{\mathrm{Mod}}
\DeclareMathOperator{\PMod}{\mathrm{PMod}}
\DeclareMathOperator{\Modtilde}{\mathrm{M\tilde{o}d}}
\newcommand{\Modgn}{\PMod(\Sigma_{g,n})}
\newcommand{\Modtgn}{\Modtilde(\Sigma_g^n)}

\newcommand{\charac}{\mathrm{char}\:}

\newcommand{\Z}{\mathbb{Z}}
\newcommand{\R}{\mathbb{R}}
\newcommand{\Q}{\mathbb{Q}}
\newcommand{\C}{\mathbb{C}}
\newcommand{\N}{\mathbb{N}}
\newcommand{\F}{\mathbb{F}}
\newcommand{\oF}{\overline{\mathbb{F}}}

\newcommand{\id}{\mathrm{id}}

\DeclareMathOperator{\tr}{Tr}
\newcommand{\abs}[1]{\lvert#1\rvert}

\title[Density and surjectivity of \texorpdfstring{$\mathrm{SO}(3)$}{SO(3)}-WRT quantum representations]%
{On the density and surjectivity of \texorpdfstring{$\bm{\mathrm{SO}(3)}$}{SO(3)}-Witten-Reshetikhin-Turaev quantum representations}
\author{Renaud Detcherry}
\date{}

\address{Institut de Math\'ematiques de Bourgogne \& Institut Universitaire de France, UMR 5584 CNRS,
Universit\'e Bourgogne Franche-Comt\'e, F-2100 Dijon, France}
\email{renaud.detcherry@u-bourgogne.fr}

\author{Pierre Godfard}
\address{Department of Mathematics,
University of North Carolina at Chapel Hill,
120 E. Cameron Avenue,
Chapel Hill, NC 27599-3250, USA}
\email{godfard@unc.edu}

\author{Ramanujan Santharoubane}
\address{Laboratoire de math\'ematique d'Orsay, UMR 8628 CNRS,
B\^atiment 307, Universit\'e Paris-Saclay,
91405 ORSAY Cedex, FRANCE}
\email{ramanujan.santharoubane@universite-paris-saclay.fr}

\begin{document}

\begin{abstract}
	In this paper, we establish several new fundamental properties of $\SO(3)$-quantum representations
	$\rho_{p,g,\ul}\colon\PMod (\Sigma_{g,n})\longrightarrow \PSU_{d_{p,g,\ul}}$ of mapping class groups of surfaces,
	at prime-order roots of unity.
	We show that for any surface $\Sigma_{g,n}$ of genus $g\geq 3$, any number $n\geq 0$ of punctures,
	and any coloration $\ul$ of the punctures, $\rho_{p,g,\ul}$ has dense image in the projective unitary group $\PSU_{d_{p,g,\ul}}$,
	extending a landmark result of Larsen and Wang. Moreover, we show that the representations $\rho_{p,g,\ul}$
	are surjective modulo any unramified maximal ideal of $\Z[\zeta_p]$,
	establishing an effective version of strong approximation for these representations.
	
	We also give several applications of our main results to residual finite simpleness of $\PMod(\Sigma_{g,n})$
	(answering a question of Masbaum and Reid); to subnormal cores of some subgroups of $\PMod(\Sigma_{g,n})$;
	to realizability of congruence classes of quantum invariants; to embedding obstructions between $3$-manifolds;
	and to homological stability for mapping class groups with coefficients in $\SO(3)$-quantum representations.
\end{abstract}
\maketitle

\section{Introduction}
\label{sec:intro}

\subsection{Mapping class groups of surfaces and their finite quotients}
For $g,n\geq 0$, let us denote by $\Sigma_{g}^{n}$ a compact oriented surface with genus $g$ and $n$ boundary components,
and by $\Sigma_{g,n}$ an oriented surface of genus $g$ with $n$ punctures. The mapping class group $\Mod(\Sigma_{g}^{n})$ is the quotient
$$\mathrm{Diff}^+(\Sigma_{g}^{n})/\mathrm{Diff}_0(\Sigma_{g}^{n}),$$
of the group of orientation-preserving self-diffeomorphisms by isotopies.
We require diffeomorphisms and isotopies to restrict to the identity on $\partial \Sigma_{g}^{n}$.
The pure mapping class group $\PMod(\Sigma_{g,n})$ is the quotient
$$\mathrm{Diff}^+(\Sigma_{g,n})/\mathrm{Diff}_0(\Sigma_{g,n}),$$
where diffeomorphisms and isotopies must restrict to the identity on the set of marked points. The group $\PMod(\Sigma_{g,n})$
may also be viewed as the quotient of $\Mod(\Sigma_{g}^{n})$ by Dehn twists along boundary components.

The groups $\Mod(\Sigma_{g}^{n})$ are important in many areas of mathematics,
among them low-dimensional topology, geometric group theory and algebraic geometry. Their structure has been studied extensively and in particular,
a classical result of Grossman \cite{Gross} shows that the groups $\Mod(\Sigma_{g}^{n})$ are residually finite.

Despite Grossman's result, finite quotients of mapping class groups are not well understood.
For instance, three of the problems in Ivanov's famous problem list on mapping class groups of surfaces \cite{Iva}
revolve around the structure of finite index subgroups of $\Mod(\Sigma_{g}^{n})$. Furthermore, very few finite simple groups are known
to occur as finite quotients of mapping class groups.
From the homology representations, the groups $\mathrm{PSp}_{2g}(\Z/p\Z)$ are finite quotients of $\Mod(\Sigma_g)$ for any prime $p$.
Moreover, results of Dunfield and Thurston \cite[Theorem 1.3]{DT06} imply that arbitrarily large alternating groups $A_n$
occur as quotients of $\Mod(\Sigma_g)$ for some $g$. Finally, in \cite{MR12}, Masbaum and Reid used $\SO(3)$-Witten-Reshetikhin-Turaev
quantum representations at prime levels to show that for each $g\geq 2$, the group $\Mod(\Sigma_g)$ admits quotients $\mathrm{PSL}_d(\F_q)$
where $d$ and $q$ can be arbitrarily large.
They deduced that every finite group is a quotient of a finite index subgroup of $\Mod(\Sigma_g)$, answering a question of Hamenstädt \cite[§1]{MR12}.
This result was later recovered by Grunewald, Larsen, Lubotzky and Malestein \cite{FLLM} using homological representations.

The above-mentioned results do not, however, provide explicit finite simple quotients of the mapping class groups $\Mod(\Sigma_g)$
other than the quotients of $\mathrm{Sp}_{2g}(\Z)$. In this paper, we provide many examples of such quotients.
Just as an illustration, an immediate corollary of our main results will be:%
\footnote{the quantity $\frac 1 {1440} p^2(p^2+11)(p^2-1)$ is $d_{p,3}$,
see next section for the definition. The formula can be derived by computing $\mathfrak{o}_3+\mathfrak{e}_3$ in \cite[Equations (24) and (25)]{GM14}.}

\begin{corollary}[of \Cref{thm:surjectivity-intro}]
	\label{prop:example} Let $p\geq 5$ be a prime, let $d=\frac 1 {1440} p^2(p^2+11)(p^2-1)$, let $q\neq p$ be another prime and
	$k$ be the order of $q$ modulo $p$. Then $\mathrm{PSL}_d(\F_{q^k})$ is a quotient of $\Mod(\Sigma_3)$ if $k$ is odd,
	and $\mathrm{PSU}_d(\F_{q^k})$ is a quotient of $\Mod(\Sigma_3)$ if $k$ is even.
\end{corollary}

\subsection{Density and surjectivity of \texorpdfstring{$\bm{\SO(3)}$}{SO(3)}-quantum representations of mapping class groups}
\label{sec:introQuantumRep}

Like Masbaum and Reid, we rely on the $\SO(3)$-WRT quantum representations at prime levels to produce finite
simple quotients of mapping class groups of surfaces. Let us recall briefly that for $p\geq 5$ a prime,
and $\Lambda=\lbrace 0,2,\ldots,p-3\rbrace$, any choice $\ul \in \Lambda^n$ of colorings of the punctures of
$\Sigma_{g,n}$ gives rise to a projective representation:\footnote{For $n=0$, we use the notations $\rho_{p,g}$ and $d_{p,g}$.}
$$\rho_{p,g,\ul}:\PMod(\Sigma_{g,n})\longrightarrow \PGL_{d_{p,g,\ul}}(\Z[\zeta_p])$$
where $\zeta_p$ is a primitive $p$-th root of unity, and $d_{p,g,\ul}$ are integers,
which can be computed explicitly using the Verlinde formula.
Those representations furthermore leave a positive definite Hermitian form invariant,
hence take values in some $\mathrm{PU}_{d_{p,g,\ul}}(\Z[\zeta_p])$.
Quotienting the ring of coefficients $\Z[\zeta_p]$ by ideals $J$ turns those representations into representations
$$\rho_{J,g,\ul}:\PMod(\Sigma_{g,n})\longrightarrow \PGL_{d_{p,g,\ul}}(\Z[\zeta_p]/J)$$ with finite images.
In \cite{MR12}, Masbaum and Reid argue that the Zariski density of the representations $\rho_{p,g}$
(which for closed surfaces was proved by Larsen and Wang \cite{LW05}) implies,
thanks to Weisfeiler's Strong Approximation Theorem, that for almost all maximal ideals $J$,
the maps $\rho_{J,g}$ are surjective onto either $\mathrm{PSL}_{d_{p,g}}(\Z[\zeta_p]/J)$ or $\mathrm{PSU}_{d_{p,g}}(\Z[\zeta_p]/J)$.

This method has two limitations: it does not say anything about mapping class groups of surfaces with boundary,
as Larsen and Wang's density result is restricted to closed surfaces, and it is not effective, since it relies on strong approximation.
The main results of this paper, which are given in the following two theorems, address those two limitations.

\begin{theorem}
	\label{thm:density-intro} Let $p\geq 5$ be a prime. Let $g\geq 3$, $n\geq 0$ and $\ul \in \Lambda^n$, then $\rho_{p,g,\ul}$
	has dense image in $\mathrm{PSU}_{d_{p,g,\ul}}$.
\end{theorem}

\begin{theorem}
	\label{thm:surjectivity-intro} Let $p\geq 5$ be a prime, and let $J\subset \Z[\zeta_p]$ be a maximal ideal, with $\Z[\zeta_p]/J\simeq \F_{q^k}$,
	with $q\neq p$ a prime. Let $g\geq 3, n\geq 0$ and $\ul \in \Lambda^n$, then $\rho_{J,g,\ul}$ surjects onto $\mathrm{PSL}_{d_{p,g,\ul}}(\F_{q^k})$
	if $k$ is odd and onto $\mathrm{PSU}_{d_{p,g,\ul}}(\F_{q^k})$ if $k$ is even.
\end{theorem}
Let us make a couple of remarks about the above results.
Theorem \ref{thm:density-intro} is a weaker version of Theorem \ref{thm:density},
which also includes, for each $p\geq 5$ prime, all cases but a finite number in genus $2$
(including the closed genus $2$ surface), and many genus $1$ cases.
As for \Cref{thm:surjectivity-intro}, we do not have a full analog of it in these extra genus $1$ and $2$ cases,
but our proof simplicity of the image modulo $J$ does cover them; see \Cref{thm:simpleGroup}.

The problem of generalizing Larsen and Wang's result to surfaces with boundary was raised by Masbaum and Reid in \cite[Remark 2]{MR16},
and the only progress since has been the work of Funar and Lochak \cite{FL18}, where they adapted
the proof of Larsen and Wang to the case of the surface $\Sigma_{g,1}$ with boundary colored by $p-3$, and with $p\equiv 3 \mod 4$ a large enough prime.

Establishing density for surfaces with boundary and arbitrary coloring of the boundary will be
crucial for some applications of our main results, as we shall see in the next section.

We note that Theorem \ref{thm:density-intro} is logically independent of Larsen and Wang's result and that the method of proof is fairly different.
As in \cite{LW05}, we start out by proving that the representations $\rho_{p,g,\ul}$ are tensor-indecomposable and
deduce that the closures of their images are simple Lie groups, but whereas Larsen and Wang prove tensor-indecomposability
relying on some facts about the representation theory of the groups $\mathrm{PSL}_2(\Z/p\Z)$ and induction on the genus,
we prove it without induction for all surfaces by fully exploiting the irreducibility of all the representations $\rho_{p,g,\ul}$,
the modular functor structure and a saturation property of $\Lambda$.
Then,
we show that the representations $\rho_{p,g,\ul}$ are \emph{weight-multiplicity-free} in the sense of Howe \cite{gelbartSchurLectures19921995}
and use a theorem of Howe classifying weight-multiplicity-free representations of simple Lie groups
to significantly narrow down the possibilities for the closure of the image.
Finally, exploiting again irreducibility of all the representations $\rho_{p,g,\ul}$ and
the modular functor structure to exclude most remaining cases, we show that the $\rho_{p,g,\ul}$ have dense images.

No analog of Theorem \ref{thm:surjectivity-intro} exists in the literature.
In fact, we prove a stronger version, for not necessarily maximal ideals in Theorem \ref{thm:surjectivityImproved}.
One way to think about this result is that it determines, away from the prime $p$,
the closure of the image of $\rho_{p,g,\ul}$ in the profinite group $\mathrm{PU}_{d_{p,g,\ul}}(\widehat{\Z[\zeta_p]})$.
Note that in the genus $1$ and genus $2$ cases covered by \Cref{thm:density}, we prove that the image of $\rho_{J,g,\ul}$
is a finite simple group (but do not identify that group; see \Cref{thm:simpleGroup}).

The proof of Theorem \ref{thm:surjectivity-intro} stems from the realization that most of the arguments involved
in our proof of density for the characteristic zero representations $\rho_{p,g,\ul}$
may be adapted (sometimes without much additional effort)
to prove surjectivity results for the representations $\rho_{J,g,\ul}$.
Indeed, most arguments relying on quantum topology (irreducibility, modular functor structure...)
stay valid for $\rho_{J,g,\ul}$, and this allows us to jointly prove that the representations $\rho_{p,g,\ul}$ and $\rho_{J,g,\ul}$
are tensor-indecomposable and have simple images.
Also, we replace the result of Howe with an analog in positive characteristic due to Zalesskii and Suprunenko \cite{MR934193}.
The main additional hurdle is that we have to first establish that the images of $\rho_{J,g,\ul}$
are finite simple groups of Lie type in the same characteristic as that of $\Z[\zeta_p]/J$.
We exclude other finite simple groups by noticing that the projective representations $\rho_{J,g,\ul}$
have an obstruction to linearizability of large order,
making use of lower bounds, due to Seitz and Zalesskii \cite{seitzMinimalDegreesProjective1993},
on the degrees of projective representations of finite simple Lie groups in non-defining characteristic,
and also relying on bounds on $\dim \rho_{J,g,\ul}$ and on the existence of large abelian $p$-groups in the image of $\rho_{J,g,\ul}$.

One final theorem completes our results about $\SO(3)$-quantum representations.
The following is a special case of Theorem \ref{thm:asymptFaith}.

\begin{theorem}\label{thm:asymptFaith-intro}
	Let $\Sigma_{g}^{n}$ be a compact oriented surface and let $(\ul (p))_{p\geq 5,\, p \, \textrm{prime}}$ be a sequence
	such that $\ul(p)\in \Lambda_p^n$ and $\frac{1}{p}\ul(p) \underset{p\rightarrow\infty}{\longrightarrow}(1/2,\ldots,1/2)$. Then
	$$\underset{p\geq 5,\, p \, \textrm{prime}}{\bigcap}\ker \rho_{p,g,\ul(p)}=Z(\Mod(\Sigma_{g}^{n})).$$
\end{theorem}

Theorem \ref{thm:asymptFaith-intro} may be seen as a refined version of the asymptotic faithfulness of $\SO(3)$-WRT quantum representations,
which was proved independently by Andersen \cite{A06} and Freedman, Walker, and Wang \cite{FWW02}.
The more general but more technical to state Theorem \ref{thm:asymptFaith} proves asymptotic faithfulness
for sequences of colorings converging to other limits than $(1/2,\ldots,1/2)$.
The proof of this result is much shorter than that of \Cref{thm:density-intro,thm:surjectivity-intro}. However, the question whether
one can establish such refined versions of asymptotic faithfulness was raised in \cite[Remark 4.4]{FL18}. Moreover, Theorem \ref{thm:asymptFaith-intro}
will turn out to be important for some of the applications described in the next section.

\subsection{Applications}
\label{sec:applicationsIntro}

While \Cref{thm:density-intro,thm:surjectivity-intro,thm:asymptFaith-intro} are concerned with $\mathrm{SO}(3)$-quantum representations,
they turn out to have applications to the structure of the groups $\Mod(\Sigma_{g,n})$ (or $\Mod(\Sigma_g^n)$) themselves.
\subsubsection{Subnormal cores of some subgroups of mapping class groups}
\label{sec:subnormalIntro}

Our first application is to give several examples of non-trivial subgroups of $\Mod(\Sigma_g^n)$ with trivial subnormal core.

Given a group $G$, we recall that a subgroup $H$ is said to be \textit{subnormal} in $G$ (and we write $H\triangleleft\triangleleft G$)
if there exists a finite sequence of subgroups
$$H=K_n\triangleleft K_{n-1}\triangleleft \dotsb \triangleleft K_1\triangleleft K_0=G.$$

For $H$ a subgroup of $G$, we define the \textit{subnormal core} $scor_G(H)$ of $H$ as
$$scor_G(H)=\langle K \subset H \ | K \triangleleft\triangleleft G \rangle.$$

It is clear that $scor_G(H)$ always contains the normal core $core_G(H)$, which is the largest normal subgroup of $G$ contained in $H$,
and that a subgroup has trivial subnormal core if and only if it does not contain any non-trivial subnormal subgroup.

Let us now restrict to $G=\Mod(\Sigma_{g}^{n})$ where $g\geq 2$ and $(g,n)\neq (2,1)$. We can consider the following subgroups of $\Mod(\Sigma_{g}^{n})$:
\begin{itemize}
	\item[(i)]When $n=0$, the handlebody groups $\mathcal{H}_g$, which are the subgroups consisting of mapping classes that extend to a fixed handlebody
	of genus $g$ with boundary $\Sigma_g$.
	\item[(ii)]The centralizers $C_{\Mod(\Sigma_{g}^{n})}(f)$, where $f$ is a finite-order non-central mapping class in $ \Mod(\Sigma_{g}^{n})$.
	\item[(iii)]The normalizers $N_{\Mod(\Sigma_{g}^{n})}(\Gamma)$, where $\Gamma$ is a finite non-central subgroup of $\Mod(\Sigma_{g}^{n})$.
\end{itemize}

As a corollary of our main results, we get:

\begin{corollary}
	\label{cor:subnormalCore} Let $g\geq 2$ and $(g,n)\neq (2,1)$, then the subgroups $\mathcal{H}_{g}$ (if $n=0$), $ C_{\Mod(\Sigma_{g}^{n})}(f)$
	and $ N_{\Mod(\Sigma_{g}^{n})}(\Gamma)$ introduced above have trivial subnormal cores in $\Mod(\Sigma_{g}^{n})$.
\end{corollary}

We expect the corollary to also be a consequence of Thurston's theorem that mapping class groups act minimally on the space of projective measured
laminations (see \cite[Exposé 6]{FLP}). However, it is striking that such a result about the structure of subgroups of the mapping class groups
may also be deduced from properties of quantum representations, as we will do in Section \ref{sec:subnormalCore}.

\subsubsection{Pure mapping class groups of surfaces are residually finite simple}
\label{sec:resFinSimpleIntro}

Recall that a group $G$ is said to be \emph{residually finite simple} if for any non-trivial $g\in G$, there exists a non-abelian finite simple
group $S$ and an epimorphism $G\overset{\varphi}{\twoheadrightarrow}S$, such that $\varphi(g)\neq 1_S$. Note that a residually finite simple group
is necessarily centerless. Our second application is to show that, apart from a few exceptions, mapping class groups of punctured surfaces are
residually finite simple:

\begin{theorem*}[\ref{thm:resSimple}]
	Let $g\geq 1$ and $n\geq 0$ such that $(g,n)\neq (1,0),(1,2),(1,3),(2,0)$, or $(2,1)$.
	Then the pure mapping class group $\PMod(\Sigma_{g,n})$ of a closed surface of genus $g$ with $n$ marked points is residually finite simple.
	
	The same is true for any subgroup of $\PMod(\Sigma_{g,n})$ containing a non-trivial subnormal subgroup.
\end{theorem*}
We note that in the case $(g,n)=(1,0)$ or $(2,0)$, the same is true of the quotient of $\Mod(\Sigma_g)$ by its center.
The case of surfaces without marked points was proved in \cite{MR12}, and the above theorem answers a question of Masbaum and Reid
(see \cite[§6]{MR16}). The motivation behind Masbaum and Reid's question was Ivanov's fifteenth problem \cite{Iva},
which asks whether any finitely generated subgroup $G$ of $\mathrm{Mod}(\Sigma_g)$ for $g\geq 3$ has a nilpotent finite Frattini subgroup $\Phi_f(G)$.
Establishing that $\PMod(\Sigma_{g,n})$ are residually finite simple already leads to new cases where one can answer Ivanov's problem affirmatively,
see \cite[Remark 2]{MR16}.
We expect that the results of this paper may be applied to Ivanov's fifteenth problem much further than what was established in \cite{MR16}.
We hope to return to this question in a follow-up work.

\subsubsection{Realizability of congruence classes of quantum invariants and applications to embedding obstructions}
\label{sec:realizabilityIntro}

In \cite{RT91}, for any prime $p\geq 5$, invariants $RT_p(M,\Gamma,c)$ of closed $3$-manifolds $M$ containing colored embedded ribbon trivalent
graphs $(\Gamma,c)$ were defined; those invariants fit in the framework of TQFTs, whose underlying mapping class group representations are exactly
the $\mathrm{SO}(3)$-quantum representations.
An interesting consequence of Larsen and Wang's density result for the representations $\rho_{p,g}$, proved by Wong \cite{Wong}, is that the
$\mathrm{SO}(3)$-Witten-Reshetikhin-Turaev invariants of closed $3$-manifolds are dense in $\C$. Our results allow us to do more:

\begin{theorem}\label{thm:realizabilityIntro}
	Let $p\geq 5$ be a prime number, let $M$ be a closed compact oriented $3$-manifold and let $n\geq 1$. Let also $J\neq (1-\zeta_p)$ be a maximal
	ideal in $\Z[\zeta_p]$. Then for any map
	$$Z:\Lambda^n\setminus \lbrace (0,\dotsc,0) \rbrace \longrightarrow \Z[\zeta_p]/J$$ there exists a banded link $L\subset M$ with $n$ components,
	such that for any $\ul \in \Lambda^n\setminus\lbrace (0,\ldots,0)\rbrace$,
	$$Z(\ul)\equiv RT_p(M,L,\ul)\mod{J}.$$
\end{theorem}
Note that even for $M=S^3$, so that the $RT_p(M,L,\ul)$ are just colored Jones polynomials of $L$, Theorem \ref{thm:realizabilityIntro} is,
to the authors' knowledge, new. In fact, Theorem \ref{thm:realizabilityIntro} is a special case of Theorem \ref{thm:realizability},
which deals more generally with embeddings of graphs.

One consequence of Theorem \ref{thm:realizability} is that, thanks to the theory of Frohman-Kania-Bartoszynska ideals \cite{FKB},
it allows one to construct complements of handlebodies in a given $3$-manifold $M$ that do not embed in another $3$-manifold $M'$,
if $M$ and $M'$ satisfy some mild conditions:
\begin{corollary*}[\ref{cor:nonEmbedding}]
	Let $M$ and $M'$ be two closed oriented $3$-manifolds, and let $H$ be a disjoint union of handlebodies (of genus at least $1$).
	Assume that for some prime number $p\geq 5$ and some maximal ideal $J$ in $\Z[\zeta_p]$ different from $(1-\zeta_p)$,
	one has $RT_p(M)\in J$ but $RT_p(M')\notin J$. Then there exists an embedding of $H$ in $M$ such that $M\setminus H$ does not embed in $M'$.
\end{corollary*}

We expect the condition on $M$ and $M'$ in Corollary \ref{cor:nonEmbedding} to be realized (for some $p$) as long as $M$ and $M'$ are distinguished
by the collection of invariants $RT_p$.

As an additional remark, let us note that strong approximation for $\SO (3)$-quantum representations was used by the first-named author and Belletti
in \cite{BD26} to construct many $3$-manifolds that do not embed in each other,
and for which embeddings cannot be obstructed by classical methods.
Another application of Theorem \ref{thm:surjectivity-intro} would be to make the results of \cite{BD26} more explicit;
in particular \cite[Theorem 1.3]{BD26} may be upgraded to give explicit lower bounds on the probability of a random $3$-manifold to be non-embeddable
in a fixed $3$-manifold $M$.

\subsubsection{Some exotic characteristic quotients of surface groups}
\label{sec:characQuotientIntro}
A major open problem about mapping class groups of surfaces is the \emph{Congruence Subgroup Problem} (see \cite[§1]{Iva}),
which asks whether all finite index subgroups of mapping class groups $\mathrm{Mod}(\Sigma_g)$ contain a congruence subgroup, that is,
a subgroup of the form
$$N_{\Gamma}=\ker \left( \Mod(\Sigma_g) \longrightarrow \mathrm{Out}(\pi_1(\Sigma_g)/\Gamma) \right)$$
where $\Gamma$ is a finite index characteristic subgroup of $\pi_1(\Sigma_g)$.
One important motivation for this problem is its connection to number theory,
as a positive answer would give an alternative proof of a conjecture of Grothendieck stating that smooth algebraic curves over $\Q$ are determined
up to isomorphism by their algebraic fundamental group; see \cite{Voe91}.

While having no direct application to the Congruence Subgroup Problem, our results allow us to construct some exotic sequences of finite
index characteristic subgroups of surface groups. Indeed, we get:

\begin{theorem}
	\label{thm:congSubPropIntro} Let $g\geq 3$. There exists a sequence $(\Gamma_n)_{n\geq 0}$ of finite index characteristic subgroups
	of $\pi_1(\Sigma_g)$ such that $\underset{n\geq 0}{\cap}\Gamma_n=\lbrace 1\rbrace$ and such that the associated maps
	$$\Mod(\Sigma_g)\longrightarrow \mathrm{Out}(\pi_1(\Sigma_g)/\Gamma_n)$$
	are all trivial.
\end{theorem}

We stress that while this statement may appear as counter-evidence to the Congruence Subgroup Problem, it is actually independent,
as the genus $1$ case shows (see the discussion at the end of \Cref{sec:congSubProp}).
The authors wonder whether, similarly to the Congruence Subgroup Problem, Theorem \ref{thm:congSubPropIntro} may have interesting arithmetic applications.

\subsubsection{Homological stability for mapping class groups with coefficients in \texorpdfstring{$\SO(3)$}{SO(3)} quantum representations}
\label{sec:applicationsHomolStab}
This application was the second-named author's initial motivation in pursuing this project; see \Cref{rmk:boundary}.

Stabilization of the homologies $H_k(\Mod(S_g^1),\Q)$ as $g\ra+\infty$ at fixed degree $k$ is
a celebrated result of Harer \cite{harerStabilityHomologyMapping1985}.
The value of their stable homologies was conjectured by Mumford in 1983
and proved by Madsen and Weiss in 2002 \cite{madsenStableModuliSpace2007}.
Similarly, stabilization and stable values of homologies with local coefficients $H_k(\Mod(S_g^1),M_g)$ have been studied
for some sequences of local coefficients $M_1\ra M_2\ra\dotsb\ra M_g\ra\dotsb$.
The most studied cases where stability is known are those of polynomial local coefficients
(see \cite{randal-williamsHomologicalStabilityAutomorphism2017}).
One property of these polynomial local coefficients
is that $\dim M_n$ grows polynomially in $n$.
An example is the sequence of symplectic representations
$H_1(S_1^1,\Q)\ra H_1(S_2^1,\Q)\ra\dotsb\ra H_1(S_g^1,\Q)\ra\dotsb$,
for which stability and computation of stable homology are known by results of Ivanov \cite{ivanovHomologyStabilityTeichmuller1993}
and Looijenga \cite{looijengaStableCohomologyMapping1995}.

In \cite{zotero-item-2951}, the second-named author proves a conditional homological stability result
for coefficients coming from general quantum representations and obtains, using \Cref{thm:density-intro,thm:surjectivity-intro}
in a crucial way,
unconditional stability results for coefficients coming from $\SO(3)$ quantum representations at prime levels.
Such coefficients are not polynomial, as their ranks grow exponentially with $g$.
This a priori excludes much of the known machinery in homological stability.
However, there are recent works on some local coefficients with such exponential growth:
Putman \cite{putmanPartialTorelliGroups2023} proved homological stability
for local coefficients corresponding to spaces of $G$-covers.
The main result of \cite{zotero-item-2951} is an adaptation of this recent work to the context of quantum representations,
and it allows one to deduce stability in any degree $k\geq 0$ from (a strong version of) stability in degree $0$.

Here, we only state that general result in the case of $\SO(3)$ quantum representations
and briefly explain how it can be used, together with \Cref{thm:density-intro,thm:surjectivity-intro},
to get an unconditional homological stability result.
Let us fix $p\geq 5$ odd.
In that case, the degree $0$ (strong) stability condition we need is that for some integer $C\geq 1$, and any $\lambda,\mu\in\Lambda$,
\begin{equation}\label{eq:conditionStability}
	\Mod(\Sigma_{C+1}^1) \cdot \mathrm{Im}\left( V_{p,C,\lambda}^*\otimes V_{p,C,\mu} \to V_{p,C+1,\lambda}^*\otimes V_{p,C+1,\mu} \right)
	= V_{p,C+1,\lambda}^*\otimes V_{p,C+1,\mu}.
\end{equation}
Under that condition, the main result of \cite{zotero-item-2951} implies that for any $\lambda\in\Lambda$,
\[ H_k(\Mod(\Sigma_g^1), \mathrm{End}(V_{p,g,\lambda})) \to H_k(\Mod(\Sigma_{g+1}^1), \mathrm{End}(V_{p,g+1,\lambda})) \]
are isomorphisms for \( g \ge (C+2)k + 2C + 2 \) and surjections for \( g \ge (C+2)k + 2C + 1 \).
The point is that \Cref{thm:density-intro} can be used to prove \Cref{eq:conditionStability} with $C=2$ over $\C$ and that
\Cref{thm:surjectivity-intro} can be used to prove \Cref{eq:conditionStability} with $C=2$ over $\Z[\zeta_p,\frac 1 p]/J$ for any $J$ maximal.
Combining the homological stability results over $\Z[\zeta_p, \frac 1 p]/J$ as $J$ varies using,
the following stability result over $\Z[\zeta_p,\frac 1 p]$ is obtained.

\begin{theorem}[{\cite{zotero-item-2951}}]
	Let $p \geq 5$ be a prime and let $\lambda\in \{0,2,\dotsc,p-3\}$.
	Then the map
	\begin{equation}\label{eq:stabilityIntegral}
		H_k(\Mod(\Sigma_g^1), \mathrm{End}(V_{p,g,\lambda})) \lra H_k(\Mod(\Sigma_{g+1}^1), \mathrm{End}(V_{p,g+1,\lambda}))
	\end{equation}
	is an isomorphism for $g \geq 4k + 9$ and a surjection for $g \geq 4k + 7$.\footnote{See the citation for details on which stabilization
	maps we consider.}
	Here, the local coefficients $\mathrm{End}(V_{p,g,\lambda})$ are defined over $\Z[\zeta_p,\frac 1 p]$.
\end{theorem}

\Cref{thm:surjectivity-intro} can also be used to prove a similar homological stability result for coefficients
in $\mathrm{End}(V_{p,g,\lambda})^{\otimes r}$ for any $r\geq 1$,
see \cite{zotero-item-2951}.
The isomorphism range is then of the form $g\geq (C_r+2)k+3C_r+3$ with $C_r$ roughly of the order of $\ln(r)$.

\begin{remark}\label{rmk:boundary}
	Even the proof of stability for $H_k(\Mod(\Sigma_g^1), \mathrm{End}(V_{p,g,0}))$
	makes use of \Cref{eq:conditionStability} \emph{for all $\lambda$ and $\mu$}.
	Hence the necessity of a density or surjectivity result for surfaces with boundary to prove homological stability.
\end{remark}

\begin{remark}
	The analog of the condition of \Cref{eq:conditionStability} in Putman's work \cite{putmanPartialTorelliGroups2023}
	is played by a statement on stabilization
	of the space of $G$-covers due to Dunfield and Thurston \cite[Proposition 6.16]{DT06}.
\end{remark}

\subsection{Organization of the paper}
\label{sec:organizationIntro}

\Cref{sec:prelim} is devoted to an introduction to $\SO(3)$ Witten-Reshetikhin-Turaev quantum representations
over $\C$ and modulo maximal ideals $J$, to the proof of several dimension identities (\Cref{sec:dimensionTQFT}),
and to comments on the notion of admissible tuples (\Cref{sec:permissive}).

In \Cref{sec:tensorIndec}, we show that for $p\geq 5$ prime the quantum representations $\rho_{p,g,\ul}$
(or their quotients $\rho_{J,g,\ul}$) are tensor-indecomposable for most choices of surfaces $\Sigma_{g,n}$ with $g\geq 1$
and of coloring data $\ul$ (\Cref{thm:tensor-indec}). We also make some comments on irreducibility of the $\rho_{p,g,\ul}$
and explain why it remains true modulo $J$.

The subject of \Cref{sec:simplicity} is the proof that (the closure of) the image of $\rho_{p,g,\ul}$ and $\rho_{J,g,\ul}$
are simple groups whenever the corresponding representations are tensor-indecomposable. This section is an adaptation of arguments
of Larsen and Wang \cite{LW05}.

The purpose of \Cref{sec:densityProof} is to prove density of the image of $\rho_{p,g,\ul}$ for most choices of surfaces $\Sigma_{g,n}$ with $g\geq 1$
and of coloring data $\ul$ (\Cref{thm:density}).

The aim of \Cref{sec:caracFinie} is to prove surjectivity of $\rho_{J,g,\ul}$ for $J$ a maximal ideal prime to $p$ and $g\geq 3$ (\Cref{thm:surjectivity}).
This is generalized to non-maximal ideals in \Cref{thm:surjectivityImproved}. In \Cref{sec:modp},
we make some comments highlighting how different the situation is modulo $p$.

In \Cref{sec:case-p=5}, we deal with the case $p=5$, which was excluded from our main results.
Surjectivity modulo $J$ prime to $5$ for $(g,n)\neq (1,0)$ is proved in \Cref{thm:casep=5}.

\Cref{sec:asympFaith} is devoted to the proof of asymptotic faithfulness for surfaces with boundary
under some condition on the sequence of boundary colors (\Cref{thm:asymptFaith}).

Finally, the proofs of applications mentioned in \Cref{sec:applicationsIntro} are laid out in \Cref{sec:applications}.

\subsection*{Acknowledgments}
The authors thank Prakash Belkale, Julien Marché, Gregor Masbaum, Andrew Putman, Eric Samperton and Alan Reid
for helpful discussions and correspondences. The IMB, host institution of the first-named author,
receives support from the EIPHI Graduate School (contract ANR-17-EURE-0002).

\section{Preliminaries on \texorpdfstring{$\SO(3)$}{SO(3)}-WRT quantum representations}
\label{sec:prelim}

\subsection{Quantum representations of mapping class groups}
\label{sec:QuantumRepDef}
Let $\Sigma_{g}^{n}$ be a genus $g$ closed, oriented surface with $n$ boundary components.
The $\mathrm{SO}(3)$ topological quantum field theories (TQFTs) take as an input an odd integer $p \geq 3$ and a coloring datum
$\ul\in \{0,2,\dotsc,p-3\}^n$
of the boundary components of $\Sigma_{g}^{n}$. As an output, they give a projective representation
\[\rho_{p,g,\ul} \colon \PMod(\Sigma_g^n)\to\PGL_{d_{p,g,\ul}}(\Z[\zeta_p,\frac{1}{p}]),\]
where the dimension $d_{p,g,\ul}$ depends on all the input data.
The notion of a TQFT was introduced by Witten \cite{Witten},
but the first rigorous construction of a TQFT was carried out by Reshetikhin and Turaev,
using the category of semisimple representations of the quantum group $U_qsl_2$ (see \cite{RT} and \cite{TuraevBook}).
We will work in the TQFT constructed by Blanchet, Habegger, Masbaum, and Vogel in~\cite{bhmv},
wherein an explicit representation associated to a TQFT is constructed using skein theory.

One can define a certain cobordism category $\mathcal{C}$ of closed surfaces with colored banded points, in which the cobordisms are decorated
by uni-trivalent colored banded graphs.

A banded point on a closed oriented surface is an oriented submanifold which is homeomorphic to the unit interval.
If a surface has multiple banded points, we will assume that these intervals are disjoint.
When one wants to study a surface with boundary from the point of view of TQFTs,
one customarily attaches a disk to each boundary component and places a single banded point in the interior of each such disk.
The banded points are moreover colored, which is to say equipped with an even integer in $\Lambda = \{0,2,\ldots,p-3\}$.

By capping off boundary components, we can start with a surface $\Sigma_{g}^{n}$ and produce a closed surface $\hat{\Sigma}_{g}^{n}$
equipped with $n$ colored banded points.

Now, for $p \geq 3$ odd and given a coloring $\ul \in\Lambda^n$, the $\mathrm{SO}(3)$-TQFT defines a finite rank free
$\Z[\zeta_p,\frac{1}{p}]$-module\footnote{For $p\equiv -1 \mod 4$, this is as in \cite{bhmv}.
For the definition over $\Z[\zeta_p,\frac{1}{p}]$ instead of $\Z[\zeta_{4p},\frac{1}{p}]$
in the case $p\equiv 1 \mod 4$; see \cite[§8]{gilmerIntegralityTQFTs2004} and \cite[§13]{gilmerIntegralLatticesTQFT2007}}%
$$V_{p,g,\ul}.$$

Let us now describe a basis of $V_{p,g,\ul}$.
Denote by $y$ the set of $n$ colored banded points on $\hat{\Sigma}_g^n$, and by $\Sigma_g$ the underlying closed surface without colored banded points.
Let $\mathcal{H}$ be a handlebody such that $\partial{\mathcal{H}} = \Sigma_g$,
and let $G\subset \mathcal{H}$ be a uni-trivalent banded graph such that $\mathcal{H}$ retracts to $G$.
We suppose that $G$ meets the boundary of $\mathcal{H}$ exactly at the banded points $y$
and this intersection consists exactly of the degree one ends of $G$.
A $p$-admissible coloring of $G$ is an assignment of an integer in $\Lambda$
to each edge of $G$ such that, at each degree three vertex $v$ of $G$, the three (non-negative integer)
colors $\{a,b,c\}$ coloring edges meeting at $v$ satisfy the following conditions:
\begin{enumerate}
	\item
	$|a-c|\leq b\leq a+c$;
	\item
	$a+b+c\leq 2p-4$;
	\item
	the color of an edge terminating at a banded point $y_i$ must equal the color of $y_i$.
\end{enumerate}
To any $p$-admissible coloring $c$ of $G$
is canonically associated an element of the Kauffman bracket skein module of $\mathcal{H}$
relative to the colored banded points in $\hat{\Sigma}_g^n$.
This skein module element is produced by cabling the edges of $G$ by appropriate Jones-Wenzl idempotents (see \cite[§4]{bhmv} for more details).
It turns out that the vectors associated to $p$-admissible colorings give a basis for $V_{p,g,\ul}$ (see \cite[Theorem 4.14]{bhmv}).

Let us now briefly describe the action of $\Mod(\Sigma_{g}^{n})$ on $V_{p,g,\ul}$.
Any homeomorphism $f$ of $\Sigma_{g}^{n}$ that is the identity on $\partial \Sigma_{g}^{n}$ can be extended by the identity on $\hat{\Sigma}_g^n$,
and one considers the mapping cylinder of $f$. The TQFT functor gives as an output a linear automorphism $\rho_{p,g,\ul}(f)$ of $V_{p,g,\ul}$.
This procedure produces a projective representation
$$ \rho_{p,g,\ul} : \Mod(\Sigma_{g}^{n}) \to \PGL(V_{p,g,\ul}).$$
Due to the so-called framing anomaly, the composition law is well-defined only up to multiplication by a root of unity.

We now describe fusion rules. For any compact surface $\Sigma$ with parametrized boundary and coloring $\ul\in \Lambda^{\partial\Sigma}$,
we will use the notation $V(\Sigma,\ul)$ for the corresponding TQFT vector space. An isomorphism $\Sigma\simeq \Sigma_g^n$
provides an identification $V(\Sigma,\ul)\simeq V_{p,g,\ul}$.
Let $\Sigma$ be a compact surface, and let $\partial_+\Sigma\sqcup\partial_-\Sigma\subset\partial \Sigma$ be two components of its boundary.
Assume for simplicity that $\Sigma$ has $1$ or $2$ connected components and that each component contains $\partial_-\Sigma$ or $\partial_+\Sigma$.
Let $\varphi_{\partial_{\pm}\Sigma}:\partial_{\pm}\Sigma\simeq S^1$ be identifications of these components with $S^1$.

Let $\Sigma_{\pm}$ be the surface obtained from $\Sigma$ by gluing $\partial_+\Sigma$ to $\partial_-\Sigma$
along $\varphi_{\partial_-\Sigma}^{-1}\circ(z\mapsto z^{-1})\circ\varphi_{\partial_+\Sigma}$.
Then for any coloring $\ul$
of the components of $\partial \Sigma_{\pm}$, we have the following fusion isomorphism.
\begin{equation*}
	V(\Sigma_{\pm},\ul)\simeq \bigoplus_{\mu\in\Lambda}V(\Sigma,\mu,\mu,\ul).
\end{equation*}
This decomposition is compatible with the action of $\Mod(\Sigma)$. If $\Sigma=\Sigma_1\sqcup\Sigma_2$
is disconnected, then by $V(\Sigma,\mu,\mu,\ul)$ we mean $V(\Sigma_1,\mu,\ul^1)\otimes V(\Sigma_1,\mu,\ul^2)$
with $\ul$ split into $\ul^1$ and $\ul^2$ according to connected components, and by $\Mod(\Sigma)$ we mean $\Mod(\Sigma_1)\times\Mod(\Sigma_2)$.

\begin{remark}\label{rmk:vacuum}
	The Dehn twists $\tau_1,\ldots,\tau_n$ along the boundary components of $\Sigma_{g}^{n}$ belong to the kernel of $\rho_{p,g,\ul}$,
	hence the representation $\rho_{p,g,\ul}$ factors through a representation of the pure mapping class group
	$$\PMod(\Sigma_{g,n})=\Mod(\Sigma_{g}^{n})/\langle \tau_1,\ldots,\tau_n\rangle.$$
	We note also that if one color of $\ul$ is zero, then we simply do not put any colored banded point on $\hat{\Sigma}_g^n$
	and by the about construction the representation factors through a representation of $\PMod(\Sigma_{g,n-1})$.
\end{remark}

For each $g, n \ge 0$, there is a so-called universal central extension
\begin{equation*}
1 \longrightarrow \mathbb{Z} \longrightarrow \Modtilde(\Sigma_g^n) \longrightarrow \Mod(\Sigma_g^n) \longrightarrow 1.
\end{equation*}
For $g \ge 4$, $H^2(\Sigma_g^n, \mathbb{Z}) \simeq \mathbb{Z}$ (see \cite[Theorem 6.1]{Korkmaz})
and the class of the universal extension is a generator.
This central extension has index 4 in the Maslov index extension more commonly used in TQFT; see \cite[Remark 3.2 and above]{MR12}.

For each $p \ge 3$ odd, $g, n \ge 0$ and $\ul \in \Lambda^n$, $\rho_{p,g,\ul}$ can be linearized into a representation
\begin{equation*}
\widetilde{\rho}_{p,g,\ul} \colon \Modtilde(\Sigma_g^n) \longrightarrow
\GL_{d_{p,g,\ul}} ( \mathbb{Z}[\zeta_p, \frac{1}{p}] ),
\end{equation*}
see \cite{bhmv} and \cite[§3]{MR12}.\footnote{Again, for the definition over $\mathbb{Z}[\zeta_p, \frac{1}{p}]$
instead of $\mathbb{Z}[\zeta_{4p}, \frac{1}{p}]$ in the case $p\equiv 1\mod 4$, we refer to
\cite[§8]{gilmerIntegralityTQFTs2004} and \cite[§13]{gilmerIntegralLatticesTQFT2007}.}

While we will not use this fact except to make some comments in \Cref{sec:modp} below,
let us point out that for $p$ prime, $\rho_{p,g,\ul}$ and $\widetilde{\rho}_{p,g,\ul}$
can in fact be defined over $\mathbb{Z}[\zeta_p]$; see \cite{gilmerIntegralityTQFTs2004,gilmerIntegralLatticesTQFT2007}.
However, the basis and fusion rules mentioned above cannot be defined over $\mathbb{Z}[\zeta_p]$.

\begin{notation}
	We will often simplify $V_{p,g,\ul}$, $\rho_{p,g,\ul}$ and $\widetilde{\rho}_{p,g,\ul}$
	into $V_{g,\ul}$, $\rho_{g,\ul}$ and $\widetilde{\rho}_{g,\ul}$.
\end{notation}

\subsection{Quantum representations modulo ideals}
\label{sec:prelimDedekindDom}
As mentioned above, for $p\geq 5$ prime, the representations $\rho_{p,g,\ul}$ of $\PMod(\Sigma_{g,n})$
may be defined over the ring $\Z[\zeta_p]$ (as projective representations), where $\zeta_p$ is a primitive $p$-th root of unity.
We now review some basic facts about the ideals of the ring $\Z[\zeta_p]$, we refer to \cite[§2, §3 and §4]{Marcus}
for more details. Note that $\Z[\zeta_p]$ is a Dedekind domain and non-zero prime ideals are maximal.
The norm of a non-zero ideal $J$ is the number of elements in the finite ring $\Z[\zeta_p]/J$.
We note that the norm is multiplicative for products of ideals. If $J$ is a maximal ideal,
then $\mathbb{Z}[\zeta_p]/J$ is a finite field of prime characteristic $q$ (where $q$ is the unique prime in $\Z$ such that $q \in J \cap \Z$).
One can factorize the ideal $(q)$, which has norm $q^{p-1}$, as the product
$$(q)=J_1^{\alpha_1} \dotsb J_k^{\alpha_k}$$
where the $J_i$ are distinct maximal ideals of $\Z[\zeta_p]$, $\alpha_i \ge 1$ and where one of the $J_i$ is $J$.
We say that a prime $q$ is ramified over $\Z[\zeta_p]$ if at least one $\alpha_i > 1$.
It is well known that the only ramified prime in $\Z[\zeta_p]$ is $p$ with the factorization
$$(p) = (h)^{p-1}$$
where $h = 1-\zeta_p$ (the ideal $(h)$ is a prime in $\Z[\zeta_p]$). Moreover $\Z[\zeta_p]/(h) = \F_p$.
When $q \neq p$, we have that $\Z[\zeta_p]/J_i\simeq \F_{q^d}$ (the finite field with $q^d$ elements) for any $1\leq i \leq k$,
with $d$ being the order of $q$ modulo $p$. Moreover, the ideals $J_1,\ldots,J_k$ are the orbit of $J$ under the Galois group of $\Q(\zeta_p)$.
It follows that $d$ is even if and only if $J$ is stable under complex conjugation, as $\mathrm{Gal}(\Q(\zeta_p))$
is cyclic and its only involution is complex conjugation.

For $J$ an ideal of $\Z[\zeta_p]$ \emph{prime to $p$}, we will write $\rho_{J,g,\ul}$ for the representation:
$$\begin{array}{rccl}
	\rho_{J,g,\ul}:& \PMod(\Sigma_{g,n})&\longrightarrow & \PGL_{d_{p,g,\ul}}(\Z[\zeta_p]/J)
	\\ & \phi & \longmapsto & \rho_{p,g,\ul}(\phi) \mod J.
\end{array} $$
Note that as $\Z[\zeta_p]/J=\Z[\zeta_p,\frac 1 p]/(J\cdot \Z[\zeta_p,\frac 1 p])$,
we do not need to know that $\rho_{p,g,\ul}$ is defined over $\Z[\zeta_p]$ to define $\rho_{J,g,\ul}$.
In this paper we will mainly focus on the images of $\rho_{J,g,\ul}$ when $J\neq (h)$ is a maximal ideal of $\Z[\zeta_p]$.
The basis given by $p$-admissible colorings of a uni-trivalent graph is defined over $\Z[\zeta_p,\frac{1}{p}]$.
Hence, as $J$ is prime to $p$, these bases persist modulo $J$.

Finally, $\rho_{p,g,\ul}$ stabilizes a Hermitian form $\langle,\rangle$ on $V_{p,g,\ul}$,
which is defined over $\Z[\zeta_p,\frac{1}{p}]$ for $p\equiv -1\mod 4$
and over $\Z[\zeta_{4p},\frac{1}{p}]$ for $p\equiv 1\mod 4$.
When $J\neq (h)$ is a maximal ideal with norm $q^d$ where $d$ is even, $J$ is stable under complex conjugation,
hence complex conjugation induces an involution on $\Z[\zeta_p]/J\simeq \F_{q^d}$. As $d$ is even, $\sqrt{-1}\in \F_{q^d}$,
so $\langle,\rangle$ gives rise to a Hermitian form on the $\F_{q^d}$-vector space $V_{J,g,\ul}$ stabilized by $\rho_{J,g,\ul}$,
even in the case $p\equiv 1\mod 4$.
Therefore, $\rho_{J,g,\ul}$ takes values in $\mathrm{PU}_{d_{p,g,\ul}}(\F_{q^d})$ in that case.
We will also use the notation $h_{p,g,\ul}$ or $h_{g,\ul}$
for $\langle,\rangle$ when disambiguation seems preferable.

As for $\rho_{p,g,\ul}$ and $\widetilde{\rho}_{p,g,\ul}$,
the form $h_{p,g,\ul}$ is in fact, up to some rescaling,
defined over $\mathbb{Z}[\zeta_p]$ \cite[§14]{gilmerIntegralityTQFTs2004}:
for $p\equiv -1\mod 4$, the Hermitian form is defined over $\mathbb{Z}[\zeta_p]$;
for $p\equiv 1\mod 4$, one has to rescale to define the form over $\mathbb{Z}[\zeta_p]$,
and the form either remains Hermitian or becomes skew-Hermitian depending on the genus.
We will only need this fact for the comments made in \Cref{sec:modp}.

\subsection{Some dimension identities}
\label{sec:dimensionTQFT}
Over the course of the proofs of \Cref{thm:density,thm:surjectivity}, various identities on the dimensions $d_{p,g,\ul}$
of the vector spaces $V_{p,g,\ul}$ will come into play.
In this section, we will collect and prove all these identities.
Many of our identities will hold even at non-prime levels,
and may be of independent interest; hence we will state them for general odd levels $p\geq 5$.

\begin{lemma}\label{lemma:dimensionDivisibleByP}
	Let $p\geq 5$ be a prime number. Then for any $g\geq 2$, $n\geq 0$ and $\ul\in\Lambda^n$, the dimension of $V_{p,g,\ul}$ is a multiple of $p$.
\end{lemma}
\begin{proof}
	Let $\gamma$ be a separating simple closed curve in $\Sigma_{g,n}$ that separates a subsurface of the type $\Sigma_{2}^{1}$.
	Splitting along $\gamma$, one gets:
	$$\dim V_{p,g,\ul}=\underset{i\in \Lambda}{\sum}\dim V_{p,2,i}\dim V_{p,g-2,(i,\ul)}.$$
	However, by \cite[Equations (22) and (23)]{GM14}, one has%
\footnote{Note that $\dim (V_{p,2,2j})=D_{2}^{(2j)}=\mathfrak{e}_2^{(2j)}+\mathfrak{o}_2^{(2j)}$ in the notations of \cite{GM14}.}
	$$\dim V_{p,2,2j}=\frac{(2j+1)p^3-(6j^2+6j)p^2+(4j^3+6j^2-1)p}{24},$$
	which is clearly divisible by $p$ if $p\geq 5$.
\end{proof}

The following lemma is a direct consequence of fusion rules.
\begin{lemma}\label{lemma:dimTori}
	Let $p\geq 5$ be odd, then for each $\lambda\in\Lambda=\{0,2,\dotsc,p-3\}$, the dimension of $V_{p,1,\lambda}$ is $\frac{p-1-\lambda}{2}$.
\end{lemma}
\begin{proof}
	A basis for $V_{p,1,\lambda}$ is indexed by $p$-admissible colorations of the trivalent graph consisting of one
	loop edge and one root edge colored by $\lambda$. The number of $p$-admissible triples $(i,i,\lambda)$ is then exactly $\frac{p-1-\lambda}{2}$.
\end{proof}

Before proving the rest of the lemmas of this section, we introduce some useful terminology.

\begin{definition}\label{def:permissive}
	Let us fix $p\geq 3$ odd. For $i\in \Lambda$, we will say that a tuple $\ul \in \Lambda^n$ \emph{permits $i$} if $\dim V_{p,0,(i,\ul)}\neq 0$.
	We will say that a tuple $\ul$ is \emph{permissive} if $\ul$ permits $i$ for all $i\in \Lambda$.
	Finally, we say that a color $i\in \Lambda$ is a \emph{permissive color} if $(i,i)$ is permissive,
	that is, $(i,i,j)$ is $p$-admissible for any $j \in \Lambda$.
\end{definition}

The following can be directly checked from the $p$-admissibility conditions:

\begin{lemma}
	\label{lemma:permColor} The color $i_0=\frac{p-1}{2}$ if $p\equiv 1\mod 4$ or $i_0=\frac{p-3}{2}$ if $p\equiv 3\mod 4$ is
	the unique permissive color in $\Lambda$.
\end{lemma}

In the remainder of this section, we will always use $i_0$ to denote the permissive color. The color $i_0$ will play an
important role in the proof of the next lemmas.

\begin{lemma}\label{lemma:dimAtLeastOne}
	Let $p\geq 3$ be odd, $g\geq 1$, $n\geq 0$ and $\ul\in \Lambda^n$.
	Then $\dim V_{p,g,\ul}\geq 1$.
\end{lemma}

\begin{proof}
	Consider the following coloring of a trivalent graph for $\Sigma_{g,n}$:
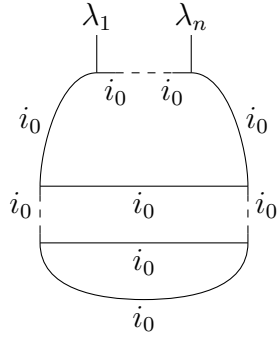
	\begin{center}
		\ctikzfig{graph_genus_at_least_one}
	\end{center}
	Since every vertex is adjacent to at least $2$ edges colored by $i_0$ and $i_0$ is permissive, the coloring is $p$-admissible.
\end{proof}
For $\gamma$ a simple closed curve on the surface $\Sigma_{g,n}$, we will denote by $t_{\gamma}$ the Dehn twist along $\gamma$.
A consequence of Lemma \ref{lemma:dimAtLeastOne} is the following.
\begin{proposition}
	\label{prop:DehnTwistSpectrum} Let $g\geq 2$ and let $\ul\in \Lambda^n$ for some $n\geq 0$, and let $\gamma$
	be a non-separating simple closed curve on $\Sigma_{g,n}$. Then, up to rescaling, we have
	$$\mathrm{Spec}(\rho_{p,g,\ul}(t_{\gamma}))=\lbrace \zeta_p^{k^2} | 1\leq k \leq \frac{p-1}{2}\rbrace.$$
\end{proposition}
\begin{proof}By the results of \cite{bhmv}, cutting $\Sigma_{g,n}$ along $\gamma$, we have a decomposition
	$$V_{p,g,\ul}=\underset{i\in \Lambda}{\bigoplus}V_{p,g-1,(\ul,i,i)}$$
	where the $i$-th term is a $(-\zeta_p)^{i(i+2)}=\zeta_p^{(i+1)^2-1}$-eigenspace of $\rho_{p,g,\ul}(t_{\gamma})$.
	As $i$ runs over $\Lambda$, those eigenvalues run over all elements of $\lbrace \zeta_p^{k^2-1} , k=1,\ldots ,\frac{p-1}{2}\rbrace$.
	Moreover, by Lemma \ref{lemma:dimAtLeastOne}, all subspaces in the decomposition are non-zero.
\end{proof}
A stronger minoration of dimensions will sometimes be useful:
\begin{lemma}\label{lemma:dimAtLeastTwo}
	Let $p\geq 5$ be odd, $g\geq 1$, $n\geq 0$ and $\ul\in(\Lambda\setminus\{0\})^n$.
	Then $\dim V_{p,g,\ul}\geq 2$ unless $g=1$, $n=1$ and $\ul=p-3$.
\end{lemma}
\begin{proof}
	We start with the case $g\geq 2$. Picking a non-separating simple closed curve $\gamma$ in $\Sigma_{g,n}$, by cutting along $\gamma$ we get
	$$\dim V_{p,g,\ul}=\underset{i\in \Lambda}{\sum} \dim V_{p,g-1,(\ul,i,i)}.$$
	However, since $g-1\geq 1$, by Lemma \ref{lemma:dimAtLeastOne}, each term in the sum is positive, so the sum is at least $\frac{p-1}{2}\geq 2$.
	
	In the case $g=1$ and $n\geq 2$, we pick the same trivalent graph as in the proof of Lemma \ref{lemma:dimAtLeastOne}.
	Coloring all internal edges by $i_0$ gives us one $p$-admissible coloring. Now, change one internal coloring
	from $i_0$ to $i_0':=i_0+2$ if $p\equiv 3 \mod 4$, or to $i_0':=i_0-2$ otherwise. One can easily see that all triples
	$(i_0,i_0',j)$ for $j\in \Lambda \setminus \lbrace 0 \rbrace$ are $p$-admissible. Therefore, $\dim V_{p,1,\ul}\geq 2$
	(in fact, $\geq n+1$, since there are $n$ internal edges) if $n\geq 2$.
	
	The last case, $g=1,n=1$ is a consequence of Lemma \ref{lemma:dimTori}.
\end{proof}

Combining \Cref{lemma:dimAtLeastTwo,lemma:dimTori} we get the following.

\begin{corollary}\label{cor:nonTriviality}
	Let $p\geq 5$ be odd and let $g\geq 1$, $n\geq 0$ and $\ul\in \Lambda^n$, then $V_{p,g,\ul}\neq 0$,
	and $\dim V_{p,g,\ul}=1$ if and only if $g=1$ and $\ul=(p-3,0,\dotsc,0)$, up to permutation.
\end{corollary}
	
\subsection{Structure of permissive tuples}
\label{sec:permissive}
The results of this section are used for our main results only in genus $1$ and $2$ but they cast some light on the notion of permissive tuple,
which was introduced in Definition \ref{def:permissive}, and also on the scope of \Cref{cor:tensor-indec,thm:density} (in genus $1$ and $2$).

In this section, we will denote by $(k)^n$ the tuple $(k,\ldots,k)$ (with $n$ repetitions of $k$) and by $t\cup t'$
the concatenation of tuples $t$ and $t'$.
Our goal is to prove the following:

\begin{proposition}
	\label{prop:permissiveTuple} For any $p\geq 3$ odd, and any $\ul\in \left( \Lambda \setminus \lbrace 0 \rbrace \right)^n$
	with $n\geq p-3$, the tuple $\ul$ is permissive.
	Moreover, $(p-3)^{p-4}$ is not permissive.
\end{proposition}

The proof of Proposition \ref{prop:permissiveTuple} will be broken into several lemmas.

\begin{lemma}
	\label{lemma:permissive1} Let $\ul \in \Lambda^n$ be permissive, then for any $i\in\Lambda$, the tuple $(\ul,i)$ is permissive.
\end{lemma}
\begin{proof}
	For $j\in \Lambda$, we have $$V_{0,(\ul,i,j)}=\underset{k \in \Lambda}{\bigoplus}V_{0,(\ul,k)}\otimes V_{0,(i,j,k)}.$$
	However, for any $i,j \in \Lambda$, there exists $k\in \Lambda$ such that $(i,j,k)$ is $p$-admissible (namely $k=\min(i+j,2p-4-i-j)$
	works since $2p-4-i-j\geq |i-j|$), so, since $\ul$ is permissive, we get that $V_{0,(\ul,i,j)}\neq 0$ for all $j\in \Lambda$ and $(\ul,i)$
	is permissive.
\end{proof}
\begin{lemma}
	\label{lemma:permissive2} Let $\ul=(\lambda_1,\ldots,\lambda_n) \in \Lambda^n$ be such that $\lambda_k\leq i_0$ for $k=1,\ldots,n$.
	Then $\ul$ permits $i_0$ if and only if $\underset{1\leq k \leq n }{\sum}\lambda_k\geq i_0$.
	Moreover, if $p\equiv 1\mod 4$, then $\ul$ permits $i_0-2$ if and only if $\underset{1\leq k \leq n }{\sum}\lambda_k\geq i_0-2$.
\end{lemma}
\begin{proof}
	A quick application of the triangular inequalities among the $p$-admissibility conditions shows that if
	$V_{0,(\ul,i)}\neq 0$ then $i\leq \underset{1\leq k \leq n }{\sum}\lambda_k$.
	This proves that the condition is necessary. To show that it is sufficient, let $m$ be the maximum integer such that
	$\lambda_1+\dotsb+\lambda_m<i_0$. Let us pick a pair of pants decomposition of $\Sigma_{0,n+1}$ whose associated trivalent graph is a comb,
	with teeth colored by $\lambda_1,\ldots,\lambda_n$. We color the edges of the spine of the comb by $x_2,\ldots,x_n$, where for $k\leq m$,
	$$x_k=\lambda_1+\dotsb+\lambda_k,$$
	and for $k>m$, $x_k=i_0$. This coloring is $p$-admissible.
	
	For the last claim, we color the edges of the spine in the same way as above; except for one, which we color by $i_0-2$ instead:
	if $\lambda_n\neq 0$, this is the last one; otherwise, we choose the one adjacent to $\lambda_k$ and $\lambda_{k-1}$ for $k$ such that
	$\lambda_s=0$ for $s>k$.
	Since $(i_0,i_0-2,j)$ is $p$-admissible for any $j\geq 2$, this is still a $p$-admissible coloring.
\end{proof}

\begin{proof}[Proof of Proposition \ref{prop:permissiveTuple}]
	Let $\ul=(\lambda_1,\ldots,\lambda_n) \in \left(\Lambda\setminus \lbrace 0 \rbrace \right)^n$ with $n\geq p-3$.
	
	First, assume that there are at least $2$ colors among the $\lambda_i$'s which are $\geq i_0$. Without loss of generality,
	we can assume that those are $\lambda_{n-1}$ and $\lambda_n$. Note that if $k=\max(|\lambda_{n-1}-\lambda_n|,2)$,
	then $(\lambda_{n-1},\lambda_n,k)$ is $p$-admissible, and moreover $k\leq i_0$. Since for any $j\in \Lambda$,
	we have that $V_{0,(\lambda_{n-1},\lambda_n,k)}\otimes V_{0,(\lambda_1,\ldots,\lambda_{n-2},k,j)}$ is a summand of $V_{0,(\ul,k)}$,
	it suffices to prove that $(\lambda_1,\ldots,\lambda_{n-2},k)$ is permissive. By repeating this process, we can assume that $\ul$
	contains at most $1$ color $>i_0$ and at least $\frac{p-3}{2}$ colors in $\lbrace 2,4,\ldots,i_0\rbrace$.
	Because of Lemma \ref{lemma:permissive1}, it suffices to treat the case where $\ul$ consists of exactly $\frac{p-3}{2}$ colors in
	$\lbrace 2,4,\ldots,i_0\rbrace$.
	
	If $p\equiv 3 \mod 4$, we partition $\ul$ into two tuples $\ul_1$ and $\ul_2$ of size $\frac{p-3}{4}$. By fusion rules, for any $j\in \Lambda$,
	$$V_{0,(\ul,j)}=V_{0,(\ul_1,\ul_2,j)}=\underset{i_1,i_2\in\Lambda}{\bigoplus}V_{0,(\ul_1,i_1)}\otimes V_{0,(i_1,i_2,j)}\otimes V_{0,(\ul_2,i_2)}.$$
	Therefore, $V_{0,(\ul_1,i_0)}\otimes V_{0,(i_0,i_0,j)}\otimes V_{0,(\ul_2,i_0)}$ is a summand of $V_{0,(\ul,j)}$.
	By Lemmas \ref{lemma:permissive2} and \ref{lemma:permColor}, this summand is non-zero for any $j\in \Lambda$, hence $\ul$ is permissive.
	
	If $p\equiv 1 \mod 4$, instead we partition $\ul$ into $\ul_1$ of size $\frac{p-1}{4}$ and $\ul_2$ of size $\frac{p-5}{4}$.
	We get from Lemma \ref{lemma:permissive2} that $V_{0,(\ul_1,i_0)}\neq 0, V_{0,(\ul_1,i_0-2)}\neq 0$, and that $V_{0,(\ul_2,i_0-2)}\neq 0$.
	However, as $(i_0,i_0-2,j)$ for any $j\in \lbrace 2,\ldots, i_0 \rbrace$ is $p$-admissible, and as $(i_0-2,i_0-2,0)$ is $p$-admissible,
	we conclude that $\ul$ is permissive.
	
	Finally, we prove that $(p-3)^{p-4}$ is not permissive. Note that for $j\in \Lambda$, we have that $(p-3,p-3,j)$ is $p$-admissible
	if and only if $j=0$ or $2$. Thus if $(p-3)^{p-4}$ were permissive,
	$(p-3)\cup (2)^n$ would permit $0$ for some $n\leq \frac{p-5}{2}$.
	But $V_{0,(0,p-3)\cup (2)^n}=V_{0,(p-3)\cup (2)^n}=0$ since $2n\leq 2\cdot \frac{p-5}{2}<p-3$.
\end{proof}

To conclude this section, we prove the following lemma that will be used
in the proof of density in genus $1$.

\begin{lemma}\label{lemma:dimensionGenusZero}
	Let $p\geq 5$ be a prime, $n\geq 0$ and $\ul\in\Lambda^n$ which can be partitioned
	into $2$ permissive tuples $\ul^1$ and $\ul^2$.
	Then for any $\mu\in \Lambda$, $\dim V_{0,(\ul,\mu)}\geq 2$.
\end{lemma}
	
\begin{proof}
	Cut $\Sigma_{0,n+1}$ along $2$ disjoint simple closed curves,
	the first isolating $\ul^1$, and the second isolating $\ul^2$; the last piece in the decomposition being a pair of pants
	containing the boundary component colored by $\mu$.
	The corresponding fusion decomposition is
	\[ V_{0,(\ul,\mu)}\simeq \bigoplus_{\mu_1,\mu_2\in\Lambda} V_{0,(\mu,\mu_1,\mu_2)}\otimes
						V_{0,(\ul^1,\mu_1)}\otimes V_{0,(\ul^2,\mu_2)}. \]
	As $\ul^1$ and $\ul^2$ are permissive, $V_{0,(\ul^1,\mu_1)}$ and $V_{0,(\ul^2,\mu_2)}$ are never zero, so the summand corresponding to $(\mu_1,\mu_2)$
	is non-zero if and only if $(\mu,\mu_1,\mu_2)$ is admissible.
	For any $\mu$, there exist at least $2$ pairs $(\mu_1,\mu_2)$ such that
	$(\mu,\mu_1,\mu_2)$ is admissible. Indeed, for $\mu\neq 0$, $(\mu,\mu,0)$ and $(\mu,0,\mu)$ are admissible,
	and for $\mu=0$, $(\mu,0,0)$ and $(\mu,2,2)$ are admissible. Hence, the above decomposition has at least $2$ non-zero summands.
	So, $\dim V_{0,(\ul,\mu)}\geq 2$.
\end{proof}
	
\section{Irreducibility and tensor-indecomposability for \texorpdfstring{$\SO(3)$}{SO(3)}-WRT representations at prime level}
\label{sec:tensorIndec}
The main result of this section will be Theorem \ref{thm:tensor-indec}, which shows that quantum representations $\rho_{p,g,\ul}$
(or their quotients $\rho_{J,g,\ul}$) for $p$ prime are tensor-indecomposable for most choices of surfaces $\Sigma_{g,n}$ with $g\geq 1$
and of coloring data $\ul$. We will make heavy use of the irreducibility of quantum representations.
It was first proved by Roberts \cite{R01} that, for closed surfaces, the quantum representations $\rho_{p,g}$ where $p$ is prime are irreducible.
This was later extended in \cite{KS16a} and \cite[Appendix A]{God25} to the representations $\rho_{p,g,\ul}$ of $\PMod(\Sigma_{g,n})$
where $p$ is prime and $n>0$. We will also need a finite characteristic version of these results. Recall that a representation is said to be
\emph{absolutely irreducible} if it is irreducible over the algebraic closure of its field of definition.

\begin{theorem}
	\label{thm:irreducibility} Let $p\geq 5$ be a prime number, $J$ be a maximal ideal in $\Z[\zeta_p]$ other than $(\zeta_p-1)$, $g\geq 0$
	and $\ul\in \Lambda^n$ for some $n\geq 0$.
	Then $\rho_{p,g,\ul}$ and $\rho_{J,g,\ul}$ are absolutely irreducible representations of $\PMod(\Sigma_{g,n})$.
\end{theorem}
The proof in finite characteristic is a minor modification of the arguments in \cite{R01} and \cite{KS16a}.

\begin{proof}
For $\gamma$ a simple closed curve on $\Sigma_{g}^{n}$, we denote by $Z_p(\gamma)$ the curve operator associated to $\gamma$.
Recall that it is the operator associated by the TQFT to the cobordism $\hat{\Sigma}_{g}^{n} \times [0,1]$
equipped with the colored banded tangle $\gamma \times[\epsilon,1-\epsilon] \sqcup \ul \times [0,1]$ for $\epsilon = 1/4$.

Recall that both $Z_p(\gamma)$ and $\rho_{p,g,\ul}(t_{\gamma})$ are diagonal on a basis given by the
$p$-admissible colorings of a certain uni-trivalent graph with respective eigenvalues
$-(\zeta_p^{2j+2}+\zeta_p^{-2j-2})$ and $\zeta_p^{j(j+2)}$ for $j \in \Lambda$.
Note that as $p$ is prime, the $\zeta_p^{j(j+2)}$ are pairwise distinct for $j \in \Lambda$.
Now, if $J$ is a maximal ideal distinct from $(\zeta_p-1)$, we still have that $\zeta_p^{j(j+2)} \neq \zeta_p^{l(l+2)} \mod J$
for $j \neq l$ in $\Lambda$.

Therefore, there exists a polynomial $P$ with coefficients in $\Z[\zeta_p]/J$ such that $$P(\rho_{J,g,\ul}(t_{\gamma})) = Z_p(\gamma)$$ modulo $J$.
Now, applying \cite[Proposition 1.6]{bhmv} and \cite[Proposition 3.2]{KS16a} we have that the $\Z[\zeta_p]/J$-algebra generated by
$\{Z_p(\alpha) \, | \, \alpha \, \, \, \text{simple closed curve} \}$ is $\mathcal{M}_{d_{p,g,\ul}}( \Z[\zeta_p]/J)$,
hence the $\Z[\zeta_p]/J$-algebragenerated by $\{\rho_{J,g,\ul}(\phi) \, | \, \phi \in \Mod(\Sigma_{g}^{n}) \}$
is $\mathcal{M}_{d_{p,g,\ul}}( \Z[\zeta_p]/J)$. We conclude that $\rho_{J,g,\ul}$ is irreducible.
\end{proof}
We now state the main result of this section:
\begin{theorem}\label{thm:tensor-indec}
	Let $p\geq 5$ be a prime number, $J$ be a maximal ideal in $\Z[\zeta_p]$ other than $(\zeta_p-1)$,
	$g\geq 1$ and $\ul\in \Lambda^n$ for some $n\geq 0$. Consider an embedding
	of $\Sigma_0^3$ in $\Sigma_{g,n}$ and denote by $\Sigma'$ the complement of the interior of $\Sigma_0^3$ in $\Sigma_{g,n}$.
	Consider the corresponding fusion decomposition of $V_{g,\ul}$:
	\begin{equation*}
		V_{g,\ul}=\bigoplus_{\mu_1,\mu_2,\mu_3\in\Lambda}V(\Sigma',\ul,\mu_1,\mu_2,\mu_3)\otimes V(\Sigma_0^3,\mu_1,\mu_2,\mu_3).
	\end{equation*}
	Assume that there exists an embedding $\Sigma_0^3\subset \Sigma_{g,n}$ such that both of the following are satisfied:
	\begin{enumerate}
		\item One of the boundary components of $\Sigma_0^3$ is mapped to a non-separating {\scc} in $\Sigma_{g,n}$,
		\item For any $\mu_1,\mu_2,\mu_3\in\Lambda$, the triplet $(\mu_1,\mu_2,\mu_3)$ is admissible if and only if the corresponding summand
		in the fusion decomposition is non-zero.
	\end{enumerate}
	Then $\rho_{g,\ul}$ and $\rho_{J,g,\ul}$ are \emph{tensor-indecomposable}
	as representations of $\PMod(\Sigma_{g,n})$ over the algebraic closure of their field of definition.
\end{theorem}

\begin{proof}
	In this proof, we will denote $V_{p,g,\ul}$ or $V_{J,g,\ul}$ simply by $V$,
	and $\rho_{p,g,\ul}$ or $\rho_{J,g,\ul}$ simply by $\rho$.
	The algebraic closure of the field of definition will be denoted by $k$. It is $\C$ for $V=V_{p,\ul}$
	and $\overline{\F}_q$ for $V=V_{J,\ul}$,
	where $\F_q=\Z[\zeta_p]/J$.
	Choose a tensor decomposition $V = W_1 \otimes W_2$, $\rho=\eta_1\otimes\eta_2$ over $k$.
	We shall prove that $W_1$ or $W_2$ has dimension $1$ by inspecting the eigenvalues
	of the action of Dehn twists on $W_1$ and $W_2$ and how they interact with fusion decompositions.

	Choose a {\scc} $\alpha\subset \Sigma_{g,n}$ and consider the associated left Dehn twist $T_\alpha\in\PMod(\Sigma_{g,n})$.
	For $i=1,2$, choose a lift $\tilde{\eta}_i(T_\alpha)\in\GL(W_i)$ of $\eta_i(T_\alpha)\in\PGL(W_i)$,
	and set $\tilde{\rho}(T_\alpha)=\tilde{\eta}_1(T_\alpha)\otimes\tilde{\eta}_2(T_\alpha)$.
	Denote by $S_i^\alpha\subset k^\times$ the spectrum of $\tilde{\eta}_i(T_\alpha)$. The spectrum of $\tilde{\rho}(T_\alpha)$
	is then the product $S^\alpha=S^\alpha_1S^\alpha_2$.
	\begin{claim}[1]
		The product map $S^\alpha_1\times S^\alpha_2\ra S^\alpha$ is a bijection.
	\end{claim}
	\begin{proof}
		For $i=1,2$, denote by $W_i=\bigoplus_{s_i\in S_i^\alpha}W_i(s_i)$ the eigenspace decomposition of $\tilde{\eta}_i(T_\alpha)$.
		Note that $\tilde{\eta}_i(T_\alpha)$ is semisimple because $\tilde{\rho}(T_\alpha)$ is.
		Then, with $V=\bigoplus_{s\in S^\alpha}V(s)$ the eigenspace decomposition of $\tilde{\rho}(T_\alpha)$,
		for any $s$ in $S^\alpha$ we have
		\begin{equation}\label{eq:eigenspaceoftwist}
			V(s)=\bigoplus_{\substack{(s_1,s_2)\in S^\alpha_1\times S^\alpha_2 \\ s_1s_2=s}} W_1(s_1)\otimes W_2(s_2).
		\end{equation}
		Hence, to prove Claim (1), we need only show that the above decomposition has only $1$ summand.
		As $p$ is prime, we have an identification $f^\alpha:S^\alpha\simeq\Lambda$ which matches
		the eigenspace decomposition $V=\bigoplus_{s\in S^\alpha}V(s)$
		with the fusion decomposition $V=\bigoplus_{\mu\in\Lambda}V(\Sigma^\alpha,\ul,\mu,\mu)$ along $\alpha$.
		Here $\Sigma^\alpha\subset \Sigma_{g,n}$ is the complement to a tubular neighborhood of $\alpha$, and is possibly disconnected.
		As $p$ is prime, each $V(\Sigma^\alpha,\ul,\mu,\mu)$ is an irreducible $\PMod(\Sigma^\alpha)$-module. Now, as $T_\alpha$
		commutes with the action of $\PMod(\Sigma^\alpha)$,
		the decomposition in \Cref{eq:eigenspaceoftwist} is a $\PMod(\Sigma^\alpha)$-module decomposition.
		Hence, it must be trivial and Claim (1) follows.
\end{proof}

	We will use the notations $S^\alpha_1\times S^\alpha_2$ and $S^\alpha_1S^\alpha_2$ interchangeably.
	Choose $2$ other {\scc s} $\beta$ and $\gamma$ with $\alpha$, $\beta$ and $\gamma$ disjoint,
	and such that there exists an embedding $i:\Sigma_0^3\hookrightarrow \Sigma_{g,n}$
	of a pair of pants with $i(\partial \Sigma_0^3)=\alpha\cup\beta\cup\gamma$.
	Denote by $\Sigma'$ the complement of $i(\Sigma_0^3)$ in $\Sigma_{g,n}$.
	\emph{From now on, we moreover assume that $i:\Sigma_0^3\hookrightarrow \Sigma_{g,n}$
	satisfies conditions \textit{(1)} and \textit{(2)} in the statement of the theorem.}

	For $i=1,2$, we choose lifts $\tilde{\eta}_i(T_\beta)$ and $\tilde{\eta}_i(T_\gamma)$.
	As we did for $\alpha$, we define $\tilde{\rho}(T_\beta):=\tilde{\eta}_1(T_\beta)\otimes \tilde{\eta}_2(T_\beta)$
	and $\tilde{\rho}(T_\gamma):=\tilde{\eta}_1(T_\beta)\otimes\tilde{\eta}_2(T_\beta)$.
	Claim (1) equally applies to $\beta$ and $\gamma$, and we will use notations $S^\beta,S^\beta_1,\dotsc$ for the corresponding spectra,
	and $f^\beta:S^\beta\simeq\Lambda$, $f^\gamma:S^\gamma\simeq\Lambda$ for the bijections.
	The fusion decomposition of $V$ obtained by cutting along $\alpha$, $\beta$ and $\gamma$
	is identified via $f^\alpha\times f^\beta\times f^\gamma:S^\alpha\times S^\beta\times S^\gamma\simeq \Lambda^3$ with the eigenspace decomposition
	\[
		V=\bigoplus_{s^\alpha\in S^\alpha,s^\beta\in S^\beta ,s^\gamma\in S^\gamma}V(s^\alpha,s^\beta,s^\gamma).
	\]
	According to the assumption on $\Sigma_0^3\subset\Sigma_{g,n}$, $V(s^\alpha,s^\beta,s^\gamma)\neq 0$
	exactly when $(f^\alpha(s^\alpha),f^\beta(s^\beta),f^\gamma(s^\gamma))$ is admissible.
	However, as in \Cref{eq:eigenspaceoftwist}, we have
	\begin{equation*}
		V(s^\alpha,s^\beta,s^\gamma)=W_1(s_1^\alpha,s_1^\beta,s_1^\gamma)\otimes W_2(s_2^\alpha,s_2^\beta,s_2^\gamma)
	\end{equation*}
	where $s_i^\alpha\in S_i^\alpha$, $s_i^\beta\in S_i^\beta$, $s_i^\gamma\in S_i^\gamma$ for $i=1,2$ are the unique elements such that
	$s^\alpha=s_1^\alpha s_2^\alpha$, $s^\beta=s_1^\beta s_2^\beta$ and $s^\gamma=s_1^\gamma s_2^\gamma$.
	Hence, $(f^\alpha(s^\alpha),f^\beta(s^\beta),f^\gamma(s^\gamma))$
	is admissible if and only if $W_1(s_1^\alpha,s_1^\beta,s_1^\gamma)\neq 0$ and $W_2(s_2^\alpha,s_2^\beta,s_2^\gamma)\neq 0$.

	\begin{claim}[2]
		The bijection $(f^\beta)^{-1}\circ f^\alpha:S_1^\alpha\times S_2^\alpha\simeq S_1^\beta\times S_2^\beta$
		is of the form $f^{\beta\alpha}_1\times f^{\beta\alpha}_2$
		with $f^{\beta\alpha}_i:S_i^\alpha\simeq S_i^\beta$. Similarly for $(f^\alpha)^{-1}\circ f^\gamma$ and $(f^\gamma)^{-1}\circ f^\beta$.
	\end{claim}
	\begin{proof}
		Define $0_i^\gamma\in S_i^\gamma$ by $f^\gamma(0_1^\gamma0_2^\gamma)=0\in \Lambda$, and similarly for $\alpha$ and $\beta$.
		Consider $s^\alpha=s_1^\alpha s_2^\alpha$ in $S^\alpha$
		and $s^\beta=s_1^\beta s_2^\beta$ in $S^\beta$. Then $f^\alpha(s^\alpha)=f^\beta(s^\beta)$
		if and only if $(f^\alpha(s^\alpha),f^\beta(s^\beta),0)$
		is admissible. So, by the above derivation: $(*)$ $f^\alpha(s^\alpha)=f^\beta(s^\beta)$ exactly when
		$W_1(s_1^\alpha,s_1^\beta,0_1^\gamma)\neq 0$ and $W_2(s_2^\alpha,s_2^\beta,0_2^\gamma)\neq 0$.
		Note that as $(0,0,0)$ is admissible, $W_1(0_1^\alpha,0_1^\beta,0_1^\gamma)\neq 0$ and $W_2(0_2^\alpha,0_2^\beta,0_2^\gamma)\neq 0$.
		Hence, applying $(*)$ twice, $f^\alpha(s^\alpha)=f^\beta(s^\beta)$ implies $f^\alpha(s^\alpha_10^\alpha_2)=f^\beta(s^\beta_1 0^\beta_2)$.
		In fact, we even have: $(**)$ $f^\alpha(s^\alpha)=f^\beta(s^\beta)$ if and only if $f^\alpha(s^\alpha_10^\alpha_2)=f^\beta(s^\beta_1 0^\beta_2)$
		and $f^\alpha(0^\alpha_1s^\alpha_2)=f^\beta(0^\beta_1 s^\beta_2)$.
		In particular, $(f^\beta)^{-1}\circ f^\alpha$ induces bijections $S_1^\alpha\times \{0_2^\alpha\}\simeq S_1^\beta\times \{0_2^\beta\}$
		and $\{0_1^\alpha\}\times S_2^\alpha\simeq \{0_1^\beta\}\times S_2^\beta$.
		Define $f^{\beta\alpha}_1:S_1^\alpha\simeq S_1^\beta$ and $f^{\beta\alpha}_2:S_2^\alpha\simeq S_2^\beta$ to be these bijections.
		Then $(**)$ is exactly the statement that $f^{\beta\alpha}_1\times f^{\beta\alpha}_2=(f^\beta)^{-1}\circ f^\alpha$, as desired.
	\end{proof}

	Thanks to Claim (2), we have canonical identifications $S_1^\alpha\simeq S_1^\beta\simeq S_1^\gamma$
	and $S_2^\alpha\simeq S_2^\beta\simeq S_2^\gamma$.
	We shall denote the first set by $S_1$, the second by $S_2$ and the canonical bijection $S_1\times S_2\simeq \Lambda$ by $f$.
	For each $i=1,2$, the elements $0_i^\alpha$, $0_i^\beta$ and $0_i^\gamma$ are identified and will be denoted $0_i\in S_i$.

	We call a subset $X$ of $\Lambda$ \emph{saturated} if for any $x_1,x_2\in X$ and $\mu\in\Lambda$, $(x_1,x_2,\mu)$ admissible implies $\mu\in X$.
	\begin{claim}[3]
		The subsets $f(0_1,S_2)$ and $f(S_1,0_2)$ of $\Lambda$ are saturated.
	\end{claim}
	\begin{proof}
		The claim is symmetric in $S_1$ and $S_2$, so we need only check it for $S_1$.
		Choose $t,t'\in S_1$ and $(s_1,s_2)\in S_1\times S_2$ such that $(f(t,0_2),f(t',0_2),f(s_1,s_2))$ is admissible.
		To prove the claim, we need to deduce that $s_2=0_2$. By the discussion above Claim (2), the admissibility implies that
		$W_2(0_2,0_2,s_2)\neq 0$. Together with $W_1(0_1,0_1,0_1)\neq 0$, this in turn implies that
		$(f(0_1,0_2),f(0_1,0_2),f(0_1,s_2))=(0,0,f(0_1,s_2))$ is admissible.
		Hence, $f(0_1,s_2)=0$ and $s_2=0_2$.
	\end{proof}

	\begin{claim}[4]
		The only non-empty saturated subsets of $\Lambda$ are $\{0\}$ and $\Lambda$.
	\end{claim}
	\begin{proof}
		Denote by $i_0\in\Lambda$ the even integer equal to $\frac{p-1}{2}$ if $p\equiv 1\pmod 4$ or to $\frac{p-3}{2}$ if $p\equiv 3\pmod 4$.
		Assume that $X\subset \Lambda$ is saturated and different from $\varnothing$ and $\{0\}$.
		Then it contains some $0<i\leq p-3$. If $i>i_0$, as $i+i+2\leq 2p-4$, $(i,i,2)$ is admissible and so $2\in X$.
		Hence, we may assume that $i\leq i_0$.
		If $i<i_0$, then $i+i+(i+2)\leq 3\frac {p-3} 2+2=\frac 3 2 p-\frac 5 2\leq 2p-4$ as $p\geq 3$. So, $(i,i,i+2)$ is admissible and hence $i+2\in X$.
		Hence, we may assume that $i_0\in X$.
		Now, for any $j\in \Lambda$, $(i_0,i_0,j)$ is admissible. Hence, $X=\Lambda$.
	\end{proof}

	By Claims (3) and (4), either $f(S_1,0_2)=\{0\}$ or $f(S_1,0_2)=\Lambda$.
	In the first case, $|S_1|=1$, so the Dehn twists $T_\alpha$, $T_\beta$, $T_\gamma$ act trivially on $W_1$.
	Now, one of $\alpha$, $\beta$, $\gamma$ is non-separating by assumption \textit{(1)}.
	If $g\geq 1$, the group $\PMod(\Sigma_{g,n})$ is normally generated by any Dehn twist along a
	non-separating simple closed curve \cite[§4.4.4]{farbPrimerMappingClass2011}.
	So, in all cases, $\eta_1$ is trivial. As $\rho$ is irreducible,
	$W_1$ must have dimension $1$. In the second case, because $|S_1|\cdot|S_2|=|\Lambda|$,
	we must have $|S_2|=1$ and thus $W_2$ has dimension $1$. This concludes the proof of \Cref{thm:tensor-indec}.
\end{proof}

\begin{corollary}\label{cor:tensor-indec}
	Let $g\geq 1$, $n\geq 0$ and $\ul\in\Lambda^n$. Assume that $g$, $n$ and $\ul$ satisfy one of the following conditions:
	\begin{enumerate}
		\item $g\geq 3$;
		\item $g=2$ and $V_{0,\ul}\neq 0$;
		\item $g=1$ and $\ul$ can be partitioned into $2$ permissive tuples $\ul^1$ and $\ul^2$.
	\end{enumerate}
	Then $\rho_{g,\ul}$ and $\rho_{J,g,\ul}$ are \emph{tensor-indecomposable}
	as representations of $\PMod(\Sigma_{g,n})$ over the algebraic closure of their field of definition.
\end{corollary}

Note that for fixed $p$, the assumptions on $\ul$ only exclude a finite number of genus $1$ and genus $2$ cases; see \Cref{prop:permissiveTuple}.
Using \Cref{lemma:permColor,lemma:permissive1}, one can get examples satisfying the conditions with a reasonable number of punctures.
In fact,
for any $n\geq 0$ and $\ul\in\Lambda^n$, $\rho_{1,(\ul,i_0,i_0,i_0,i_0)}$ and $\rho_{2,(\ul,i_0,i_0)}$ are tensor-indecomposable.

\begin{proof}
	Let us use the notation of \Cref{thm:tensor-indec}. We need to show that, for each of the cases \textit{(1)}-\textit{(3)},
	there is a choice of $\Sigma_0^3\subset\Sigma_{g,n}$
	with at least $1$ of the boundary components of $\Sigma_0^3$ non-separating in $\Sigma_{g,n}$,
	such that the fusion decomposition summand
	\begin{equation*}
		V(\Sigma',\ul,\mu_1,\mu_2,\mu_3)\otimes V(\Sigma_{0,3},\mu_1,\mu_2,\mu_3)
	\end{equation*}
	is non-zero if and only if $(\mu_1,\mu_2,\mu_3)$ is admissible.
	If the summand is non-zero, then $V(\Sigma_{0,3},\mu_1,\mu_2,\mu_3)\neq 0$ and so $(\mu_1,\mu_2,\mu_3)$ is admissible by definition.
	For the converse, we note that any choice of $\Sigma_0^3\subset \Sigma_{g,n}$ as above can be extended to a pair of pants decomposition.
	This pair of pants decomposition can be represented by a dual graph and admissible colorings of this dual graph index a basis of $V$.
	Hence, to prove the converse, it is enough to exhibit a trivalent graph and for each choice of
	admissible $(\mu_\alpha,\mu_\beta,\mu_\gamma)\in\Lambda^3$, an extension
	of this triplet to an admissible coloring of the graph. Let us consider the three cases.
	\begin{enumerate}
		\item For $g\geq 3$, the following colored graph is admissible when $(\mu_1,\mu_2,\mu_3)$ is, where the color $i_0$
		is defined in Lemma \ref{lemma:permColor}.
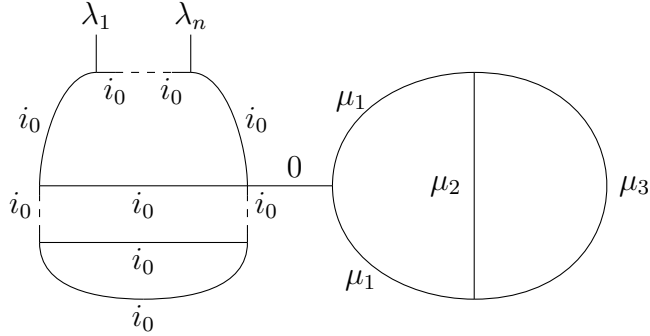
			\begin{center}
				\ctikzfig{graph_genus_at_least_three}
			\end{center}
		\item For $g=2$, the colored graph below works, where the existence of $a_1,\dotsc,a_{n-3}$ is enabled by the assumption that $V_{0,\ul}\neq 0$.
			\begin{center}
				\ctikzfig{graph_genus_two}
			\end{center}
		\item For $g=1$ and $\ul$ partitioned into $2$ permissive tuples $\ul^1$ and $\ul^2$ of respective lengths $n_1$ and $n_2$,
		consider an embedding $\Sigma_0^3\subset \Sigma_{1,n}$
		such that the complement $\Sigma'$ has $2$ components: a sphere $\Sigma_{0,n_1}^2$ marked by $\ul^1$ and a sphere $\Sigma_{0,n_2}^1$ marked
		by $\ul^2$. Then $V(\Sigma',\ul,\mu_1,\mu_2,\mu_3)$ is the tensor product of $V_{0,(\ul^1,\mu_1,\mu_2)}$ and $V_{0,(\ul^2,\mu_3)}$.
		These are non-zero for any $\mu_1,\mu_2,\mu_3\in \Lambda$: $V_{0,(\ul^2,\mu_3)}\neq 0$ as $\ul^2$ is permissive and
		$V_{0,(\ul^1,\mu_1,\mu_2)} \neq 0$ as $(\ul^1,\mu_1)$ is permissive by \Cref{lemma:permissive1}.
	\end{enumerate}
\end{proof}

\section{Simplicity of the images of quantum representations}
\label{sec:simplicity}
In this section, we adapt arguments of Larsen and Wang \cite{LW05}
to prove that for any prime $p$, any surface $\Sigma_{g,n}$ with $g\geq 1$ and any coloring $\ul$ of the boundary components
such that $\rho_{p,g,\ul}$ is irreducible and tensor-indecomposable, the closure of $\mathrm{Im}\rho_{p,g,\ul}$
is a simple centerless Lie subgroup of $\PSU_{d_{p,g,\ul}}$ (see \Cref{thm:simpleGroup}).

By the same method, we also prove that for any maximal ideal $J$ of norm $q$ coprime to $p$, the image of $\rho_{J,g,\ul}$
is a simple subgroup of $\PGL_{d_{p,g,\ul}}(\F_q)$.

\Cref{thm:simpleGroup} will serve as in input in the proofs of \Cref{thm:density,thm:surjectivity}.

\subsection{Preliminary linear algebra lemmas}
\label{sec:prelimLinAlg}

We begin by recalling an important lemma from linear algebra that appeared in \cite[Lemma 1.2]{FLW},
which we extend to the case of finite characteristic:

\begin{lemma}\label{lemma:spectrumAdditiveDecomp}
	Let $V$ be a non-zero vector space over a field $\F$ that admits a decomposition:
	$$V=\underset{i=1}{\overset{n}{\bigoplus}}W_i,$$
	and let $t\in \GL(V)$ be semisimple and such that for some non-trivial permutation $\sigma \in S_n$, $t(W_i)\subset W_{\sigma(i)}$
	for all $1\leq i \leq n$. Then the spectrum of $t$ in the algebraic closure $\overline{\F}$ of $\F$ contains a coset of
	a non-trivial group of roots of unity.
\end{lemma}
\begin{proof}
	Without loss of generality, we can assume that $\sigma$ is an $n$-cycle for some $n>1$. Let $P\in \F[X]$ be the minimal polynomial of $t$.
	We claim that $P(X)=Q(X^n)$ for some polynomial $Q\in \F[X]$. Indeed, if
	$$P(X)=a_d X^d +\dotsb +a_0$$
	then the block diagonal part of $P(t)$ is the sum of the $a_i t^i$ such that $n|i$, and since the block diagonal part of $P(t)$ must
	vanish we get the claim by setting
	$$Q(X)=\underset{i\geq 0}{\sum} a_{ni}X^i.$$
	
	Now, if $n=pd$ where $p= \mathrm{char}(\F)$, then $P'(X)=nX^{n-1}Q'(X^n)=0$, which contradicts the fact that $t$ is semisimple. Hence, $p\nmid n$,
	and the spectrum of $t$, which is the set of roots of $P$, is stable under multiplication by $n$-th roots of unity. Note that the set of $n$-th
	roots of unity in $\overline{\F}$ is non-trivial if $n$ is not a power of $p$.
\end{proof}

We will also need a version of Lemma \ref{lemma:spectrumAdditiveDecomp} that holds for tensor product decompositions of a vector space.
\begin{lemma}\label{lemma:spectrumTensorDecomp}
	Let $V$ be a vector space over a field $\F$ with a decomposition
	$$V=W^{\otimes n}$$
	where $\dim W\geq 2$ and $n\geq 2$. Let $t\in \GL(V)$ be semisimple and such that for any $w_1,\ldots,w_n \in W$, one has:
	$$t(w_1\otimes \dotsb \otimes w_n)=t_1(w_{\sigma(1)})\otimes \dotsb \otimes t_n(w_{\sigma(n)})$$ for some $t_1,\ldots,t_n\in \GL(W)$ and
	for some non-trivial permutation $\sigma \in S_n$. Then the spectrum of $t$ in the algebraic closure $\overline{\F}$ of $\F$ contains a
	coset of a non-trivial group of roots of unity.
\end{lemma}
\begin{proof}
	Again, without loss of generality, assume that $\sigma$ is an $n$-cycle, i.e.
	$$t(w_1\otimes \dotsb \otimes w_n)=t_1(w_2)\otimes \dotsb \otimes t_{n-1}(w_n)\otimes t_n(w_1).$$
	In fact, conjugating $t$ by $\mathrm{id}_W\otimes t_1 \otimes t_1t_2 \otimes \dotsb \otimes t_1\dotsb t_{n-1}$, one may assume
	that $t_1=\dotsb=t_{n-1}=\mathrm{id}_W$ and that $t_n=s\in \GL(W)$.
	
	Note that $t^n=s\otimes \dotsb \otimes s$, hence $s$ is semisimple as well. Since $\dim W\geq 2$, let $e_0,e_1\in W$ be
	non-collinear eigenvectors of $s$, of eigenvalues $\lambda$ and $\mu$. For $1\leq i \leq n$, let
	$$u_i=e_0^{\otimes i-1} \otimes e_1 \otimes e_0^{\otimes n-i}.$$
	It is clear that $t$ leaves the subspace $$U=\mathrm{Span}_{\overline{\F}}(u_1,\ldots,u_n)$$ invariant;
	in fact, one has $tu_i=\lambda u_{i-1}$ if $i>1$, and $tu_1=\mu u_n$.
	
	It follows that the minimal polynomial of $t_{\mid U}$ is $X^n-\lambda^{n-1}\mu$. If $ \mathrm{char}(\F)=p>0$, then $p\nmid n$ since $t$ is semisimple.
	Therefore, we conclude that the spectrum of $t$ contains a coset of the group of $n$-th roots of unity, which is a non-trivial group.
\end{proof}

\begin{proposition}\label{prop:IsoOfTensor}
	Let $N_1,\dotsc,N_k$ be groups and $\F$ be a field. For each $i$, consider $2$ irreducible projective representations $W_i^1$ and $W_i^2$ of $N_i$
	over $\F$.
	If $\bigotimes_i W_i^1$ and $\bigotimes_i W_i^2$ are isomorphic as projective representations of $N_1\times\dotsb\times N_k$,
	then for each $i$, $W_i^1$ and $W_i^2$ are isomorphic as projective representations of $N_i$.
\end{proposition}
\begin{proof}
	For each $i$, choose a central extension $\widetilde{N}_i$ of $N_i$ such that $W_i^1$ and $W_i^2$ lift to linear representations of $\widetilde{N}_i$.
	Then we have an isomorphism $\bigotimes_i W_i^1\simeq\left(\bigotimes_i W_i^2\right)\otimes \chi$ of
	linear representations of $\widetilde{N}_1\times\dotsb\times \widetilde{N}_k$
	for some character $\chi$. Now, $\chi=\chi_1\otimes\dotsb\otimes\chi_k$ with $\chi_i$ a character of $\widetilde{N}_i$.
	Replacing each $W_i^2$ by $W_i^2\otimes\chi_i$,
	we may assume without loss of generality that $\bigotimes_i W_i^1\simeq\bigotimes_i W_i^2$ as linear representations of
	$\widetilde{N}_1\times\dotsb\times \widetilde{N}_k$.
	Then
	\begin{equation*}
		\bigotimes_i \left( W_i^1\otimes (W_i^2)^*\right)^{\widetilde{N}_i}\simeq \left(\left(\bigotimes_i W_i^1\right) \otimes
		\left(\bigotimes_i W_i^2\right)^*\right)^{\widetilde{N}_1\times\dotsb\times \widetilde{N}_k}\neq 0.
	\end{equation*}
	So, for each $i$, $\left( W_i^1\otimes (W_i^2)^*\right)^{\widetilde{N}_i}\neq 0$, i.e. $W_i^1\simeq W_i^2$
	as linear representations of $\widetilde{N}_i$ and thus also
	as projective representations of $N_i$.
\end{proof}

\begin{corollary}\label{cor:OuterAutomTensor}
	Let $\rho:N \longrightarrow \PGL(W)$ be an irreducible representation, where $W$ is a finite-dimensional vector space over some field $\F$.
	Let $W^{\otimes n}$ be the associated irreducible representation of $N^n$, for some $n\geq 2$. Assume that $\rho^{\otimes n}$
	extends to a projective representation $\widetilde{\rho}$ of an extension $G$ of $N^n$. Let $t\in G$ be an element that acts on $N^n$ by
	\begin{equation}\label{eq:OuterAutom}
		t(x_1,\ldots,x_n)t^{-1}=(\varphi_2(x_2),\ldots,\varphi_{n}(x_n),\varphi_1(x_1)),
	\end{equation}
	where $\varphi_1,\ldots,\varphi_n \in \mathrm{Aut}(G)$.
	Then $\widetilde{\rho}(t)$ is of the form
	$$\widetilde{\rho}(t)(w_1 \otimes \dotsb \otimes w_n)=t_2(w_2)\otimes \dotsb \otimes t_{n}(w_n)\otimes t_1(w_1),$$
	for some $t_1,\ldots,t_n\in \PGL(W)$.
\end{corollary}
\begin{proof}
	Let us apply \Cref{prop:IsoOfTensor}. Set $k=n$ and for each $i$, set $N_i=N$, $W_i^1=\rho$ and $W_i^2=\rho\circ\varphi_i$.
	Note that the underlying vector space of each $W_i^1$ and $W_i^2$ is $W$.
	Then $\widetilde{\rho}(t)$ induces an isomorphism $\bigotimes_iW_i^1\simeq\bigotimes_iW_i^2$ of projective representations of $N^n$.
	So, by \Cref{prop:IsoOfTensor}, for each $i$, there exists an isomorphism $t_i:W_{i}^1\simeq W_{i}^2$. Each $t_i$ is an element of $\PGL(W)$,
	and they satisfy the desired equation by definition.
\end{proof}

The following corollary of \Cref{prop:IsoOfTensor} will help us provide tensor decompositions of some projective representations.

\begin{corollary}\label{cor:OuterAutomTensor2}
	Let $N\simeq N_1\times N_2$ be a normal subgroup of a group $G$, and assume that a tensor product
	$\rho_1\otimes \rho_2:N_1\times N_2\longrightarrow \PGL(W_1\otimes W_2)$ of projective absolutely irreducible representations of
	$N_1$ and $N_2$ of dimensions $\geq 2$ over a field $\F$ extends to a projective representation $\rho$ of $G$. If the image of the morphism
	
	$$G\longrightarrow \mathrm{Aut}(N_1\times N_2)$$
	induced by conjugation has image in $\mathrm{Aut}(N_1)\times \mathrm{Aut}(N_2)$, then $\rho$ is tensor-decomposable as well.
\end{corollary}
\begin{proof}
	Let $g\in G$. Let $\varphi_1\in \mathrm{Aut}(N_1)$ and $\varphi_2\in \mathrm{Aut}(N_2)$ be the automorphisms induced by conjugation by $g$.
	Then $\rho(g)$ induces an isomorphism $\rho_1\otimes\rho_2\simeq (\rho_1\circ\varphi_1)\otimes(\rho_2\circ\varphi_2)$ of
	projective representations of $N_1\times N_2$.
	By \Cref{prop:IsoOfTensor}, there are isomorphisms $\rho_1(g):\rho\simeq \rho\circ\varphi_1$ and $\rho_2(g):\rho\simeq \rho\circ\varphi_2$.
	As $\rho_1$ and $\rho_2$ are absolutely irreducible, $\rho_1(g)\in \PGL(W_1)$ and $\rho_2(g)\in \PGL(W_2)$ are unique.
	Hence, $\rho(g)=\rho_1(g)\otimes\rho_2(g)$ in $\PGL(W_1\otimes W_2)$. One sees that $\rho_1$ and $\rho_2$
	thus defined are representations of $G$ on $W_1$ and $W_2$
	extending those of $N_1\times N_2$, and that $\rho=\rho_1\otimes\rho_2$ is a non-trivial tensor decomposition, as desired.
\end{proof}

We end this section by recalling the following proposition, well-known to representation theory experts, about projective
irreducible representations of direct products of groups:

\begin{proposition}
	\label{prop:projIrredProduct} Let $G$ and $H$ be two compact Lie groups, such that at least one is perfect.
	Then any finite-dimensional projective irreducible representation of $G\times H$ over a field $\F$ is of the form $\rho\otimes \rho'$,
	where $\rho$ is a projective irreducible representation of $G$, and $\rho'$ is a projective irreducible representation of $H$.
\end{proposition}

We note that the above proposition will be used later on for $G$ and $H$ that are either finite groups or subgroups of $\PSU_d$ for some $d$.
\begin{proof}
	A projective irreducible representation $r$ of $G\times H$ lifts to a linear irreducible representation $\widetilde{r}$
	of a central extension of $G\times H$. However, since%
\footnote{This follows from the fact that $H^*(G\times H,M)=H^*(G,H^*(H,M))$ for any abelian group $M$ with trivial $G\times H$-action.
	We are working with continuous cohomology.}
	$$H^2(G\times H,\F^*)=H^2(G,\F^*)\oplus H^2(H,\F^*)\oplus \mathrm{Hom}(G^{ab}\otimes H^{ab},\F^*)=H^2(G,\F^*)\oplus H^2(H,\F^*)$$
	when $G$ and $H$ are perfect, any central extension of $G\times H$ is a product $\widetilde{G}\times \widetilde{H}$
	of central extensions of $G$ and $H$. Therefore, $\widetilde{r}$ is the product of irreducible linear representations
	of $\widetilde{G}$ and $\widetilde{H}$, and $r$ is a product of irreducible projective representations of $G$ and $H$.
\end{proof}

\subsection{Analysis of the restrictions of \texorpdfstring{$\bm{\rho_{p,g,\ul}}$}{SO(3) quantum representations}
to normal subgroups}\label{sec:normalSubgroups}

For the remainder of this section, let us fix a prime number $p\geq 5$ and a maximal ideal $J$ in $\Z[\zeta_p]$ of norm $q$ coprime to $p$.

For $g\geq 0$ and $\ul\in \Lambda^n$ for some $n\geq 0$, we denote by $G_{g,\ul}$
the image of the representation $\rho_{p,g,\ul}\subset \mathrm{PSU}_{d_{p,g,\ul}}$
of $\PMod(\Sigma_{g,n})$, and similarly by $G_{J,g,\ul}\subset \mathrm{PGL}_{d_{p,g,\ul}}(\F_q)$
the image of $\rho_{J,g,\ul}$. The TQFT space $V_{p,g,\ul}$
will be viewed as a $\C$-representation of $G_{g,\ul}$ or as a $\F_q$-representation of $G_{J,g,\ul}$,
depending on context.

Recall that a (projective) representation of a group $G$ is called \emph{isotypic} if it is isomorphic to $V^{\oplus n}$
for some irreducible representation $V$ of $G$.
We also recall that in any finite-dimensional representation $W$ of a group $G$, there exists a unique maximal semisimple $G$-subrepresentation $U$:
the sum of all semisimple subrepresentations. If $W\neq 0$, then $U\neq 0$.

The following lemma controls the decomposition of $V_{p,g,\ul}$ when restricted to normal subgroups:
\begin{lemma}\label{lemma:normalIsotypic}
	Let $N$ be a normal subgroup of $G_{g,\ul}$ or of $G_{J,g,\ul}$. Then $V_{p,g,\ul}$
	is isotypic as an $N$-representation.
\end{lemma}
\begin{proof}
	Let us denote $G_{g,\ul}$ or $G_{J,g,\ul}$ simply by $G$.
	First we argue that $V_{p,g,\ul}$ is semisimple as a representation of $N$.
	Let $U\subset V_{p,g,\ul}$ be the maximal semisimple $N$-subrepresentation.
	Then $U$ is preserved by $G$. As $V_{p,g,\ul}$ is irreducible as a $G$-representation by \Cref{thm:irreducibility},
	$U=V_{p,g,\ul}$ as desired.
	Let
	
	$$V_{p,g,\ul}=W_1^{\oplus n_1}\oplus \dotsb \oplus W_k^{\oplus n_k}$$
	be the decomposition of $V_{p,g,\ul}$ as a sum of irreducible $N$-representations, where distinct $W_i$ are non-isomorphic.
	
	We argue that $k=1$. Since $N$ is a normal subgroup, any element $t$ in $G$ is such that
	$$t(W_i^{n_i})\subset W_{\sigma_t(i)}^{n_{\sigma_t(i)}}$$
	for all $i$ and for some permutation $\sigma_t \in S_k$; that is, elements of $G$ permute the isotypic components of $V_{p,g,\ul}$.
	Since $V_{p,g,\ul}$ is an irreducible $G$-representation, and since $\PMod(\Sigma_{g,n})$ is generated by Dehn twists,
	if $k>1$ there must be a Dehn twist $t\in G$ such that $\sigma_t$ is a non-trivial permutation.
	
	However, by Proposition \ref{prop:DehnTwistSpectrum}, $t$ is semisimple and its spectrum (up to rescaling) is included in
	$\lbrace \zeta_p^{i^2} \ | \ 1 \leq i \leq \frac{p-1}{2} \rbrace$, which is not the full set of $p$-th roots
	of unity and hence does not contain any coset of a non-trivial group of roots of unity.
	By Lemma \ref{lemma:spectrumAdditiveDecomp}, this implies that the permutation $\sigma_t$ is trivial, which is a contradiction. Hence, $k=1$.
\end{proof}
In the case where $\rho_{p,g,\ul}$ (or $\rho_{J,g,\ul}$) is tensor-indecomposable,
we get an even better control of the restriction to normal subgroups:
\begin{lemma}
	\label{lemma:normalIrred} Let $g$ and $\ul$ be such that $\rho_{p,g,\ul}$ and $\rho_{J,g,\ul}$
	are tensor-indecomposable. Let $N$ be a non-trivial normal subgroup of $G_{g,\ul}$ or of $G_{J,g,\ul}$.
	Then $V_{p,g,\ul}$ is irreducible as an $N$-representation.
\end{lemma}

\begin{proof}
	Let $G:=G_{g,\ul}$ or $G_{J,g,\ul}$, and let $\F$ be either $\C$ or $\F_q$ accordingly.
	
	By Lemma \ref{lemma:normalIsotypic}, we know that $V_{p,g,\ul}$ is isotypic as an $N$-representation;
	let us write $V_{p,g,\ul}\simeq W \otimes \F^n$ where $W$ is an irreducible $N$-representation and $\F^n$
	is a trivial $N$-representation.
	
	We lift the projective $G$-representation to a linear representation of a central extension $\widetilde{G}$ of $G$;
	denote also by $\widetilde{N}$ the pre-image of $N$. View $\F[\widetilde{N}]$ as a subalgebra of $\mathrm{End}(W\otimes \F^n)$
	(here $\mathrm{End}$ just denotes linear maps); then $\F[\widetilde{N}]=\mathrm{End}(W)\otimes \mathrm{id}_{\F^n}$,
	since $W$ is (absolutely) irreducible as an $N$-representation. Since $\widetilde{N}$ is normal,
	and since $\widetilde{G}$ is a central extension, conjugation by $g\in G$ induces an automorphism of $\F[\widetilde{N}]$.
	Thus, for any $g\in G$ and $f\in \mathrm{End}(W)$, we have
	$$g(f\otimes \mathrm{id}_{\F^n}) g^{-1}=\psi_g(f)\otimes \mathrm{id}_{\F^n},$$
	for some $\psi_g \in \mathrm{Aut}(\mathrm{End}(W))$. By Skolem-Noether, we can write $\psi_g(f)=a(g)fa(g)^{-1}$ for a unique $a(g) \in \PGL(W)$.
	
	Now, we get that $(a(g)^{-1}\otimes \mathrm{id}_{\F^n})g$ commutes with $f \otimes \mathrm{id}_{\F^n}$ for any $f\in \mathrm{End}(W)$.
	Thus we can write
	$$g=a(g)\otimes b(g)$$
	for some $b(g)\in \PGL(\F^n)$.
	
	It is easy to verify that $g\mapsto a(g)$ and $g\mapsto b(g)$ are morphisms. Then, since $\rho_{p,g,\ul}$
	(or $\rho_{J,g,\ul}$) is tensor-indecomposable, one must have $n=1$ or $\dim W=1$.
	The latter is not possible, unless $N$ is the trivial subgroup, as $G$ and $N$ are embedded in $\PGL(V_{p,g,\ul})$.
	Therefore, $n=1$ and $V_{p,g,\ul}$ is irreducible as an $N$-representation.
\end{proof}

Finally, we have the following easy lemma:

\begin{lemma}
	\label{lemma:perfect} Assume that $g\geq 1$ and $p$ is a prime number $\geq 7$. Let $G=G_{g,\ul}$ or
	$G=G_{J,g,\ul}$. Then $G$ is perfect.
\end{lemma}

\begin{proof}
	Assume first that $g\geq 2$.
	It is well known (see for example \cite[§5]{Korkmaz}) that $H_1(\Mod(\Sigma_{g}^{n}),\Z)$ is trivial if $g\geq 3$,
	and is $\simeq \Z/10\Z$ if $g=2$. In the latter case, the abelianization is generated by the image of the Dehn
	twist along any non-separating simple closed curve.
	
	However, the image of such a Dehn twist by $\rho_{p,g,\ul}$ or $\rho_{J,g,\ul}$ has order $p$,
	which is coprime to $10$ since $p\geq 7$. We thus get that $G$ is perfect if $g\geq 2$.
	
	Now, if $g=1$, we use that $\rho_{p,g,\ul}$ factors through the pure mapping class group of the surface with
	genus $1$ and $n$ marked points.
	The latter has abelianization $\Z/12\Z$ (again see \cite[§5]{Korkmaz}), which is generated by the image of the Dehn twist
	along any non-separating simple closed curve.
	Since the latter is sent to an element of order $p$ by $\rho_{p,g,\ul}$,
	we conclude that $G$ is perfect in that case as well.
\end{proof}

We are now ready to prove the main result of this section:

\begin{theorem}
	\label{thm:simpleGroup} Let $p \geq 7$ be a prime. Let $g\geq 1$ and $\ul\in \Lambda^n$ for some $n\geq 0$.
	Assume that either $g\geq 3$, or that $g=2$ and $V_{0,\ul}\neq 0$, or that $g=1$ and $\ul$
	can be partitioned into $2$ permissive tuples. Then:

	\begin{enumerate}
		\item[(1)] $\overline{\rho_{p,g,\ul}(\PMod(\Sigma_{g,n}))}$ is a centerless simple Lie subgroup of
		$\mathrm{PSU}_{d_{p,g,\ul}}$;
		\item[(2)] For any maximal ideal $J\subset \Z[\zeta_p]$ of norm $q$ coprime to $p$, the image $\rho_{J,g,\ul}(\PMod(\Sigma_{g,n}))$ is a simple subgroup
		of $\mathrm{PGL}_{d_{p,g,\ul}}(\F_q)$.
	\end{enumerate}
\end{theorem}
Before embarking on the proof, we recall that the \emph{socle} of a finite group $G$ is the subgroup $\mathrm{soc}(G)$
which is the product of all minimal normal subgroups of $G$. It is well known that for any finite group $G$, the subgroup
$\mathrm{soc}(G)$ is characteristic in $G$, and is isomorphic to a direct product of finite simple groups (some of which may be abelian).
We refer to \cite[§2.A]{Isaacs} for an introduction to this notion.

The proof of \Cref{thm:simpleGroup} follows closely the proof of Step 6 of \cite[Theorem 2]{LW05}, with some simplifications due to
Lemma \ref{lemma:spectrumTensorDecomp}.
\begin{proof}
	Let $G:=G_{g,\ul}$ or $G_{J,g,\ul}$. In the former case, let $N=G^{\circ}$ be the connected
	component of the identity in $G$; in the latter case, let $N=\mathrm{soc}(G)$.
	Since by \Cref{lemma:normalIrred}, $V_{g,\ul}$ is irreducible as a
	projective $N$-representation, and furthermore, is a faithful $G$-representation, we get that $Z(N)=1$.
	Hence, in the former case, $N$ is a direct product of
	connected simple centerless Lie groups; in the latter, $N$ is a direct product of non-abelian finite simple groups.
	Let
	$$N\simeq K_1^{n_1}\times \dotsb \times K_k^{n_k}$$
	be the decomposition of $N$ into simple factors, where $K_i$ is non-isomorphic to $K_j$ for $i\neq j$.
	Since the $K_i$'s are simple and centerless, they are perfect. By Proposition \ref{prop:projIrredProduct},
	the irreducible $N$-representation $V_{g,\ul}$ is a tensor product $W_1^{\otimes n_1}\otimes \dotsb \otimes W_k^{\otimes n_k}$,
	where $W_i$ is an irreducible $K_i$-representation for each $i$.
	
	We claim that $k=1$; otherwise the image of $G\longrightarrow \mathrm{Aut}(N)$ leaves each factor $K_i^{n_i}$
	invariant and \Cref{cor:OuterAutomTensor2} provides a tensor decomposition for $V_{g,\ul}$,
	which is a contradiction, since $V_{g,\ul}$ is tensor-indecomposable by Theorem \ref{thm:tensor-indec}.
	
	Now, consider the morphism:	
	\begin{equation}\label{eq:morphismOuterAutom}
		G\longrightarrow \mathrm{Aut}(N)\longrightarrow \mathrm{Out}(N)\simeq \mathrm{Out}(K_1)^{n_1} \rtimes S_{n_1}.
	\end{equation}	
	First, we claim that the image lands in $\mathrm{Out}(K_1)^{n_1}$. Otherwise, since $G$ is (topologically)
	generated by the image of Dehn twists,
	there would be a Dehn twist $t\in G$ such that conjugation by $t$ permutes the factors of $K_1^{n_1}$ non-trivially.
	However, by \Cref{cor:OuterAutomTensor} and Lemma \ref{lemma:spectrumTensorDecomp},
	this would imply that the spectrum of $t$ contains a coset of a non-trivial group of roots of unity,
	which it does not by Proposition \ref{prop:DehnTwistSpectrum}.
	
	By Lemma \ref{lemma:perfect}, $G$ is perfect. However, for $K_1$ a finite simple group or a simple Lie group,
	$\mathrm{Out}(K_1)$ is solvable (see \cite[§4.7 p. 133]{dixonPermutationGroups1996} and \cite[§10.6.10]{LieGroups2007}).
	We conclude that the morphism in \Cref{eq:morphismOuterAutom} is the trivial morphism.
	
	Hence, conjugation by an element $g\in G$ induces an inner automorphism on $N$, say by $n$.
	Then $gn^{-1}$ acts trivially on $N$, but since $V_{g,\ul}$ is a projective irreducible $N$-representation,
	this implies $g=n$. Hence, $G=N$.
	Since $V_{g,\ul}$ is tensor-indecomposable, we have $n_1=1$ and $G=K_1$; so $G$ is a finite simple group or a simple Lie group.
\end{proof}

\section{End of proof of density for surfaces with boundary}
\label{sec:densityProof}

In this section, we prove the following using simplicity of the image (\Cref{thm:simpleGroup}), weight-multiplicity-freeness
(\Cref{prop:weightMultiplicityFree}) and fusion rules (\Cref{prop:notSymmetricPower,prop:notExteriorPower}).

\begin{theorem}\label{thm:density}
	Let $p\geq 7$, $g\geq 1$, $n\geq 0$ and $\ul\in\Lambda^n$ such that either $g\geq 3$, or $g=2$ and $V_{0,\ul}\neq 0$,
	or $g=1$ and $\dim V_{1,\ul}$ is not a power of $4$ and $\ul$ can be partitioned into $2$ permissive tuples.
	Then the image of $\rho_{g,\ul}$ is dense in $\PU_{d_{p,g,\ul}}$.
\end{theorem}

The case $n=0$, $g\geq 2$ of the theorem corresponds to the main results \cite[Theorems 2 and 3]{LW05} of Larsen-Wang.
Note that we provide a different proof which does not rely on any inductive arguments and also covers cases with $n>0$ or with $g=1$.
For fixed $p$ and $g=2$, our assumption on $\ul$ only excludes a finite number of cases; see \Cref{prop:permissiveTuple}.

As mentioned in \Cref{sec:prelim},
the mapping class group representation $\rho_{g,\ul}$
can be linearized into a representation $\widetilde{\rho}_{g,\ul}$ of the universal central extension $\Modtilde(\Sigma_g^n)$.
When $g\geq 2$, the image of $\widetilde{\rho}_{g,\ul}$ lands in $\SU_{d_{p,g,\ul}}$ (\Cref{prop:imageInSU})
and we can deduce the following from \Cref{thm:density}.

\begin{corollary}\label{cor:density}
	Let $p\geq 7$, $g\geq 2$, $n\geq 0$ and $\ul\in\Lambda^n$ such that either $g\geq 3$, or $g=2$ and $V_{0,\ul}\neq 0$.
	Then the image of $\widetilde{\rho}_{g,\ul}$ is dense in $\SU_{d_{p,g,\ul}}$.
\end{corollary}

\begin{remark}\label{rk:tensorProdNot}
	In this whole section and the next one, we will make use of two different notations for tensor products of representations,
	$U\otimes V$ or $U\boxtimes V$.
	We use the latter when we wish to stress that they are representations of two distinct groups $G_1$ and $G_2$.
\end{remark}

\subsection{Reduction to the case where \texorpdfstring{$\bm{V_{g,\ul}}$}{the TQFT vector space} is a symmetric or exterior power}
\label{sec:multiplicityFreeness}

\begin{proposition}\label{prop:notSelfDual}
	Let $p\geq 5$, $g\geq 1$, $n\geq 0$ and $\ul\in\Lambda^n$ such that either $g\geq 2$, or $g=1$ and $\ul$ is permissive.
	Then $V_{g,\ul}$ is not self-dual as a projective $\PMod(\Sigma_{g,n})$-representation.
\end{proposition}
\begin{proof}
	Consider a separating {\scc} $\alpha\subset\Sigma_{g,n}$ such that one of the components of $\Sigma_{g,n}\setminus \alpha$
	is homeomorphic to $\Sigma_{1,1}$, and the corresponding fusion decomposition
	\begin{equation}
		V_{g,\ul}\simeq \bigoplus_{\mu\in\Lambda} V_{1,\mu}\boxtimes V_{g-1,(\ul,\mu)}.
	\end{equation}
	All summands are non-zero thanks to \Cref{cor:nonTriviality} if $g\geq 2$ and thanks to $\ul$ being permissive if $g=1$.
	This decomposition is both a decomposition of $V_{g,\ul}$
	into irreducible $\Mod(\Sigma_1^1)\times\PMod(\Sigma_{g-1,n}^{1})$-representations and the eigenspace decomposition of the Dehn twist $T_\alpha$.
	Assume by contradiction that $V_{g,\ul}\simeq V_{g,\ul}^*$ as a projective $\PMod(\Sigma_{g,n})$-representation.
	Then there exists a bijection $\sigma:\Lambda\ra\Lambda$ such that for each $\mu\in\Lambda$,
	$V_{1,\mu}\boxtimes V_{g-1,(\ul,\mu)}$ is isomorphic to $V_{1,\sigma(\mu)}^*\boxtimes V_{g-1,(\ul,\sigma(\mu))}^*$
	as a $\Mod(\Sigma_1^1)\times\PMod(\Sigma_{g-1,n}^{1})$-representation. However, as recalled in \Cref{lemma:dimTori},
	$V_{1,\sigma(\mu)}$ and $V_{1,\mu}$ have the same dimension if and only if $\mu=\sigma(\mu)$.
	Hence, $\sigma$ is the identity. In terms of eigenspaces for $T_\alpha$, this means that
	there exists a constant $u$ such that for any $\mu\in\Lambda$, the eigenvalue of $T_\alpha$ on
	$V_{1,\mu}\boxtimes V_{g-1,(\ul,\mu)}$ is $u$ times that of $T_\alpha$ on $V_{1,\mu}^*\boxtimes V_{g-1,(\ul,\mu)}^*$.
	This means that $\zeta_p^{i^2}=u\zeta_p^{-i^2}$ for any $i\in \lbrace 1,\ldots,\frac{p-1}{2}\rbrace$, which is impossible for $p\geq 5$.
\end{proof}

\begin{corollary}\label{cor:groupsWithNonSelfDualReps}
	Let $p\geq 7$ be a prime, $g\geq 1$, $n\geq 0$ and $\ul\in\Lambda^n$ such that either $g\geq 3$, or $g=2$ and $V_{0,\ul}\neq 0$,
	or $g=1$ and $\ul$ can be partitioned into $2$ permissive tuples.
	Then $G_{g,\ul}$ is isomorphic to one of
	\begin{equation*}
		\PU_k,\; k\geq 3,\;\;\PSO_{4k+2},\; k\geq 2,\text{ or }\mathrm{E}_6/Z(\mathrm{E}_6).
	\end{equation*}
\end{corollary}
\begin{proof}
	By \Cref{thm:simpleGroup}, $G_{g,\ul}$ is a simple centerless compact Lie group. These are classified and so $G_{g,\ul}$
	is one of $\PU_k$, $k\geq 2$, $\PSO_{k}$, $k\geq 5$, $\mathrm{PSp}_{2k}$, $k\geq 2$, $\mathrm{E}_6/Z(\mathrm{E}_6)$, $\mathrm{E}_7/Z(\mathrm{E}_7)$,
	$\mathrm{E}_8$, $\mathrm{F}_4$ or $\mathrm{G}_2$. Now, for all of these but the ones listed in the statement of the corollary, all representations
	are self-dual. Indeed, given a maximal torus, the dual of the finite-dimensional irreducible representation $V_\nu$ with highest weight $\nu$
	is $V_{-w_0(\nu)}$ with $w_0$ the longest element in the Weyl group \cite[§21.4, Exercise 6]{humphreysIntroductionLieAlgebras1972},
	and for all but the examples in the statement, $w_0$ acts by $-\id$ \cite[VI, §4-13, (XI)]{bourbakiSystemesRacines2007}.
	This concludes the proof, since if $V_{g,\ul}$ were self-dual as a $G_{g,\ul}$-representation,
	then it would also be as a $\PMod(\Sigma_{g,n})$-representation, contradicting \Cref{prop:notSelfDual}.
\end{proof}

A finite-dimensional representation $V$ of a simple compact Lie group $G$ is \emph{weight-multiplicity-free} if given a maximal torus $T\subset G$,
each character of $T$ has multiplicity at most one in $V$, i.e. for any $\chi:T\ra S^1$, $\dim\{v\in V\mid \forall t\in T,\: t\cdot v=\chi(t)v\}\leq 1$.
A finite-dimensional projective representation of a \emph{simply-connected} simple Lie group is always linearizable in a unique way.
We will say that a \emph{projective} representation of a simple compact Lie group $G$ is weight-multiplicity-free if it is weight-multiplicity-free
as a linear representation of the universal cover $\widetilde{G}$ of $G$.

\begin{proposition}\label{prop:weightMultiplicityFree}
	Let $p\geq 7$, $g\geq 1$, $n\geq 0$ and $\ul\in\Lambda^n$ such that either $g\geq 3$, or $g=2$ and $V_{0,\ul}\neq 0$,
	or $g=1$ and $\ul$ can be partitioned into $2$ permissive tuples.
	Then $V_{g,\ul}$ is a \emph{weight-multiplicity-free} projective representation of $G_{g,\ul}$.
\end{proposition}
\begin{proof}
	Consider a pair of pants decomposition of $\Sigma_{g}^{n}$ whose corresponding {\scc s} $\alpha_1,\dotsc,\alpha_{3g-3+n}$ are non-separating.
	This is always possible as $g\geq 1$, for example by choosing the decomposition with the dual graph below.
	\begin{center}
		\ctikzfig{graph_only_nonsep_curves}
	\end{center}
	Denote by $A$ the finite abelian subgroup generated by the images $t_{\alpha_1},\dotsc,t_{\alpha_{3g-3+n}}$
	of the corresponding Dehn twists in $G_{g,\ul}$.
	Every element of $A$ is $p$-torsion. As $p$ is prime, the decomposition of $V_{g,\ul}$ into characters for $A$
	is given by the TQFT basis corresponding to the pair of pants
	decomposition. In particular, every character of $A$ has multiplicity at most $1$ in $V_{g,\ul}$.
	Hence, if we knew that $A$ was included in a maximal torus
	of $G_{g,\ul}$, we could conclude that $V_{g,\ul}$ is weight-multiplicity-free. It is not the case in general that every
	finite abelian subgroup of a simple
	compact Lie group is contained in a maximal torus. Nevertheless, we now prove that it holds for $A\subset G_{g,\ul}$.
	Thanks to \Cref{cor:groupsWithNonSelfDualReps}, we are reduced to the following three cases.
	
	\textbf{Case 1:} $G_{g,\ul}$ is $\PSO(W)$ for $\dim W=4k+2$, $k\geq 2$.
	Consider the decomposition $W\otimes_{\R}\C=\bigoplus_\chi V_\chi$ into characters for $A$.
	As $A$ has \emph{odd order}, each character appearing in the decomposition is either the trivial character $\chi_1$ or not real.
	Hence, we have a decomposition $W=\left(\bigoplus_{\{\chi,\overline{\chi}\}\neq\{\chi_1\}} V_\chi\oplus V_{\overline{\chi}}\right)\oplus V_{\chi_1}$
	preserved by $A$. On each $V_\chi\oplus V_{\overline{\chi}}$, $A$ acts via a torus. Hence, $A$ is included in a maximal torus, as desired.

	\textbf{Case 2:} $G_{g,\ul}$ is $\mathrm{E}_6/Z(\mathrm{E}_6)$. By \cite[§5.3, corollaire 3]{bourbakiGroupesAlgebresLie2007},
	$A$ is contained in the normalizer $N_{\mathrm{E}_6/Z(\mathrm{E}_6)}(T)$ of some maximal torus $T$.
	Now, $N_{\mathrm{E}_6/Z(\mathrm{E}_6)}(T)/T$ is the Weyl group of $\mathrm{E}_6$,
	which has order $2^73^45$ \cite[VI, §12, (IX)]{bourbakiSystemesRacines2007}. As $A$ is $p$-torsion and $p$ is prime and at least $7$,
	$A$ is included in $T$, as desired.

	\textbf{Case 3:} $G_{g,\ul}$ is $\PU(W)$ with $d=\dim W\geq 3$. Then $\widetilde{G}_{g,\ul}$ is $\SU(W)$.
	By spectral theory, any finite abelian subgroup of $\SU_d$ is conjugate in $\SU_d$ to a group of diagonal matrices.
	Hence, we need only show that the pre-image of $A\subset\PU(W)$ in $\SU(W)$ is abelian.
	Applying \Cref{lemma:liftTorusToSL} with $H=\PSU(W)$, $\widetilde{H}=\SU(W)$ and $m=p$ concludes.
\end{proof}

\begin{remark}
	The proof of \Cref{prop:weightMultiplicityFree} relies on the fact that characters of $A$ act on $V_{g,\ul}$ without multiplicity.
	This fact fails for general $\mathrm{SO}(3)$ quantum representations (when $p$ is not assumed to be prime)
	and also for general $\mathrm{PSU}(d)$ quantum representations when $d\geq 3$.
\end{remark}

\begin{lemma}\label{lemma:liftTorusToSL}
	Let $g\geq 1$ and $n\geq 0$.
	Consider \( T \subseteq \PMod(\Sigma_{g,n}) \) the subgroup generated by the images \( t_1, \dots, t_{3g-3+n} \)
	of some non-separating Dehn twists forming a pair of pants decomposition
	and a morphism from $\PMod(\Sigma_{g,n})/\langle t_i^m\mid i=1,\dotsc,3g-3+n\rangle$ to a group $H$ with $m\geq 1$ odd.
	Denote by $A$ the image of $T$ in $H$.
	Let $1\ra U\ra \widetilde{H}\ra H\ra 1$ be a central extension.
	Then the pre-image of $A$ in $\widetilde{H}$ is abelian.
\end{lemma}

\begin{proof}
	Consider the commutator map $H\times H\ra Z(\widetilde{H})$,
	$(x,y)\mapsto [\widetilde{x},\widetilde{y}]=\widetilde{x}\widetilde{y}\widetilde{x}^{-1}\widetilde{y}^{-1}$,
	where $\widetilde{x}$ and $\widetilde{y}$ are any lifts of $x$ and $y$ to $\widetilde{H}$.
	It restricts to a map of abelian groups $\phi:\Lambda^2 A\ra U$.
	The pre-image of $A$ in $\widetilde{H}$ is abelian if and only if $\phi$ is trivial.

	We claim that for any $1\leq i<j\leq 3g-3+n$ there exists $f\in \PMod(\Sigma_{g,n})$ such that $f(\alpha_i)=\alpha_j$ and $f(\alpha_j)=\alpha_i$.
	Indeed, consider the surface $\Sigma'\ra \Sigma_{g,n}$ obtained by cutting along $\alpha_i$ and $\alpha_j$. It
	contains $4$ boundary components: the $2$ pre-images $\alpha_i^1$ and $\alpha_i^2$ of $\alpha_i$, and the $2$ pre-images
	$\alpha_j^1$ and $\alpha_j^2$ of $\alpha_j$. As $\alpha_i$ and $\alpha_j$ are non-separating, $\Sigma'$ has $1$ or $2$ components,
	and in each case we may order the pre-images so that for each $k\in\{1,2\}$, $\alpha_i^k$ and $\alpha_j^k$ lie in the same component.
	Then there exists a self-homeomorphism $f'$ of $\Sigma'$ such that $f'(\alpha_i^k)=\alpha_j^k$ and $f'(\alpha_j^k)=\alpha_i^k$ for $k=1,2$.
	Gluing back, we get the desired $f\in \PMod(\Sigma_{g,n})$.

	Conjugation by the image of $f$ in $H$ is an inner automorphism $\varphi$ of $\widetilde{H}$
	which sends a lift of $t_{\alpha_i}$ to one of $t_{\alpha_j}$
	and vice-versa.
	Hence, in $\widetilde{H}$,
	\begin{equation*}
		\phi(t_{\alpha_i}\wedge t_{\alpha_j})=[\widetilde{t_{\alpha_i}},\widetilde{t_{\alpha_j}}]=%
		\varphi\left([\widetilde{t_{\alpha_i}},\widetilde{t_{\alpha_j}}]\right)=[\varphi(\widetilde{t_{\alpha_i}}),\varphi(\widetilde{t_{\alpha_j}})]%
		=[\widetilde{t_{\alpha_j}},\widetilde{t_{\alpha_i}}]=-\phi(t_{\alpha_i}\wedge t_{\alpha_j}).
	\end{equation*}
	So, for any $1\leq i<j\leq 3g-3+n$, $\phi(t_{\alpha_i}\wedge t_{\alpha_j})$ is both $2$-torsion and $m$-torsion, hence $0$ as $m$ is odd.
	We conclude that $\phi$ is $0$, as desired.
\end{proof}

We now make use of the following classification of weight-multiplicity-free representations due to Howe. The statement in the reference is written in
terms of simply-connected simple algebraic groups,
which have the same representation theory as simply-connected simple compact Lie groups.

\begin{theorem}[{\cite[Theorem 4.6.3]{gelbartSchurLectures19921995}}]\label{thm:HoweMultiplicityFree}
	The only non-trivial weight-multiplicity-free projective representations of $\PU(W)$, $\dim W\geq 2$ are
	the symmetric powers $S^m(W)$ and $S^m(W^*)$, $m\geq 1$,
	and the exterior powers $\Lambda^m(W)$, $1\leq m\leq \dim W-1$. For $\PSO_{4k+2}$, $k\geq 2$,
	the only ones are the regular representation and the two half-spin representations of dimension $2^{2k}$.
	For $\mathrm{E}_6/Z(\mathrm{E}_6)$, the only ones are the two $27$-dimensional representations.
\end{theorem}

Combining \Cref{thm:HoweMultiplicityFree}, \Cref{prop:weightMultiplicityFree} and \Cref{lemma:dimensionDivisibleByP}, we get the following.

\begin{proposition}\label{prop:isSymmetricOrExteriorPower}
	Let $p\geq 7$, $g\geq 1$, $n\geq 0$ and $\ul\in\Lambda^n$ such that either $g\geq 3$, or $g=2$ and $V_{0,\ul}\neq 0$,
	or $g=1$, $\dim V_{1,\ul}$ is not a power of $4$ and $\ul$ can be partitioned into $2$ permissive tuples.
	Then $G_{g,\ul}$ is some $\PU(W)$ with $\dim W\geq 3$, and $V_{g,\ul}$ is a symmetric power $S^m(W)$ or $S^m(W^*)$, $m\geq 1$,
	or an exterior power $\Lambda^m(W)$, $1\leq m\leq \dim W-1$.
\end{proposition}
\begin{proof}
	By \Cref{cor:groupsWithNonSelfDualReps}, we know that $G_{g,\ul}$ is of the form $\PU(W)$, $\dim W\geq 3$, $\PSO(W)$, $\dim(W)=4k+2, k\geq 2$,
	or $\mathrm{E}_6/Z(\mathrm{E}_6)$. We now apply \Cref{thm:HoweMultiplicityFree} and \Cref{prop:weightMultiplicityFree}. If $G_{g,\ul}$ is $\PSO(W)$,
	$V_{g,\ul}$ is $W$ or a half-spin representation of dimension $2^{2k}$, and if $G_{g,\ul}$
	is $\mathrm{E}_6/Z(\mathrm{E}_6)$, $V_{g,\ul}$ has dimension $27$.
	None of these cases are possible. Indeed $W$ is self-dual but $V_{g,\ul}$ is not by \Cref{prop:notSelfDual},
	by \Cref{lemma:dimensionDivisibleByP} the dimension of $V_{g,\ul}$ is divisible by $p$ for $g\geq 2$, hence not $27$ or a power of $4$.
	If $g=1$, the dimension of $V_{g,\ul}$ cannot be $27$: as $\ul$ can be partitioned into $2$ admissible tuples by assumption,
	$$\dim V_{1,\ul}\geq \sum_{\substack{\mu_1,\mu_2\\\in\Lambda}}\sum_{\substack{(\nu,\mu_1,\mu_2)\\\text{ admissible}}} \dim V_{1,\nu}=
	\sum_{\substack{\mu_1,\mu_2\\\in\Lambda}}\sum_{\substack{(\nu,\mu_1,\mu_2)\\\text{ admissible}}} \frac{p-1-\nu}{2}$$
	which exceeds $27$ for $p\geq 11$ (the lower bound above is $155$ for $p=11$ and increases with $p$); for $p=7$,
	if $\ul$ can be partitioned into $3$ admissible tuples,
	by a similar argument, the dimension exceeds $\sum_{\mu_1,\mu_2,\mu_3\in\Lambda}\dim V_{1,(\mu_1,\mu_2,\mu_3)}=129>27$,
	and if it cannot be partitioned into $3$ admissible tuples,
	then by \Cref{prop:permissiveTuple}, $\ul$ has at most $11$ non-zero entries and we ran a SageMath script to
	check that the dimension of $V_{7,1,\ul}$ is never $27$ for $\ul\in\Lambda^n$ with $n\leq 11$.
	Hence, $G_{g,\ul}$ is of the form $\PU(W)$ and $V_{g,\ul}$ is a symmetric or exterior power.
\end{proof}

\subsection{\texorpdfstring{$\bm{V_{g,\ul}}$}{The TQFT vector space} is not a non-trivial symmetric or exterior power}
\label{sec:notSymmetricOrExteriorPower}

Consider $W$ as in \Cref{prop:isSymmetricOrExteriorPower}.
In this section, we prove that $V_{g,\ul}$ cannot be a non-trivial symmetric or exterior power of $W$
by showing that this would not be compatible with fusion rules along a non-separating curve.

To start with, we recall a standard formula for the decomposition of symmetric or exterior powers of the tensor product of two representations.
We will write $y\vdash m$ if $y$ is a Young diagram of size $m$ and we denote the dual diagram of $y$ by $y'$. For $y$ a Young diagram of size $m$,
we denote by $S^y$ the corresponding Schur functor, defined, according to Schur-Weyl duality, by:
$$U^{\otimes m}=\underset{y \vdash m}{\bigoplus} S^y(U)\otimes V_y,$$
for any finite-dimensional $\C$-vector space, where $V_y$ is the irreducible representation of the symmetric group $\mathfrak{S}_n$ associated to $y$.
We note that for $U$ a finite-dimensional $\C$-vector space, $S^y(U)\neq 0$ whenever the number of rows of $y$ is $\leq \dim U$.
The following is a consequence of Schur-Weyl duality (see \cite[Exercise 6.11]{fultonRepresentationTheory2004} for instance).

\begin{proposition}[Cauchy decomposition formula] \label{prop:CauchyDecomp} Let $U_1$ and $U_2$ be two finite-dimensional $\C$-vector spaces.
	Then, for any $m\geq 0$, one has the following decompositions into irreducible $\GL(U_1)\times \GL(U_2)$-representations:
	$$S^m(U_1\boxtimes U_2)=\bigoplus_{y\vdash m}S^y(U_1)\boxtimes S^y(U_2)$$
	and
	$$\Lambda^m(U_1 \boxtimes U_2)=\bigoplus_{y\vdash m}S^y(U_1)\boxtimes S^{y'}(U_2).$$
	In particular, $S^m(U_1\boxtimes U_2)$ is a reducible $\GL(U_1)\times \GL(U_2)$-representation when $\dim U_1$, $\dim U_2$ and $m$
	are each at least $2$.
	Similarly, $\Lambda^m(U_1 \boxtimes U_2)$ is a reducible $\GL(U_1)\times \GL(U_2)$-representation when $\dim U_1$, $\dim U_2$ are at least $2$
	and $2\leq m\leq \dim (U_1 \boxtimes U_2)-2$.
\end{proposition}

We will need the following lemma.

\begin{lemma}\label{lem:tensorProduct}
	Let $W$ be a finite-dimensional vector space over an algebraically closed field.
	Consider semisimple $t_1,t_2\in \SU(W)$ that commute in $\PU(W)$.
	Set $V=S^m(W)$ with $m\geq 1$ or $V=\Lambda^m(W)$ with $1\leq m\leq \dim W-1$.
	Assume that the actions of $t_1$ and $t_2$ on $V$ commute and that the action of $t_1$ on $V$
	does not contain a coset of a non-trivial group of roots of unity.
	Then $t_1$ and $t_2$ commute.
\end{lemma}
\begin{proof}
	Assume by contradiction that $t_1t_2=u t_2t_1$ where $u\neq 1$ is a scalar.
	Let $W=\bigoplus_x W_x$ be the eigenspace decomposition for $t_1$. Then for any $x$, $t_2 W_x = W_{ux}$.
	Hence, $u$ is a root of unity, say primitive of order $d$, and the spectrum of $t_1$ contains
	a coset of the group of $d$-th roots of unity. Choose $x$ such that $W_x\neq 0$ and $w\in W_x\setminus\{0\}$, and set $w_i=t_2^i(w)$.
	Let us first deal with the case $V=S^m(W)$.
	Then the $w_iw_0^{m-1}$ for $0\leq i\leq d-1$ are eigenvectors for the action of $t_1$ on $S^m(W)$ and the corresponding
	eigenvalues form a coset of the set of $d$-th roots of unity, a contradiction. Let us now turn to the case $V=\Lambda^m(W)$.
	As the actions of $t_1$ and $t_2$ on $V$ commute and satisfy $t_1t_2=u^m t_2t_1$, we must have $u^m=1$. Hence, $d$ divides $m$.
	As $m<\dim W$ and both are multiples of $d$, we have $m\leq \dim W-d$. Hence, there exist $u_1,\dotsc,u_{m-1}\in W$
	linearly independent and independent of $w_0,\dotsc,w_{d-1}$. As $t_1$ is semisimple, we may and will assume that
	$u_1,\dotsc,u_{m-1}$ are eigenvectors for $t_1$. Then the $w_i\wedge u_1\wedge\dotsb\wedge u_{m-1}$ for
	$0\leq i\leq d-1$ are eigenvectors for the action of $t_1$ on $\Lambda^m(W)$ and the corresponding
	eigenvalues form a coset of the set of $d$-th roots of unity, a contradiction.
\end{proof}

We are now ready to prove the main results of this section.
\begin{proposition}\label{prop:notSymmetricPower}
	Using the notation of \Cref{prop:isSymmetricOrExteriorPower}, $V_{g,\ul}$ is not one of $S^m(W)$ or $S^m(W^*)$ for $m\geq 2$.
\end{proposition}
\begin{proof}
	Replacing $W$ by $W^*$ if necessary, we may assume without loss of generality that $V_{g,\ul}=S^m(W)$ for some $m\geq 2$.
	Consider a separating simple closed curve $\alpha\subset\Sigma_{g,n}$ such that $\Sigma_{g,n}\setminus \alpha$ contains a component
	homeomorphic to $\Sigma_{1,1}$ and consider the corresponding fusion decomposition
	\begin{equation}\label{eq:fusionDecomposition}
		V_{g,\ul}=\bigoplus_{\mu\in\Lambda}V_{1,\mu}\boxtimes V_{g-1,(\ul,\mu)}.
	\end{equation}
	All summands are non-zero by \Cref{cor:nonTriviality} if $g\geq 2$ and by \Cref{lemma:dimensionGenusZero} if $g=1$.
	This is both the decomposition into eigenspaces for the Dehn twist $T_\alpha$ and the decomposition into
	irreducible $\Mod(\Sigma_1^1)\times\PMod(\Sigma_{g-1,n}^1)$-representations. We may choose a decomposition
	\begin{equation}\label{eq:decompositionOfW}
		W=\bigoplus_{i\in I}W_i^1\boxtimes W_i^2
	\end{equation}
	of $W$ into irreducible $\Mod(\Sigma_1^1)\times\PMod(\Sigma_{g-1,n}^1)$-representations.
	Indeed, let $\widetilde{G}_1$ and $\widetilde{G}_2$
	be the pre-images in $\SU(W)$ of the images in $\PU(W)$ of $\Mod(\Sigma_1^1)$ and $\PMod(\Sigma_{g-1,n}^1)$.
	Then, by \Cref{prop:DehnTwistSpectrum,lem:tensorProduct}, for any $\widetilde{t}_i\in \widetilde{G}_i$, $i=1,2$ lifts of images of Dehn twists,
	$\widetilde{t}_1$ and $\widetilde{t}_2$ commute.
	As $\Mod(\Sigma_1^1)$ and $\PMod(\Sigma_{g-1,n}^1)$ are generated by Dehn twists, $\widetilde{G}_1$ and $\widetilde{G}_2$ commute.
	Hence, we may obtain a decomposition as in \Cref{eq:decompositionOfW} by considering the decomposition
	of $W$ into irreducible $\widetilde{G}_1\times \widetilde{G}_2$-modules.

	Then $V_{g,\ul}=S^m(W)$ gives the following decomposition
	of $V_{g,\ul}$ into $\Mod(\Sigma_1^1)\times\PMod(\Sigma_{g,n}^1)$-representations
	\begin{equation}\label{eq:symmetricPowerDecomposition}
		V_{g,\ul}=\bigoplus_{\substack{f:I\ra \N\\\sum_i f(i)=m}} \bigotimes_{i\in I}S^{f(i)}\left(W_i^1\boxtimes W_i^2\right).
	\end{equation}
	Now, $T_\alpha$ acts by a scalar on each summand in \Cref{eq:symmetricPowerDecomposition}. Hence, each summand is an irreducible
	$\Mod(\Sigma_1^1)\times\PMod(\Sigma_{g,n}^1)$-representation and we have a bijection between $\Lambda$
	and $F=\{f:I\ra\N\mid \sum_if(i)=m\}$ identifying summands in \Cref{eq:symmetricPowerDecomposition} and in \Cref{eq:fusionDecomposition}.

	Let us decompose $I$ according to dimensions. Set $I_{12}=\{i\mid \dim W_i^1=\dim W_i^2=1\}$,
	$I_1=\{i\mid \dim W_i^1=1\}\setminus I_{12}$, $I_2=\{i\mid \dim W_i^2=1\}\setminus I_{12}$, $I_0=I\setminus (I_{12}\sqcup I_1\sqcup I_2)$.
	If $I_0\neq \varnothing$, we may choose $f\in F$ such that $f(i)=m$ for some $i\in I_0$; then the corresponding summand in
	\Cref{eq:symmetricPowerDecomposition} would be reducible by Proposition \ref{prop:CauchyDecomp}. So, $I_0=\varnothing$. If $I_2\neq\varnothing$,
	then setting $f\in F$ such that $f(i)=m$ for some $i\in I_2$, the corresponding summand
	$V_{1,\mu}\boxtimes V_{g-1,(\ul,\mu)}$ would satisfy $\dim V_{g-1,(\ul,\mu)}=1$ and $\dim V_{1,\mu}\geq 2$. This is impossible:
	if $g=2$ and $\ul=(0,\dotsc,0)$, this is absurd as $\dim V_{1,\mu}$ would be both $=1$ and $\geq 2$, and otherwise
	this contradicts either \Cref{lemma:dimAtLeastTwo} or \Cref{lemma:dimensionGenusZero}. So, $I_2=\varnothing$.

	If $g=2$ and $\ul=(0,\dotsc,0)$, we have $I_1=\varnothing$ for the same reason. So, $I=I_{12}$ and each summand in \Cref{eq:symmetricPowerDecomposition}
	has dimension $1$, contradicting \Cref{lemma:dimTori}. In the other cases, the same reasoning gives $I_{12}=\varnothing$ and thus
	$I=I_1$. So, each summand $V_{1,\mu}\boxtimes V_{g-1,(\ul,\mu)}$ satisfies $\dim V_{1,\mu}=1$, contradicting \Cref{lemma:dimTori}.
	Both cases lead to contradictions, as desired.
\end{proof}

\begin{proposition}\label{prop:notExteriorPower}
	Using the notation of \Cref{prop:isSymmetricOrExteriorPower}, $V_{g,\ul}$ is not one of $\Lambda^m(W)$ for $2\leq m\leq \dim W-2$.
\end{proposition}
\begin{proof}
	The proof follows the same outline as that of \Cref{prop:notSymmetricPower} but is a little more technical.
	Again, we consider $\alpha\subset \Sigma_{g,n}$ and corresponding decompositions of $V_{g,\ul}$ and $W$ into
	irreducible representations given in \Cref{eq:fusionDecomposition,eq:decompositionOfW} (again applying \Cref{lem:tensorProduct} to obtain the latter).
	The latter induces the decomposition
	\begin{equation}\label{eq:exteriorPowerDecomposition}
		V_{g,\ul}=\bigoplus_{\substack{f:I\ra \N\\\sum_i f(i)=m\\\forall i,\: f(i)\leq \dim W_i^1\boxtimes W_i^2}}%
		\bigotimes_{i\in I}\Lambda^{f(i)}\left(W_i^1\boxtimes W_i^2\right).
	\end{equation}
	Similarly to the case of symmetric powers, each summand in \Cref{eq:exteriorPowerDecomposition}
	is an irreducible $\Mod(\Sigma_1^1)\times\PMod(\Sigma_{g,n}^1)$-representation and
	we have a bijection between $\Lambda$ and
	$F=\{f:I\ra\N\mid \sum_i f(i)=m\text{ and }\forall i,\: f(i)\leq \dim W_i^1\boxtimes W_i^2\}$ identifying the summands in \Cref{eq:fusionDecomposition}
	with those in \Cref{eq:exteriorPowerDecomposition}.

	As in the proof of \Cref{prop:notSymmetricPower}, consider the decomposition $I=I_{12}\sqcup I_1\sqcup I_2\sqcup I_0$.
	Assume by contradiction that $I_0\neq \varnothing$. Choose some $f\in F$. We split the proof into the following cases:
	
	\noindent\textbf{(0)} If there exists $i\in I_0$ such that $2\leq f(i)\leq \dim (W_i^1\boxtimes W_i^2) - 2$,
	then the corresponding summand in \Cref{eq:exteriorPowerDecomposition}
	is reducible by Proposition \ref{prop:CauchyDecomp}, contradiction.
	
	\noindent\textbf{(1)} Assume there exists $i\in I_0$ such that $f(i)=1$.
	
	\noindent\textbf{(1.1)} If there exists $i '\in I\setminus\{i\}$ such that $f(i')\geq 1$, consider $\widetilde{f}\in F$
	which coincides with $f$ outside
	of $\{i,i'\}$ and satisfies $\widetilde{f}(i)=2$ and $\widetilde{f}(i')=f(i')-1$. Then $\widetilde{f}$
	leads to a contradiction as in case \textbf{(0)}.
	
	\noindent\textbf{(1.2)} If for any $i '\in I\setminus\{i\}$, $f(i')=0$, then $m=\sum_jf(j)=1$, contradicting $m\geq 2$.
	So, for any $i\in I_0$, $f(i)\neq 1$.
	
	\noindent\textbf{(2)} Assume there exists $i\in I_0$ such that $f(i)=\dim(W_i^1\boxtimes W_i^2)-1$.
	This case is symmetric to \textbf{(1)} and hence leads to a contradiction.
	More precisely, replacing $W$ by $W^*$, we have $\Lambda^m(W^*)\simeq \Lambda^{\dim W-m}(W)$ and this induces a bijection between corresponding
	decompositions in \Cref{eq:exteriorPowerDecomposition} sending the summand corresponding to $f$ to that corresponding to
	$i\mapsto (\dim(W_i^1\boxtimes W_i^2)-f(i))$.
	
	\noindent\textbf{(3)} Assume there exists $i\in I_0$ such that $f(i)=0$.
	
	\noindent\textbf{(3.1)} If there exists $i'\in I\setminus\{i\}$ such that $f(i')\geq 2$, consider $\widetilde{f}\in F$
	which coincides with $f$ outside
	of $\{i,i'\}$ and satisfies $\widetilde{f}(i)=2$ and $\widetilde{f}(i')=f(i')-2$.
	Then $\widetilde{f}$ leads to a contradiction as in case \textbf{(0)}.
	
	\noindent\textbf{(3.2)} If for any $i'\in I\setminus\{i\}$, $f(i')\leq 1$, as $\sum_{j\neq i}f(j)=\sum_{j}f(j)=m\geq 2$,
	there exist distinct elements $i'$ and $i''$ of $I\setminus\{i\}$ such that $f(i')=f(i'')=1$.
	Consider $\widetilde{f}\in F$ which coincides with $f$ outside
	of $\{i,i',i''\}$ and satisfies $\widetilde{f}(i)=2$ and $\widetilde{f}(i')=\widetilde{f}(i'')=0$.
	Then $\widetilde{f}$ leads to a contradiction as in case \textbf{(0)}.
	
	\noindent\textbf{(4)} Assume there exists $i\in I_0$ such that $f(i)=\dim(W_i^1\boxtimes W_i^2)$.
	By the same reasoning as in \textbf{(2)}, this case is symmetric
	to case \textbf{(3)} and hence also leads to a contradiction.
	As all cases lead to contradictions, \emph{we have $I_0=\varnothing$.}

	If $I_1=\varnothing$, then for any $f\in F\simeq \Lambda$, the corresponding summand $V_{1,\mu}\boxtimes V_{g-1,(\ul,\mu)}$ satisfies
	$\dim V_{g-1,(\ul,\mu)}=1$, contradicting \Cref{cor:nonTriviality} or \Cref{lemma:dimensionGenusZero}. So, $I_1\neq\varnothing$.
	Similarly, $I_2\neq\varnothing$.

	Let $d_{12}=|I_{12}|=\sum_{i\in I_{12}} \dim(W_i^1\boxtimes W_i^2)$, $d_1=\sum_{i\in I_{1}} \dim(W_i^1\boxtimes W_i^2)$
	and $d_2=\sum_{i\in I_{2}} \dim(W_i^1\boxtimes W_i^2)$.
	If $m< d_2+d_{12}$, then there exists $f\in F$ which vanishes on $I_1$ and such that $0<f(i)<\dim W_i^1=\dim W_i^1\boxtimes W_i^2$
	for some $i$ in $I_2$.
	If $m>d_1$, then there exists $f\in F$ such that for any $i$ in $I_1$,
	$f(i)=\dim W_i^2=\dim W_i^1\boxtimes W_i^2$, and such that $1<f(i)<\dim W_i^1$ for some $i$ in $I_2$.
	In both cases, the summand $V_{1,\mu}\boxtimes V_{g-1,(\ul,\mu)}$ corresponding to $f$ will satisfy
	$\dim V_{g-1,(\ul,\mu)}=1$ and $\dim V_{1,\mu}\geq 2$, contradicting \Cref{cor:nonTriviality} or \Cref{lemma:dimensionGenusZero}.
	So, $d_1\geq m\geq d_2+d_{12}$.
	
	As $m\leq d_1$, there exists $f_1\in F$ which vanishes on $I_2$. As $m\geq d_2$, there exists $f_2\in F$
	such that for any $i$ in $I_2$, $f_2(i)=\dim W_i^1=\dim W_i^1\boxtimes W_i^2$.
	The respective summands $V_{1,\mu_1}\boxtimes V_{g-1,(\ul,\mu_1)}$ and $V_{1,\mu_2}\boxtimes V_{g-1,(\ul,\mu_2)}$
	then satisfy $\mu_1\neq \mu_2$, $\dim V_{1,\mu_1}=1$ and $\dim V_{1,\mu_2}=1$. This contradicts \Cref{lemma:dimTori}.
	
	All cases lead to contradictions, as desired.
\end{proof}

Combining \Cref{prop:isSymmetricOrExteriorPower,prop:notSymmetricPower,prop:notExteriorPower} proves \Cref{thm:density}.
Let us now prove \Cref{cor:density}.

\begin{proposition}\label{prop:imageInSU}
	Let $g\geq 2$, $n\geq 0$, $p\geq 7$ and $\ul\in\Lambda^n$.
	Then $\widetilde{\rho}_{g,\ul}$ is valued in $\SU_{d_{p,g,\ul}}$.
\end{proposition}

\begin{proof}
	Let us show that $\det(\widetilde{\rho}_{g,\ul})$ is trivial.
	Let $c$ be a generator of the central extension $\Z\subset\Modtilde(\Sigma_g^n)$.
	Then $c$ and Dehn twists along boundary components act on $V_{g,\ul}$ by scalars that are $p$-th roots of unity.
	As $g\geq 2$, by \Cref{lemma:dimensionDivisibleByP},
	$\dim V_{g,\ul}$ is a multiple of $p$. Hence, these elements act with trivial determinant and
	$\det(\widetilde{\rho}_{g,\ul})$ factors through $\PMod(\Sigma_{g,n})$.
	Now, $\PMod(\Sigma_{g,n})^\mathrm{ab}$ is either trivial or $\Z/10\Z$ generated by the image of a non-separating twist,
	see \cite[§5]{Korkmaz}.
	A non-separating twist has order $p$ in the image of $\rho_{g,\ul}$, which is prime to $10$ since $p\geq 7$.
	Hence, $\det(\widetilde{\rho}_{g,\ul})$ is trivial.
\end{proof}

\begin{proof}[Proof of \Cref{cor:density}]
	By \Cref{prop:imageInSU}, $\widetilde{\rho}_{g,\ul}$ is valued in $\SU_{d_{p,g,\ul}}$.
	Let $H$ denote the closure of its image. By \Cref{thm:density}, $H$ surjects onto $\PSU_{d_{p,g,\ul}}$.
	Now, $\SU_{d_{p,g,\ul}}$ is perfect, hence the commutator map $(x,y)\mapsto xyx^{-1}y^{-1}$,
	$\PSU_{d_{p,g,\ul}}\times \PSU_{d_{p,g,\ul}}\ra \SU_{d_{p,g,\ul}}$ is surjective.
	Thus, the commutator map $H\times H\ra \SU_{d_{p,g,\ul}}$ is surjective and $H=\SU_{d_{p,g,\ul}}$, as desired.
\end{proof}

\section{Surjectivity mod \texorpdfstring{$J$}{J} of the \texorpdfstring{representations $\rho_{p,g,\ul}$}{SO(3) quantum representations}}
\label{sec:caracFinie}

The aim of this section is to prove the following.

\begin{theorem}\label{thm:surjectivity}
	Let $p\geq 7$ be a prime, $g\geq 3$, $n\geq 0$, $\ul\in\Lambda^n$ and $J\subseteq\Z[\zeta_p]$ a maximal ideal prime to $p$.
	Then the image $G_{J,g,\ul}$ of $\rho_{J,g,\ul}:\PMod(\Sigma_{g,n})\ra \PSL(V_{J,g,\ul})$ is
	all of $\PSL(V_{J,g,\ul})$ if the degree of $\F_q=\Z[\zeta_p]/J$ over its prime field is odd, and $\PSU(V_{J,g,\ul})$ otherwise.
\end{theorem}

\begin{remark}
	In the case where the degree of $\F_q$ over its prime field is even, the sesquilinear form on $V_{J,g,\ul}$
	preserved by $\rho_{J,g,\ul}$ is induced by the TQFT sesquilinear form defined over $\Z[\zeta_p]$.
	When the degree of $\F_q$ over its prime field is odd, this form disappears as $J$ is then not stable
	under complex conjugation.
\end{remark}

\begin{remark}
	\Cref{thm:surjectivity}, unlike \Cref{thm:density}, is limited to genus $g\geq 3$ because we use that
	the obstruction to linearizing $\rho_{p,g,\ul}$ has order divisible by $p$ in \Cref{prop:typeANotSemiTransverse}.
	That proposition is crucial to exclude the possibility that $p\mid \ell$ in \Cref{sec:nonDefiningCharacteristic}.
\end{remark}

We now explain an improvement of \Cref{thm:surjectivity} to the case where the non-zero ideal $J\subset \Z[\zeta_p]$ is not necessarily maximal.
To state it, we need to work over $\Zrp=\Z[\zeta_p]\cap \R$ rather than $\Z[\zeta_p]$.
Let us make some reminders from \Cref{sec:prelim}.
The mapping class group representation $\rho_{g,\ul}$
can be linearized into a representation $\widetilde{\rho}_{g,\ul}$ of the universal central extension $\Modtilde(\Sigma_g^n)$.
The vector space $V_{J,g}$ has a natural $\Z[\zeta_p]$-lattice preserved by the $\Modtilde(\Sigma_g^n)$-action
and is equipped with an invariant Hermitian form $h_{g,\ul}$, defined over $\Z[\zeta_{p}]$ and non-degenerate over $\Z[\zeta_p,\frac{1}{p}]$.%
\footnote{The Hermitian form is sometimes skew over $\Z[\zeta_{p}]$; see comments at the end of \Cref{sec:prelimDedekindDom}.
This detail is irrelevant except for the comments made in \Cref{sec:modp}.}

Then by \Cref{prop:imageInSU}, for $g\geq 2$,
$\widetilde{\rho}_{g,\ul}$ is valued in the group $\SU(h_{g,\ul})$ of elements of $\SL(V_{J,g})$ preserving $h_{g,\ul}$.
This is an algebraic group defined over $\Zrp$.

For any ideal $I$ of $\Zrp$, one can consider the reduction $\widetilde{\rho}_{I,g,\ul}$
of $\widetilde{\rho}_{g,\ul}$ modulo $I$, which is valued in $\SU(h_{g,\ul})(\Zrp/I)$.
For example, when $I$ is maximal, for any $J\subset \Zzp$ maximal with $I=J\cap \Zrp$, $\widetilde{\rho}_{I,g,\ul}$ is a linearization
of $\rho_{J,g,\ul}$ and $\SU(h_{g,\ul})(\Zrp/I)$ is $\SU(V_{J,g,\ul})$ if $\overline{J}=J$ and $\SL(V_{J,g,\ul})$ if $\overline{J}\neq J$.
Note that in the latter case $\rho_{J,g,\ul}$ and $\rho_{\overline{J},g,\ul}$ are dual representations
and that $J$ and $\overline{J}$ are the only ideals over $I$.
The improvement of \Cref{thm:surjectivity} is the following.

\begin{theorem}\label{thm:surjectivityImproved}
	Let $p\geq 7$ be a prime, $g\geq 3$, $n\geq 0$, $\ul\in\Lambda^n$ and $I\subseteq\Zrp$ a non-zero ideal prime to $p$.
	Then the reduction
	\begin{equation*}
		\widetilde{\rho}_{I,g,\ul}:\Modtilde(\Sigma_g^n)\lra \SU(h_{g,\ul})(\Zrp/I)
	\end{equation*}
	is surjective.
\end{theorem}

\begin{remark}\label{rmk:surjectivityProduct}
	In particular, \Cref{thm:surjectivityImproved} implies that given distinct maximal ideals $J_1,\dotsc,J_r$, $J'_1,\dotsc,J'_s\subset\Zzp$
	with $J_u\neq\overline{J_v}$ for any $1\leq u,v\leq r$ and $J'_u=\overline{J'_u}$ for $1\leq u \leq s$, the product representation
	\begin{equation*}
		\prod_{i=1}^r\rho_{J_i,g,\ul}\times\prod_{i=1}^s\rho_{J'_i,g,\ul} :\PMod(\Sigma_{g,n})\lra
		\prod_{i=1}^r\PSL(V_{J_i,g,\ul})\times\prod_{i=1}^s\PSU(V_{J'_i,g,\ul})
	\end{equation*}
	is surjective.
\end{remark}

Let us now make some comments on the proofs of \Cref{thm:surjectivity,thm:surjectivityImproved}.
Thanks to \Cref{thm:simpleGroup}, we know that, under the assumptions of \Cref{thm:surjectivity}, $G_{J,g,\ul}$
is a simple finite group. Then the proof of \Cref{thm:surjectivity} is in two steps.
We first show in \Cref{sec:nonDefiningCharacteristic} that $G_{J,g,\ul}$ is of Lie type A of the same characteristic as $\F_q=\Z[\zeta_p]/J$.
The main inputs are knowledge about the linearity defect of $\rho_{J,g,\ul}$ and
lower bounds of Seitz-Zalesskii on the dimensions of non-trivial representations
of a finite simple group of Lie type in non-defining characteristic \cite{seitzMinimalDegreesProjective1993}.
We then show that $G_{J,g,\ul}$ must be $\PSL(V_{J,g,\ul})$ or $\PSU(V_{J,g,\ul})$. This is the content
of \Cref{sec:definingCharacteristic} and the arguments parallel those of \Cref{sec:densityProof}:
this is essentially enabled by the close proximity between the modular representation theory of finite groups of Lie type
in defining characteristic and representation theory of Lie groups in positive characteristic.

\Cref{thm:surjectivityImproved} is deduced from \Cref{thm:surjectivity} in two steps: we first prove the theorem for reduced $I$
using fusion decomposition arguments, and then deduce the case of general $I$ using a result
\cite[Theorem 1.3]{vasiuSurjectivityCriteriaPadic2003} of Vasiu
relating surjectivity modulo a maximal ideal to surjectivity modulo its powers. The proof is given in
\Cref{sec:conclusionSurjectivity}.

\subsection{Excluding non-Lie type and non-defining characteristic}
\label{sec:nonDefiningCharacteristic}

Let $p\geq 7$ be a prime, $g\geq 3$, $n\geq 0$, $\ul\in \Lambda^n$ and $J\subset\Z[\zeta_p]$ maximal ideal prime to $p$.
Set $\F_q=\Z[\zeta_p]/J$.
\emph{We fix these notations until the end of this section.}

Thanks to \Cref{thm:simpleGroup}, we know that $G_{J,g,\ul}$ is a simple finite group.
In this section, we prove that $G_{J,g,\ul}$ is some $\PSL_k(\F_\ell)$ or $\PSU_k(\F_\ell)$ for some integer $k\geq 2$ and some field $\F_\ell$ of the
\emph{same characteristic} as $\F_q$.

As a first step, we show that $G_{J,g,\ul}$ is of Lie type A in characteristic prime to $p$ using knowledge
on the linearity defect of $\rho_{J,g,\ul}$.

\begin{proposition}\label{prop:typeANotSemiTransverse}
	Let $p\geq 7$ be a prime and $J\subset\Z[\zeta_p]$ maximal ideal prime to $p$.
	Let $g\geq 3$, $n\geq 0$ and $\ul\in \Lambda^n$.
	Then $G_{J,g,\ul} = \PSL_k(\mathbb{F}_\ell)$ or $G_{J,g,\ul}=\PSU_k(\mathbb{F}_\ell)$ with $p \nmid \ell$
	and for some integer $k$.
\end{proposition}

\begin{proof}
	We shall use the notations \( r = \dim V_{g,\ul} \) and \( \mathcal{Z} = Z(\SL_r(\mathbb{F}_q)) \).
	Note that \( \mathcal{Z} \cong \mathbb{Z}/r' \mathbb{Z} \),
	where \( r' \) is the largest divisor of \( r \) prime to \( q \).
	As $p\nmid q$, we have that \( p \mid r \) and hence \( p \mid r' \).
	The proof will follow from the two claims below.

	\begin{claim}[1]
		If \( G_{J,g,\ul} \) is not of Lie type A or if \( G_{J,g,\ul} \) is of Lie type A in characteristic \( p \),
			then \( \abs{H^2(G_{J,g,\ul}, \mathcal{Z})} \) is prime to \( p \).
	\end{claim}
	\begin{proof}
		By \Cref{thm:simpleGroup}, $G_{J,g,\ul}$ is a simple non-cyclic finite group. First assume that \( G_{J,g,\ul} \) is not of Lie type A.
		By the classification of simple finite groups, the Schur multiplier \( \abs{H_2(G_{J,g,\ul}, \mathcal{Z})} \)
		is of the form $2^a3^b$ (see \cite[§1.3, §3.3]{conwayAtlasFiniteGroups1985}).
		Then by the Universal Coefficient Theorem and the fact that \( G_{J,g,\ul} \) is perfect (\Cref{lemma:perfect}):
		\[ H^2(G_{J,g,\ul}, \mathcal{Z}) \cong \operatorname{Hom}_{\mathbb{Z}}(H_2(G_{J,g,\ul}, \mathbb{Z}), \mathcal{Z}). \]
		This proves the claim in that case. Now, assume that \( G_{J,g,\ul} \) is of Lie type A in characteristic \( p \).
		Then \( G_{J,g,\ul} \) is some $\PSL_k(\mathbb{F}_{p^e})$ or some $\PSU_k(\mathbb{F}_{p^e})$. Now
		the Schur multipliers \( \abs{H_2(\PSL_k(\mathbb{F}_{p^e}), \mathbb{Z})} = k \wedge (p^e - 1) \)
		and \( \abs{H_2(\PSU_k(\mathbb{F}_{p^e}), \mathbb{Z})} = k \wedge (p^{\frac{e}{2}} + 1) \) are both prime to \( p \),
		see \cite[§3.3]{conwayAtlasFiniteGroups1985}.
		Again, we conclude using the Universal Coefficient Theorem.
	\end{proof}

	\begin{claim}[2]
	The obstruction \( \omega \in H^2(G_{J,g,\ul}, \mathcal{Z}) \) to linearizing \( G_{J,g,\ul}\ra \PSL_r(\F_q) \) has order divisible by \( p \).
\end{claim}
\begin{proof}
	Let \( \Modtilde(\Sigma_{g}^{n}) \) denote the universal central extension of the mapping class group \( \Mod(\Sigma_{g}^{n}) \).
	It is the pullback along $\Mod(\Sigma_{g}^{n})\ra\Mod(\Sigma_g)$
	of the central extension denoted $\widetilde{\Gamma}_g^{++}$ in \cite{gilmerMaslovIndexLagrangians2013}.
	The corresponding extension class $\widetilde{\omega}\in H^2(\Mod(\Sigma_{g}^{n}),\Z)$ is
	primitive in the free $\Z$-module $H^2(\Mod(\Sigma_{g}^{n}),\Z)/\mathrm{torsion}$. Indeed,
	the pullback of $\widetilde{\omega}$ along the map $\Mod(\Sigma_g^1)\ra\Mod(\Sigma_{g}^{n})$ is a generator of
	$H^2(\Mod(\Sigma_g^1),\Z)\simeq\Z$; see \cite[§6]{Korkmaz}.
	
	The representation $\rho_{J,g,\ul}:\Mod(\Sigma_{g}^{n})\ra \PSL_r(\F_q)$ can be linearized into
	$\widetilde{\rho}_{J,g,\ul}:\Modtilde(\Sigma_{g}^{n})\ra \SL_r(\F_q)$ and the generator of the central extension
	then acts by $\zeta_p^{- 6}I_r$; see \cite[§11]{gilmerMaslovIndexLagrangians2013}.
	Consider the commutative diagram with exact rows
	\begin{center}
		\begin{tikzcd}
			1 \arrow[r] & \mathbb{Z} \arrow[r] \arrow[d] & \Modtilde(\Sigma_{g}^{n}) \arrow[r] \arrow[d] & \Mod(\Sigma_{g}^{n}) \arrow[r] \arrow[d] & 1 \\
			1 \arrow[r] & \mathcal{Z} \arrow[r] & \SL_r(\mathbb{F}_q) \arrow[r] & \PSL_r(\mathbb{F}_q) \arrow[r] & 1.
		\end{tikzcd}
	\end{center}
	For any choice of set-theoretic section $\widetilde{s}:\Mod(\Sigma_{g}^{n})\ra \Modtilde(\Sigma_{g}^{n})$, the assignment
	$\widetilde{\varphi}(x,y)=\widetilde{s}(x)\widetilde{s}(y)\widetilde{s}(xy)^{-1}$ defines a group cohomology cocycle representing
	$\widetilde{\omega}\in H^2(\Mod(\Sigma_{g}^{n}),\Z)$. Similarly, for any set-theoretic section $s':\PSL_r(\mathbb{F}_q)\ra \SL_r(\mathbb{F}_q)$,
	$\varphi'(x,y)=s'(x)s'(y)s'(xy)^{-1}$ defines a cocycle representing
	the universal obstruction in $H^2(\PSL_r(\mathbb{F}_q), \mathcal{Z})$ to linearization.
	The obstruction $\omega'\in H^2(\Mod(\Sigma_{g}^{n}), \mathcal{Z})$ to linearizing $\rho_{J,g,\ul}$
	is then represented by the cocycle $(x,y)\mapsto \varphi'(\rho_{J,g,\ul}(x),\rho_{J,g,\ul}(y))$.
	Taking compatible choices of $\widetilde{s}$ and $s'$ and making use of the above commutative diagram,
	we see that $\varphi'(\rho_{J,g,\ul}(x),\rho_{J,g,\ul}(y))=\pi (\widetilde{\varphi}(x,y))$
	where $\pi$ is the map $1 \in \mathbb{Z} \mapsto \zeta_p^{-6} I_r \in \mathcal{Z}$.
	Hence, $\omega'$ is just the image of \( \widetilde{\omega} \) under the map
	\[ H^2(\Mod(\Sigma_{g}^{n}), \mathbb{Z}) \xrightarrow{\alpha} H^2(\Mod(\Sigma_{g}^{n}), \mathcal{Z})\]
	induced by $\pi$.
	We have a factorization \( \mathbb{Z} \twoheadrightarrow \mathbb{Z}/p\mathbb{Z} \subseteq \mathbb{Z}/r'\mathbb{Z} \cong \mathcal{Z} \) of $\pi$.
	As \( H_1(\Mod(\Sigma_{g}^{n}), \mathbb{Z}) = 0 \), by the Universal Coefficient Theorem
	\begin{equation}\label{eq:universalCoefficientsObstruction}
		H^2(\Mod(\Sigma_{g}^{n}), \mathbb{Z}) \otimes_{\mathbb{Z}} \mathbb{Z}/p\mathbb{Z} \xrightarrow{\sim} H^2(\Mod(\Sigma_{g}^{n}),%
		\mathbb{Z}/p\mathbb{Z}).
	\end{equation}
	Consider \( 0 \to \mathbb{Z}/p\mathbb{Z} \to \mathcal{Z} \to \mathbb{Z}/(r'/p)\mathbb{Z} \to 0 \).
	We have an exact sequence
	\begin{equation*}
		H^1(\Mod(\Sigma_{g}^{n}), \mathbb{Z}/(r'/p)\mathbb{Z}) \to H^2(\Mod(\Sigma_{g}^{n}), \mathbb{Z}/p\mathbb{Z}) \to H^2(\Mod(\Sigma_{g}^{n}),%
		\mathcal{Z}).
	\end{equation*}
	As $\Mod(\Sigma_{g}^{n})$ is perfect, $H^1(\Mod(\Sigma_{g}^{n}), \mathbb{Z}/(r'/p)\mathbb{Z})=0$.
	Hence, $H^2(\Mod(\Sigma_{g}^{n}), \mathbb{Z}/p\mathbb{Z})$ injects into $H^2(\Mod(\Sigma_{g}^{n}), \mathcal{Z})$.
	
	As $\widetilde{\omega}$ is primitive, by \Cref{eq:universalCoefficientsObstruction},
	it maps to a non-zero element of $H^2(\Mod(\Sigma_{g}^{n}), \mathbb{Z}/p\mathbb{Z})$.
	By the above injection, the image $\omega'$ of $\widetilde{\omega}$
	in $H^2(\Mod(\Sigma_{g}^{n}), \mathcal{Z})$ has order $p$.
	Now, \( \omega' \) is the image of \( \omega \) under
	\( H^2(G_{J,g,\ul}, \mathcal{Z}) \to H^2(\Mod(\Sigma_{g}^{n}), \mathcal{Z}) \),
	so \( \omega \) has order divisible by \( p \).
\end{proof}

Back to the proof of \Cref{prop:typeANotSemiTransverse}.
By Claim (2), $p$ divides $|H^2(G_{J,g,\ul},\mathcal{Z})|$. By Claim (1), this is only possible
when $G_{J,g,\ul}$ is of Lie type $A$ in characteristic prime to $p$, as desired.
\end{proof}

In the remainder of this section, we prove that $G_{J,g,\ul}$ cannot
be of Lie type in characteristic different from that of $\F_q=\Z[\zeta_p]/J$.
To do so, we bound below the rank of $G_{J,g,\ul}$ by considering a large abelian subgroup
generated by Dehn twists (\Cref{cor:lowerBoundOnK}), and bound it above by approximating $\dim V_{g,\ul}$
and applying bounds of Seitz-Zalesskii on minimal dimensions of representations of simple groups of Lie type
\cite{seitzMinimalDegreesProjective1993} (\Cref{prop:upperBoundonK}).
We then show the lower and upper bounds are incompatible (\Cref{prop:excludingUnequalCharacteristic}).

Consider \( A \subseteq G_{J,g,\ul} \) a subgroup generated by the images
\( t_1, \dots, t_{3g-3+n} \) of Dehn twists along non-separating {\scc}s forming a pair of pants decomposition.
By \Cref{rmk:vacuum}, we may and will assume that $\ul\in(\Lambda\setminus\{0\})^n$.

\begin{lemma}\label{lemma:independenceOfTwists}
	We have $A\simeq (\Z/p\Z)^{3g-3+n}$.
\end{lemma}
\begin{proof}
Let us denote by $x_1,\ldots,x_{3g-3+n}$ the images of $t_1,\ldots,t_{3g-3+n}$ in $G_{J,g,\ul}$, and by $\alpha_1,\ldots,\alpha_{3g-3+n}$
the associated curves. We know that each $x_i$ has order $p$, and that the $x_i$'s commute (since the Dehn twists $t_i$ commute).
Therefore, if $A$ were not isomorphic to $(\Z/p\Z)^{3g-3+n}$, up to relabeling, one would have a relation of the form
\begin{equation}
	\label{eq:relation}x_1=\underset{2\leq i \leq 3g-3+n}{\prod}x_i^{n_i}
\end{equation}
for some $n_i \in \Z/p\Z$. This would imply that $x_1$ acts as a scalar on each joint eigenspace of $x_2,\ldots,x_{3g-3+n}$.

Since $g\geq 1$, consider two colorings $c,c'$ of the edges of a banded trivalent graph dual to the pants decomposition
$\lbrace \alpha_1,\ldots,\alpha_{3g-3+n}\rbrace$ defined as follows. The color $c$ simply colors each interior edge by $i_0$.
The color $c'$ is obtained from $c$ by replacing the color along the edge dual to $\alpha_1$ by $i_0'=i_0+2$ if $p\equiv 3 \mod 4$
and $i_0'=i_0-2$ if $p\equiv 1 \mod 4$. We recall that $(i_0,i_0,j)$ is admissible for any $j\in \Lambda$ and $(i_0,i_0',j)$ is
admissible for any $j\in \Lambda \setminus \lbrace 0 \rbrace$; this implies that both $c$ and $c'$ are admissible colorings.

However, the corresponding basis vectors $\varphi_c,\varphi_c' \in V_{p,g,\ul}$ are in the same joint eigenspace of $x_2,\ldots,x_{3g-3+n}$,
but in distinct eigenspaces of $x_1$.

Therefore, there is no such relation as \Cref{eq:relation} and $A\simeq (\Z/p\Z)^{3g-3+n}$.
\end{proof}

Define $B$ to be the $p$-primary part of the pre-image of \( A \) in $\SL_k(\F_\ell)$ if $G_{J,g,\ul}=\PSL_k(\F_\ell)$
or in $\SU_k(\F_\ell)$ if $G_{J,g,\ul}=\PSU_k(\F_\ell)$. Note that thanks to \Cref{lemma:liftTorusToSL},
the pre-image is abelian.
By \Cref{lemma:independenceOfTwists}, we have \( \dim_{\mathbb{F}_p}(B/pB) \geq 3g-3+n \).

From now on, \emph{denote by \( d \) the order of \( \ell \pmod p \)}, i.e., \( p \mid \ell^d - 1 \) but \( p \nmid \ell^\delta - 1 \)
for \( 1 \leq \delta < d \).

\begin{lemma}\label{lemma:miminalDimOfRepOfAbGroup}
	If a \( p \)-primary abelian group \( B \) has a faithful representation in \( \SL_k(\mathbb{F}_\ell) \)
	\emph{with $p\nmid \ell$}, then
	\( \dim_{\mathbb{F}_p}(B/pB) \leq k-1 \) if \( d=1 \),
	and \( \dim_{\mathbb{F}_p}(B/pB) \leq \frac{k}{d} \) if \( d>1 \).
\end{lemma}

\begin{proof}
	The group $B$ is of the form $\prod_{i=1}^\alpha (\mathbb{Z}/p^{n_i}\mathbb{Z})$
	with \( \alpha = \dim_{\mathbb{F}_p}(B/pB) \).
	Consider \( B' \subseteq B \) isomorphic to \( (\mathbb{Z}/p\mathbb{Z})^\alpha \)
	and a decomposition into irreducible \( B' \)-modules:
	\[ \F_\ell^k = \bigoplus_{j \in J} V_j. \]

	\textbf{Case 1:} \( d > 1 \). Set \( V_j = V_j^1 \otimes \dots \otimes V_j^\alpha \),
	with the \( V_j^i \)'s simple \( (\mathbb{Z}/p\mathbb{Z}) \)-modules.
	Consider \( V \) any non-trivial \( (\mathbb{Z}/p\mathbb{Z}) \)-module.
	Choose \( g \in \mathbb{Z}/p\mathbb{Z} \) generator such that \( g \) acts non-trivially.
	Then \( \F_\ell[g] \cdot v \simeq \F_\ell[\zeta_p] \).
	But \( \dim_{\F_\ell} \F_\ell[\zeta_p] = d \).
	So, each \( V_j^i \) is either trivial or of dimension at least $d$.
	Now
	\begin{align*}
		k &= \sum_{j \in J} \prod_{i=1}^\alpha \dim(V_j^i) \\
		&\geq \sum_{j \in J} \sum_{i=1}^\alpha \begin{cases} 0 & \text{if } \dim(V_j^i) = 1 \\ \dim(V_j^i) & \text{if } \dim(V_j^i) > 1 \end{cases} \\
		&\geq \sum_{j \in J} \sum_{i=1}^\alpha d \cdot \mathbf{1}_{V_j^i \text{ is non-trivial}}.
	\end{align*}

	As the representation is faithful, for each $i$, there exists $j$ such that \( V_j^i \) is non-trivial.
	Hence, \( k \geq \alpha d \), as desired.

	\textbf{Case 2:} \( d = 1 \). Each \( V_j \) has dimension $1$ and is a character \( \chi \colon (\mathbb{Z}/p\mathbb{Z})^\alpha \to \F_\ell^\times \).
	So, up to change of basis, \( (\mathbb{Z}/p\mathbb{Z})^\alpha \subseteq (\F_\ell^\times)^k \cap \SL_k(\F_\ell) \subseteq \GL_k(\F_\ell) \).
	But \( (\F_\ell^\times)^k \cap \SL_k(\F_\ell) \simeq (\F_\ell^\times)^{k-1} \) and the \( p \)-torsion part of \( (\F_\ell^\times)^{k-1} \)
	is isomorphic to \( (\mathbb{Z}/p\mathbb{Z})^{k-1} \).
	So, \( (\mathbb{Z}/p\mathbb{Z})^\alpha \hookrightarrow (\mathbb{Z}/p\mathbb{Z})^{k-1} \).
	Hence, \( \alpha \leq k-1 \), as desired.
\end{proof}

Applying \Cref{lemma:miminalDimOfRepOfAbGroup} to our situation, we have:

\begin{corollary}\label{cor:lowerBoundOnK}
	With $k$ such that $G_{J,g,\ul}=\PSL_k(\F_\ell)$ or $G_{J,g,\ul}=\PSU_k(\F_\ell)$,
 	\( 3g-3+n \leq k-1 \) if \( d=1 \) and \( 3g-3+n \leq \frac{k}{d} \) if \( d>1 \).
\end{corollary}
\begin{proof}
	As $p\nmid \ell$ by \Cref{prop:typeANotSemiTransverse}, this follows from \Cref{lemma:miminalDimOfRepOfAbGroup}.
\end{proof}

Note that as \( g \geq 3 \), the corollary implies \( k \geq 3 \).

\begin{proposition}\label{prop:upperBoundonK}
	With $k\geq 3$ and $\ell$ such that $G_{J,g,\ul}=\PSL_k(\F_\ell)$ or $G_{J,g,\ul}=\PSU_k(\F_\ell)$,
	if \( \ell \wedge q = 1 \) (non-defining characteristic), then
	\[ \ell^{k-1} \leq 2 \left( p/2 \right)^{3g-3+n}. \]
\end{proposition}

\begin{proof}
	We realize the left-hand side of the inequality as a lower bound on $\dim V_{J, g, \ul}$
	and the right-hand side as an upper bound on $\dim V_{J, g, \ul}$.
	As \( \ell \wedge q = 1 \), we have the following lower bounds \cite{seitzMinimalDegreesProjective1993}
	of Seitz-Zalesskii on projective representations
	of $\PSL_k(\F_\ell)$ or $\PSU_k(\F_\ell)$ in non-defining characteristic:
	\[ \dim V_{J, g, \ul} \geq \begin{cases}
		\frac{\ell^k - 1}{\ell - 1} - k & \text{if } G_{J, g, \ul} = \PSL_k(\mathbb{F}_\ell); \\
		\frac{\ell^k - \ell}{\ell - 1} & \text{if } k \text{ odd}, G_{J, g, \ul} = \PSU_k(\mathbb{F}_\ell); \\
		\frac{\ell^k - 1}{\ell + 1} & \text{if } k \text{ even}, G_{J, g, \ul} = \PSU_k(\mathbb{F}_\ell).
	\end{cases} \]
	In all cases we have \( \dim V_{J, g, \ul} \geq \frac{1}{2} \ell^{k-1} \).
	Indeed:
	\begin{itemize}
		\item \( \frac{\ell^k - 1}{\ell - 1} = 1 + \ell + \dots + \ell^{k-1} \geq 1 + \ell(k-1) \geq 1 + 2k - 1 = 2k \); \\
		So, \( \frac{\ell^k - 1}{\ell - 1} - k \geq \frac{1}{2} \frac{\ell^k - 1}{\ell - 1} \geq \frac{1}{2} \ell^{k-1} \) as \( \ell^k \geq \ell \);
		\item \( \frac{\ell^k - \ell}{\ell - 1} = \ell \frac{\ell^{k-1} - 1}{\ell - 1} \geq \ell \frac{\ell^{k-1}}{\ell}\geq \frac{1}{2} \ell^{k-1} \)
		as \( k \geq 2 \) (Note: \( a \geq b>0 \implies \frac{a+1}{b+1} \leq \frac{a}{b} \));
		\item \( \frac{\ell^k - 1}{\ell + 1} = \frac{\ell - 1}{\ell + 1} \frac{\ell^k - 1}{\ell - 1} \geq \frac{1}{2} \ell^{k-1}\).
	\end{itemize}

	Consider a pair of pants decomposition of $\Sigma_{g,n}$ with associated {\scc}s $\alpha_1,\dotsc,\alpha_{3g-3+n}$.
	The corresponding basis of $V_{g,\ul}$ is indexed by assignments
	$\{\alpha_1,\dotsc,\alpha_{3g-3+n}\}\ra\Lambda$ subject to admissibility conditions.
	In particular, as $|\Lambda|=\frac{p-1}{2}$, we have the upper bound
	\begin{equation*}
		\dim V_{J, g, \ul} \leq \left( \frac{p-1}{2} \right)^{3g-3+n}\leq \left( \frac{p}{2} \right)^{3g-3+n}.
	\end{equation*}
	Hence, \( \frac{1}{2} \ell^{k-1} \leq \left( \frac{p}{2} \right)^{3g-3+n} \), as desired.
\end{proof}

\begin{proposition}\label{prop:excludingUnequalCharacteristic}
	Let $p\geq 7$ be a prime and $J\subset\Z[\zeta_p]$ maximal ideal prime to $p$.
	Let $g\geq 3$, $n\geq 0$ and $\ul\in \Lambda^n$.
	Then $G_{J,g,\ul} \simeq \PSL_k(\F_\ell)$ or $G_{J,g,\ul}\simeq\PSU_k(\F_\ell)$ for some $k\geq 3$
	with $\F_\ell$ and $\F_q=\Z[\zeta_p]/J$ of the \emph{same characteristic}.
\end{proposition}

\begin{remark}
	This proposition does not give the value of $k$.
	The proof that under the same assumptions $k=\dim V_{J,g,\ul}$ is the main result of \Cref{sec:definingCharacteristic}.
\end{remark}

\begin{proof}[Proof of \Cref{prop:excludingUnequalCharacteristic}]
	Assume \(\ell \wedge q = 1\), then we have
	\begin{alignat}{2}
		\label{eq:inequality1} & p \mid \frac{\ell^d - 1}{\ell - 1}\text{ if }d>1\text{ and } p\mid \ell-1\text{ if }d=1
			& \quad & \text{(Definition of } d\text{)}, \\[6pt]
		\label{eq:inequality2} & 3g-3+n \leq \begin{cases} k-1 & \text{if } d=1 \\ \frac{k}{d} & \text{if } d>1 \end{cases}
			& \quad & \text{(\Cref{cor:lowerBoundOnK})}, \\[6pt]
		\label{eq:inequality3} & \ell^{k-1} \leq 2(p/2)^{3g-3+n}
			& \quad & \text{(\Cref{prop:upperBoundonK})}.
	\end{alignat}

	\textbf{Case 1:} \(d=1\). By \Cref{eq:inequality2,eq:inequality3},
	\[ \ell^{k-1} \leq 2(p/2)^{3g-3+n} \leq 2(p/2)^{k-1},\text{ and hence } \ell \leq 2^{\frac{1}{k-1}}\frac{p}{2}. \]
	Applying \Cref{eq:inequality1}, this implies $p \leq 2^{\frac{1}{k-1}} \frac{p}{2}$ and thus $2 \leq 2^{\frac{1}{k-1}}$.
	But \(k-1 \geq 3g-3+n\geq 6\) as $g\geq 3$ and hence \(2 \leq 2^{\frac{1}{6}}\), which is a contradiction.

	\textbf{Case 2:} \(d>1\). By \Cref{eq:inequality2,eq:inequality3},
	\[ \ell^{k-1} \leq 2(p/2)^{3g-3+n} \leq 2(p/2)^{\frac{k}{d}},\text{ and hence } \ell^{d-\frac{d}{k}}\leq 2^{\frac{d}{k}}\frac{p}{2}.\]
	So, $\ell^d \leq (2\ell)^{\frac{d}{k}} \frac{p}{2}$.
	Since \(\frac{k}{d} \geq 3g-3+n \geq 6\), we have \(\ell^d \leq (2\ell)^{\frac{1}{6}} \frac{p}{2}\).

	But by \Cref{eq:inequality1}, \(p \leq \frac{\ell^d - 1}{\ell - 1} \leq \frac{\ell}{\ell - 1} \ell^{d-1} \leq 2 \ell^{d-1}\).
	So:
	\[ \ell^d \leq (2\ell)^{\frac{1}{6}} \ell^{d-1} = \ell^d \left( \frac{2^{\frac{1}{6}}}{\ell^{\frac{5}{6}}} \right) \]
	Thus $\ell^5 \leq 2$. However $\ell \geq 2$, contradiction.
\end{proof}

\subsection{Defining characteristic: \texorpdfstring{$\bm{G_{J,g,\ul}}$}{the image} is \texorpdfstring{$\bm{\PSL(V_{J,g,\ul})}$}{PSL(V)} or
\texorpdfstring{$\bm{\PSU(V_{J,g,\ul})}$}{PSU(V)}}
\label{sec:definingCharacteristic}

In this section, we work under the assumption that $G_{J,g,\ul}$ is a simple finite group of Lie type in characteristic
\emph{equal} to that of $\F_q=\Z[\zeta_p]/J$. Under this assumption, for $p\geq 7$ with $g\geq 3$
or $g=2$ and $V_{0,\ul}\neq 0$,
we prove that $G_{J,g,\ul}$ is $\PSL(V_{J,g,\ul})$ or $\PSU(V_{J,g,\ul})$ (\Cref{prop:definingCharacteristicImage}).

We first show that $G_{J,g,\ul}$ is $\PSL(W)$ or $\PSU(W)$ for some $W$ with $\dim W\geq 3$
such that $V_{J,g,\ul}$ is an exterior or truncated symmetric power of $W$ (\Cref{prop:isTruncatedSymmetricOrExteriorPower}).
We then rule out non-trivial exterior and truncated symmetric powers (\Cref{prop:notTruncatedSymmetricPower}).

The proofs of this section are adaptations in positive characteristic of the proofs of \Cref{sec:multiplicityFreeness,sec:notSymmetricOrExteriorPower}:
the result is also deduced from weight-multiplicity-freeness, with a classification result for such representations
in positive characteristic due to Zalesskii-Suprunenko and Seitz (\Cref{thm:SeitzMultiplicityFree})
replacing that of Howe in characteristic $0$ (\Cref{thm:HoweMultiplicityFree}).

Note that we work here with slightly weaker assumptions than the conclusion of \Cref{prop:excludingUnequalCharacteristic}.
Namely we do not assume that $G_{J,g,\ul}$ is type A; hence some of our results apply in genus $2$ as well.
This extra generality is included for possible future use.

\begin{proposition}\label{prop:noTits}
	For $p\geq 7$ and $g\geq 2$, $G_{J,g,\ul}$ is not isomorphic to ${}^2\mathrm{F}_4(2)'$ (Tits group).
\end{proposition}
\begin{proof}
	As $g\geq 2$ there exist $2$ distinct Dehn twists along disjoint non-separating simple closed curves
	in $\PMod(\Sigma_{g,n})$ and these will generate a subgroup of order $p^2$ in $G_{J,g,\ul}$ (see \Cref{lemma:independenceOfTwists}).
	As $p\geq 7$ and ${}^2\mathrm{F}_4(2)'$ has order $2^{11} 3^3 5^2 13$ \cite[p. 74]{conwayAtlasFiniteGroups1985},
	$G_{J,g,\ul}$ is not isomorphic to ${}^2\mathrm{F}_4(2)'$.
\end{proof}

\begin{notation}
	Under the assumptions of this section,
	as $G_{J,g,\ul}$ is not ${}^2\mathrm{F}_4(2)'$ by \Cref{prop:noTits}, there exists a simple simply-connected algebraic group $\bfG$
	over $\overline{\F}_q$ and a Frobenius or Steinberg automorphism $F$ of $\bfG$ such that $G_{J,g,\ul}=\bfG^F/Z(\bfG)^F$
	(The Tits group is the only exception, see \cite[§3]{rouquierModularRepresentationsFinite2024}).
	
	We will denote $V_{J,g,\ul}\otimes_{\F_q}\oF_q$ by $\oV_{J,g,\ul}$.
	As $\oV_{J,g,\ul}$ is an irreducible projective representation of $G_{J,g,\ul}$,
	there exists a simple linear rational representation $\bfV$ of $\bfG$ such that $\oV_{J,g,\ul}$ is isomorphic to the restriction
	of $\bfV$ to $G_{J,g,\ul}$ (see \cite[§4]{rouquierModularRepresentationsFinite2024}).
\end{notation}

\begin{proposition}\label{prop:notSelfDualModJ}
	Let $g\geq 2$ and $p\geq 5$ prime, then $\oV_{J,g,\ul}$ is not self-dual as a projective $\PMod(\Sigma_{g,n})$-representation.
\end{proposition}
\begin{proof}
	The proof of \Cref{prop:notSelfDual} applies as is: fusion rules and eigenvalue decompositions for Dehn twists still apply modulo $J$.
\end{proof}

\begin{corollary}\label{cor:groupsWithNonSelfDualRepsModJ}
	Let $g\geq 2$ and $p\geq 7$, $n\geq 0$ and $\ul\in\Lambda^n$
	such that $G_{J,g,\ul}$ is a simple finite group of Lie type in the same characteristic as $\F_q=\Z[\zeta_p]/J$.
	Then $\bfG$ is one of:
	\begin{equation*}
		\SL_k,\; k\geq 3,\;\;\SO_{4k+2},\; k\geq 2,\text{ or }\mathrm{E}_6.
	\end{equation*}
\end{corollary}
\begin{proof}
	The proof is essentially the same as that of \Cref{cor:groupsWithNonSelfDualReps}.
	By classification, the algebraic group $\bfG$ is one of
	$\SL_k$, $k\geq 2$, $\SO_{k}$, $k\geq 5$, $\mathrm{Sp}_{2k}$, $k\geq 2$, $\mathrm{E}_6$, $\mathrm{E}_7$,
	$\mathrm{E}_8$, $\mathrm{F}_4$ or $\mathrm{G}_2$. Now, for all of these but the ones listed in the statement of the corollary, all representations
	are self-dual. Indeed, given a maximal torus of $\bfG$, the dual of the finite-dimensional irreducible representation $\bfV_\nu$
	with highest weight $\nu$
	is $\bfV_{-w_0(\nu)}$ with $w_0$ the longest element in the Weyl group \cite[§2.2, Example 2]{humphreysModularRepresentationsFinite2005},
	and for all but the examples in the statement, $w_0$ acts by $-\id$ \cite[VI, §4-13, (XI)]{bourbakiSystemesRacines2007}.
	This concludes the proof, since if $\bfV$ were self-dual as a $\bfG$-representation,
	$\oV_{J,g,\ul}$ would also be self-dual as a $\PMod(\Sigma_{g,n})$-representation,
	contradicting \Cref{prop:notSelfDual}.
\end{proof}

A finite-dimensional representation $\bfV$ of a simply-connected simple algebraic group $\bfG$
is \emph{weight-multiplicity-free} if given a maximal torus $\bfT\subset \bfG$,
each character of $\bfT$ has multiplicity at most one in $V$, i.e. for any $\chi:\bfT\ra \mathbb{G}_m$,
the dimension of a subspace of $V$ on which $\bfT$ acts by $\chi$ is at most $1$.

\begin{proposition}\label{prop:weightMultiplicityFreeModJ}
	Let $p\geq 7$, $g\geq 2$, $n\geq 0$ and $\ul\in\Lambda^n$ such that $G_{J,g,\ul}$ is a simple finite group of Lie type
	of same characteristic as $\F_q=\Z[\zeta_p]/J$.
	Then $\bfV$ is a \emph{weight-multiplicity-free} representation of $\bfG$.
\end{proposition}
\begin{proof}
	The proof is essentially the same as that of \Cref{prop:weightMultiplicityFree}.
	As in that proposition, consider the finite abelian $p$-torsion subgroup $A\subset G_{J,g,\ul}$ generated
	by the images $t_{\alpha_1},\dotsc,t_{\alpha_{3g-3+n}}$ of Dehn twists along a collection of non-separating {\scc}s
	forming a pair of pants decomposition of $\Sigma_{g}^{n}$. Then $A$ is also a finite subgroup of $\bfG/Z(\bfG)$.
	
	As $p$ is prime, the decomposition of $\oV_{g,\ul}$ into characters for $A$
	is given by the TQFT basis corresponding to the pair of pants
	decomposition. In particular, every character of $A$ has multiplicity at most $1$ in $\oV_{g,\ul}$, and hence also in $\bfV$.
	Thus, if we knew that $A$ was included in a maximal torus
	of $\bfG/Z(\bfG)$, we could conclude that $\bfV$ is weight-multiplicity-free.
	Thanks to \Cref{cor:groupsWithNonSelfDualReps}, we are reduced to the following three cases.
	
	\textbf{Case 1:} The group $\bfG$ is $\SO(W)$ for $\dim W=4k+2$, $k\geq 2$. Then $\bfG/Z(\bfG)$ is $\PSO(W)$.
	Consider the decomposition $W=\bigoplus_\chi V_\chi$ according to projective characters
	$\chi\in \mathrm{Hom}(A,\overline{\F}_q^\times)/\overline{\F}_q^\times$.
	Each $V_\chi$ is isotropic for $\chi\neq 1$ and $V_\chi$ is orthogonal to $V_{\chi'}$ whenever $\chi'\neq \chi^{-1}$.
	As $|A|$ is odd, if $V_\chi\neq 0$,
	either $\chi\neq \chi^{-1}$ or $\chi$ is the trivial character $\chi_1$.
	So, we have an orthogonal decomposition
	$W=\left(\bigoplus^{\bot}_{\{\chi,\chi^{-1}\}\neq\{\chi_1\}} V_\chi\oplus V_{\chi^{-1}}\right)\oplus^{\bot} V_{\chi_1}$.
	As the form is non-degenerate, we must have $\dim V_\chi=\dim V_{\chi^{-1}}$ for each $\chi$.
	On each $V_\chi\oplus V_{\chi^{-1}}$, $A$ acts by a torus. Hence, $A$ is included in a maximal torus of $\PSO(W)$, as desired.

	\textbf{Case 2:} The group $\bfG$ is $\mathrm{E}_6$. Then $\bfG^{\mathrm{ad}}:=\bfG/Z(\bfG)$ is $\mathrm{E}_6/Z(\mathrm{E}_6)$.
	As $A$ is finite, it is contained in some $\bfG^{\mathrm{ad}}(\F_{q'})$
	for some finite extension $\F_{q'}$ of $\F_q$. Now, as $\bfG^{\mathrm{ad}}$ is connected reductive,
	any $p$-Sylow of $\bfG^{\mathrm{ad}}(\F_{q'})$ is contained in $N_{\bfG^{\mathrm{ad}}}(\bfT)(\F_{q'})$
	for $N_{\bfG^{\mathrm{ad}}}(\bfT)$ the normalizer of some maximal torus $\bfT$ \cite{cabanes1994unicite}
	(see also \cite[§22, Exercise 6]{cabanesRepresentationTheoryFinite2004}).
	In particular, $A\subset N_{\bfG^{\mathrm{ad}}}(\bfT)$ for some maximal torus $\bfT$.
	Now, $N_{\mathrm{E}_6/Z(\mathrm{E}_6)}(\bfT)/\bfT$ is the Weyl group of $\mathrm{E}_6$,
	which has order $2^73^45$ \cite[VI, §12, (IX)]{bourbakiSystemesRacines2007}. As $A$ is $p$-torsion and $p$ is prime and at least $7$,
	$A$ is included in $\bfT$, as desired.

	\textbf{Case 3:} The group $\bfG$ is $\SL(W)$ with $d=\dim W\geq 3$. Then $\bfG/Z(\bfG)$ is $\PSL(W)$.
	By linear algebra, any finite abelian subgroup of $\SL(W)$ is conjugate in $\SL(W)$ to a group of diagonal matrices.
	Hence, we need only show that the pre-image of $A\subset\PSL(W)$ in $\SL(W)$ is abelian.
	Applying \Cref{lemma:liftTorusToSL} with $H=\PSL(W)$ and $\widetilde{H}=\SL(W)$ concludes.
\end{proof}

We now make use of a classification result for weight-multiplicity-free representations due to Zalesskii-Suprunenko,
improving a result of Seitz (\Cref{thm:SeitzMultiplicityFree} below).

Truncated symmetric powers $TS^m(W)$ of a vector space $W$ over a field of characteristic $q'>0$ are defined as follows.
Consider the symmetric algebra $S(W)$ of $W$ and its $\N$-grading.
As we are in positive characteristic, the graded ideal generated by the $q'$-th
powers $w^{q'}\in S(W)$ of elements $w$ of $W$ is stable under the action of $\GL(W)$. The quotient by this ideal is
the truncated symmetric algebra $TS(W)$, and $TS^m(W)$ is the subspace of elements with grading $m$.
Note that $TS^m(W)\neq 0$ exactly when $0\leq m\leq (q'-1)\dim W$.
See \cite{zalesskiiReducedSymmetricPowers1990} for more on truncated symmetric powers.

Given a rational representation $\bfV$ of an algebraic group $\bfG$ over $\overline{\F}_{q'}$, $q'$ prime,
and an integer $u\geq 1$, one can twist $\bfV$ by the $u$-th power of the absolute Frobenius $F:x\mapsto x^{q'}$
to obtain another representation $\bfV^{[u]}$. For $\bfG=\GL_n$, this amounts to
$\bfV^{[u]}=\bfV$ as a vector space and making $P\in \GL_n(\overline{\F}_{q'})$
act on $\bfV^{[u]}$ by the action of $F^u(P)$ on $\bfV$,
where $F^u(P)$ is the matrix whose entries are the $q'^u$-th powers of the entries of $P$.

\begin{theorem}[{\cite[Proposition 6.1]{MR934193}, see also \cite[Theorem 6.1]{seitzMaximalSubgroupsClassical1987}}]\label{thm:SeitzMultiplicityFree}
	All non-trivial weight-multiplicity-free representations of $\SL(W)$, $\dim W\geq 3$
	are of the form $\bfV_1^{[u_1]}\otimes\dotsb\otimes\bfV_u^{[u_s]}$ for $s\geq 1$, $0\leq u_1<u_2<\dotsb<u_s$ and with each $\bfV_i$
	either a truncated symmetric power $TS^m(W)$, $1\leq m\leq (\charac \F_q-1)\dim W-1$
	or an exterior power $\Lambda^m(W)$, $1\leq m\leq \dim W-1$. For $\SO_{4k+2}$, $k\geq 2$,
	these are of the same form with each $\bfV_i$ the regular representation or one of the two half-spin representations of dimension $2^{2k}$.
	For $\mathrm{E}_6$, these are of the same form with each $\bfV_i$ one of the two $27$-dimensional representations
	(of highest weights $\omega_1$ or $\omega_6$ in standard notation).
\end{theorem}

The half-spin representations of $\SO(W)$ and the representations $V_{\omega_1}$ and $V_{\omega_6}$ of $\mathrm{E}_6$
have the same dimensions in positive characteristic and in characteristic $0$ ($2^{2k}$ and $27$ respectively) because their highest
weights are minuscule; see \cite[Proposition II.2.2b and Introduction of §II.6]{jantzenRepresentationsAlgebraicGroups2007}.

Our formulation of \Cref{thm:SeitzMultiplicityFree} differs slightly from that of \cite[Proposition 6.1]{MR934193}.
In that reference, only the weights of the representations are given.
To get the statement above from the weights,
see Steinberg's Tensor Product Theorem \cite[§2.7]{humphreysModularRepresentationsFinite2005}
and the description of highest weights of truncated symmetric powers \cite{zalesskiiReducedSymmetricPowers1990}.

Combining the classification (\Cref{thm:SeitzMultiplicityFree}),
multiplicity-freeness (\Cref{prop:weightMultiplicityFreeModJ}),
tensor-indecomposability (\Cref{thm:tensor-indec}) and a fact on dimensions
(\Cref{lemma:dimensionDivisibleByP}), we get the following.

\begin{proposition}\label{prop:isTruncatedSymmetricOrExteriorPower}
	Let $p\geq 7$, $g\geq 2$, $n\geq 0$ and $\ul\in\Lambda^n$ such that $g\geq 3$ or $g=2$ and $V_{0,\ul}\neq 0$.
	If $G_{J,g,\ul}$ is a finite group of Lie type in the same characteristic as $\F_q=\Z[\zeta_p]/J$, then $\bfG$
	is isomorphic to some $\SL(W)$ with $\dim W\geq 3$, and $\bfV$ is a truncated symmetric power
	$TS^m(W)$, $1\leq m\leq (\charac \F_q-1)\dim W-1$ or an exterior power $\Lambda^m(W)$, $1\leq m\leq \dim W-1$.
\end{proposition}
\begin{proof}
	By \Cref{cor:groupsWithNonSelfDualRepsModJ}, we know that $\bfG$ is of the form $\SL(W)$, $\dim W\geq 3$, $\SO(W)$, $\dim(W)=4k+2, k\geq 2$,
	or $\mathrm{E}_6$. We now apply \Cref{thm:SeitzMultiplicityFree} and \Cref{prop:weightMultiplicityFreeModJ}.
	In each case, $\bfV$ is some tensor product $\bfV_1^{[u_1]}\otimes\dotsb\otimes\bfV_u^{[u_s]}$.
	By \Cref{thm:tensor-indec}, $\oV_{J,g,\ul}$ is tensor-indecomposable as a $\PMod(\Sigma_{g,n})$-representation, hence so is $\bfV$ as a $\bfG$-module.
	Thus $s=1$ and $\bfV=\bfV_1^{[u_1]}$. Applying an inverse Frobenius twist to $\bfG$ if necessary,
	we may assume that $u_1=1$.
	
	If $\bfG$ is $\SO(W)$,
	$\bfV$ is $W$ or a half-spin representation of dimension $2^{2k}$, and if $\bfG$ is $\mathrm{E}_6$, $\bfV$ has dimension $27$.
	None of these cases are possible: $W$ is self-dual but $\bfV$ is not by \Cref{prop:notSelfDualModJ},
	and by \Cref{lemma:dimensionDivisibleByP}, $\dim\bfV=\dim V_{g,\ul}$ is divisible by $p$, which is prime and at least $7$.
	Hence, $\bfG$ is of the form $\SL(W)$ and $\bfV$ is a truncated symmetric power or an exterior power.
\end{proof}

We now prove that $\bfV$ cannot be a non-trivial exterior or truncated symmetric power
by showing that this would not be compatible with fusion rules along a non-separating curve.
To start with, we show the following replacements in positive characteristic for \Cref{prop:CauchyDecomp}.

\begin{lemma}\label{lemma:DecompTSofTensor} Let $U_1$ and $U_2$ be two finite-dimensional vector spaces over an algebraically closed field
	of odd positive characteristic $q'$.
	If $\dim U_1\geq 2$ and $\dim U_2\geq 2$,
	then for any $2\leq m\leq\dim U_1\boxtimes U_2-2$,
	$\Lambda^m(U_1\boxtimes U_2)$ is a reducible representation of $\GL(U_1)\times \GL(U_2)$
	and for any $2\leq m\leq (q'-1)\dim U_1\boxtimes U_2-2$,
	$TS^m(U_1\boxtimes U_2)$ is a reducible representation of $\GL(U_1)\times \GL(U_2)$.
\end{lemma}
\begin{proof}
	We write a common proof for both cases, denoting $\Lambda$ or $TS$ by $G$.
	Let $d_2=\dim U_2$ and set
	$d_1'=\dim U_1$ if $G=\Lambda$ and $d_1'=(q'-1)\dim U_1$ if $G=TS$.

	Assume by contradiction that for some $2\leq m\leq d_1'd_2-2$, $G^m(U_1\boxtimes U_2)$ is irreducible
	as a representation of $\SL(U_1)\times \SL(U_2)$. Then it is of the form
	$V_1\boxtimes V_2$ with $V_i$ an irreducible $\SL(U_i)$-representation for $i=1,2$.
	In particular $G^m(U_1\boxtimes U_2)$ would be an isotypic $\SL(U_1)$-representation, i.e. semisimple with a unique irreducible composition factor.
	Then, as an $\SL(U_1)$-representation, $G^m(U_1\boxtimes U_2)$ is just
	$G^m(U_1^{\oplus d_2})$. As for any vector spaces $W$ and $W'$,
	$G^m(W\oplus W')=\bigoplus_{a+b=m}G^a(W)\otimes G^b(W')$, we have
	\begin{equation}\label{eq:decompExteriormodJ}
		G^m(U_1^{\oplus d_2})=\bigoplus_{\substack{f:\{1,\dotsc,d_2\}\ra \{0,\dotsc,d_1'\}\\ \sum_i f(i)=m}}\bigotimes_i G^{f(i)}(U_1).
	\end{equation}
	Set $m=sd_1'+r$ with $s\in\Z$ and $0\le r<d_1'$.
	As $G^m(U_1\boxtimes U_2)$ is dual to $G^{d_1'd_2-m}(U_1\boxtimes U_2)$ and dualizing preserves irreducibility,
	we may assume without loss of generality that $m\leq d_1'd_2/2$. Then either $s\leq d_2-2$ or $d_2=2$, $s=1$ and $r=0$.
	We consider the following cases.
	
	\textbf{(1)} If $r\geq 2$, then $s\leq d_2-2$ and there exists $f$ as in \Cref{eq:decompExteriormodJ} such that
	$f(1)=r$, $f(2)=0$ and $f(i)=0$ or $f(i)=d_1'$ for each $i\geq 3$. The corresponding summand is $G^r(U_1)$,
	which is irreducible and non-trivial. So, $V_1\simeq G^r(U_1)$.
	Now, consider $\tilde{f}$ which coincides with $f$ for $i\geq 3$ and such that $\tilde{f}(1)=r-1$, $\tilde{f}(2)=1$.
	The corresponding summand is $G^{r-1}(U_1)\otimes U_1$. If $G=\Lambda$, as $r\geq 2$,
	$\omega_1+\omega_{r-1}$ is the highest weight of $\Lambda^{r-1}(U_1)\otimes U_1$
	which is different from the highest weight $\omega_r$ of $\Lambda^r(U_1)$, so $\Lambda^m(U_1\boxtimes U_2)$ cannot be isotypic for $\SL(U_1)$,
	contradiction. If $G=TS$ and $r\geq q'$, then again there is a highest weight mismatch.
	Indeed, $TS^r(U_1)$ has highest weight $(q'-1-r')\omega_{a}+r'\omega_{a+1}$ with $r=a(q'-1)+r'$, $0\le r'<q'-1$.
	If $r'>0$, the highest weight of $TS^{r-1}(U_1)\otimes U_1$ has coordinate $<r'$ in $\omega_{a+1}$.
	If $r'=0$, then $a\geq 2$ and the highest weight of $TS^{r-1}(U_1)\otimes U_1$ is $\omega_1+(q'-1)\omega_{a}$,
	which is different from $(q'-1)\omega_{a}$.
	If $G=TS$ and $r<q'$, then $TS^r(U_1)=S^r(U_1)$
	and $TS^{r-1}(U_1)\otimes U_1=S^{r-1}(U_1)\otimes U_1$. Now, $S^r(U_1)$ appears in $S^{r-1}(U_1)\otimes U_1$ as
	a composition factor exactly once. As $S^{r-1}(U_1)\otimes U_1$ is strictly larger than $S^r(U_1)$, it must have a different factor,
	so $TS^m(U_1\boxtimes U_2)$ cannot be isotypic for $\SL(U_1)$, contradiction.

	\textbf{(2)} If $r=1$, then $0<s\leq d_2-2$ and there exists $f$ as in \Cref{eq:decompExteriormodJ} such that
	$f(1)=1$, $f(2)=d_1'$ and $f(i)=0$ or $f(i)=d_1'$ for each $i\geq 3$. The corresponding summand in the equation is $U_1$,
	so $V_1\simeq U_1$.
	If $d_1'\geq 3$, consider $\tilde{f}$ with $\tilde{f}(1)=2$, $\tilde{f}(2)=d_1'-1$ and $\tilde{f}(i)=f(i)$ for $i\geq 3$.
	The corresponding summand $G^2(U_1)\otimes U_1^*$ has a highest weight, $\omega_2+\omega_{n-1}$ for $G=\Lambda$
	and $2\omega_1+\omega_{n-1}$ for $G=TS$, which differs from the highest weight $\omega_1$ of $U_1$, contradiction.
	If $d_1'=2$, then $G=\Lambda$. If moreover $d_2=2$, then $\Lambda^m(U_1\boxtimes U_2)$ must be
	$\Lambda^2(U_1\boxtimes U_2)=U_1^{\oplus 2}\oplus\Lambda^2(U_1)$, which is not an isotypic $\SL(U_1)$-representation, contradiction.
	If $d_2\geq 3$, as $m\leq d_1'd_2/2=d_2$, there exists $i_0\geq 3$ such that $f(i_0)=0$.
	Consider $\tilde{f}$ such that $\tilde{f}(1)=1$, $\tilde{f}(2)=1$, $\tilde{f}(i_0)=1$ and $\tilde{f}(i)=f(i)$
	for other values of $i$. The corresponding summand is $U_1^{\otimes 3}$, with highest weight $3\omega_1$ different from that of $V_1=U_1$,
	contradiction.

	\textbf{(3)} If $r=0$, then there exists $f$ as in \Cref{eq:decompExteriormodJ} such that
	$f(i)=0$ or $f(i)=d_1'$ for each $i$. The corresponding summand is the trivial representation. So, $V_1$ is trivial.
	As $m\geq 2$, there exists $i_0$ such that $f(i_0)=d_1'$. Now, consider $\tilde{f}$ which matches $f$ outside $\{1,i_0\}$
	and such that $\tilde{f}(1)=1$, $\tilde{f}(i_0)=d_1'-1$. The corresponding summand is $U_1\otimes U_1^*$, which is non-trivial,
	contradiction.
\end{proof}

\begin{proposition}\label{prop:notTruncatedSymmetricPower}
	Using the notation of \Cref{prop:isTruncatedSymmetricOrExteriorPower}, $\bfV$ is not one of $\Lambda^m(W)$ for $2\leq m\leq \dim W-2$
	or $TS^m(W)$ for $2\leq m\leq (\charac\F_q-1)\dim W-2$.
\end{proposition}
\begin{proof}
	Denote by $q'$ the prime divisor of $q$.
	Replacing \Cref{prop:CauchyDecomp} by \Cref{lemma:DecompTSofTensor}, the proof of \Cref{prop:notExteriorPower}
	applies as is to show that $\bfV$ is not $\Lambda^m(W)$ for some $2\leq m\leq \dim W-2$.

	The situation is similar for truncated symmetric powers. Let us make some comments.
	As in \Cref{prop:notSymmetricPower}, considering $\alpha\subset \Sigma_{g,n}$ and corresponding decompositions of $\oV_{J,g,\ul}$ and $W$ into
	irreducible representations:\footnote{As in \Cref{prop:notSymmetricPower}, we need to apply \Cref{lem:tensorProduct} to obtain the second decomposition.}
	\begin{equation}\label{eq:decompositionsForTruncatedCase}
		\oV_{J,g,\ul}=\bigoplus_{\mu\in\Lambda}\oV_{J,1,\mu}\boxtimes \oV_{J,g-1,(\ul,\mu)},\;W=\bigoplus_{i\in I}W_i^1\boxtimes W_i^2.
	\end{equation}
	The latter induces the decomposition
	\begin{equation}\label{eq:truncatedSymmetricPowerDecomposition}
		\oV_{J,g,\ul}=\bigoplus_{\substack{f:I\ra \N\\\sum_i f(i)=m\\\forall i,\: f(i)\leq (q'-1)\dim W_i^1\boxtimes W_i^2}}%
		\bigotimes_{i\in I}TS^{f(i)}\left(W_i^1\boxtimes W_i^2\right).
	\end{equation}
	Now, truncated symmetric powers behave similarly to exterior powers:
	$TS^m(U)$ is non-zero exactly for $m$ in an interval, namely $0\leq m \leq (q'-1)\dim U$, has dimension $1$ exactly
	when $\dim U=1$ or when $m$ is at the ends
	of that interval, namely $m=0$ and $m=(q'-1)\dim U$, and duals are given by flipping the interval, i.e.
	$TS^m(U^*)\simeq TS^m(U)^*\simeq TS^{(q'-1)\dim U-m}(U)$.
	Thanks to this observation, we see that the rest of the proof of \Cref{prop:notExteriorPower}
	applies also to truncated symmetric powers, replacing
	\Cref{prop:CauchyDecomp} by \Cref{lemma:DecompTSofTensor} and each occurrence of ``$\dim W_i^1\boxtimes W_i^2$''
	by ``$(q'-1)\dim W_i^1\boxtimes W_i^2$''.
\end{proof}

Combining \Cref{prop:isTruncatedSymmetricOrExteriorPower,prop:notTruncatedSymmetricPower}, we arrive at the aim of this section.

\begin{proposition}\label{prop:definingCharacteristicImage}
	Let $p\geq 7$, $g\geq 2$, $n\geq 0$ and $\ul\in\Lambda^n$ such that $g\geq 3$ or $g=2$ and $V_{0,\ul}\neq 0$.
	If $G_{J,g,\ul}$ is a finite group of Lie type in the same characteristic as $\F_q=\Z[\zeta_p]/J$,
	then $G_{J,g,\ul}\simeq\PSL(V_{J,g,\ul})$ if the degree of $\F_q$ over its prime field is odd, and $G_{J,g,\ul}\simeq\PSU(V_{J,g,\ul})$ otherwise.
\end{proposition}
\begin{proof}
	By \Cref{prop:isTruncatedSymmetricOrExteriorPower}, $\bfG$ is $\SL(W)$ for some $W$ with $\dim W\geq 3$
	and $\bfV$ is an exterior or truncated symmetric power of $W$. By \Cref{prop:notTruncatedSymmetricPower},
	$\bfV$ is either $W$ or $W^*$. So, $\bfG$ is isomorphic to $\SL(\bfV)$.
	Now, $G_{J,g,\ul}=\SL(\bfV)^F/Z(\SL(\bfV))^F$ must be either $\PSL(V')$ or $\PSU(V')$ for a vector space $V'$
	over some finite field $\F_{q'}\subset \oF_q$ such that $V'\otimes_{\F_{q'}}\oF_q=\bfV$.
	Now, $G_{J,g,\ul}\subseteq \PSL(V_{J,g,\ul})$ and $V_{J,g,\ul}\otimes_{\F_q}\oF_q=\bfV$.
	So, $\F_{q'}\subseteq \F_q$ and $V_{J,g,\ul}\simeq V'\otimes_{\F_{q'}}\F_q$.
	By \Cref{lemma:fieldOfDefinition} below, $\F_{q'}=\F_q$.

	If the degree of $\F_q$ over its prime field is odd, then $\PSU(V_{J,g,\ul})$ is not defined, so $G_{J,g,\ul}\simeq\PSL(V_{J,g,\ul})$.
	Now, assume that $\F_q=\Z[\zeta_p]/J$ has even degree over its prime field.
	As explained in \Cref{sec:prelimDedekindDom},
	the TQFT sesquilinear form on $V_{g,\ul}$ then descends to a non-degenerate $\PMod(\Sigma_{g,n})$-invariant
	sesquilinear form on $V_{J,g,\ul}$.
	So, $G_{J,g,\ul}\subseteq \PSU(V_{J,g,\ul})$. As $G_{J,g,\ul}$ is either $\PSU(V_{J,g,\ul})$ or $\PSL(V_{J,g,\ul})$, $G_{J,g,\ul}= \PSU(V_{J,g,\ul})$.
\end{proof}

\begin{lemma}\label{lemma:fieldOfDefinition}
	Let $p\geq 5$ be a prime, $g\geq 1$, $n\geq 0$ and $\ul\in\Lambda^n$ such that $g\geq 2$ or $g=1$ and $\ul$ is permissive.
	Let $J\subset\Z[\zeta_p]$ be a maximal ideal prime to $p$ and set $\F_q=\Z[\zeta_p]/J$. Then $\rho_{J,g,\ul}$ cannot be defined over a strict
	sub-field $\F_{q'}\subset \F_q$.
\end{lemma}
\begin{proof}
	Assume that $\rho_{J,g,\ul}$ can be defined over $\F_{q'}\subseteq \F_q$. Let us show that $\F_{q'}= \F_q$.
	Consider the fusion decomposition
	\begin{equation}\label{eq:fusionFieldOfDefinition}
		V_{J,g,\ul}=\bigoplus_{\mu\in\Lambda} V_{J,1,\mu}\boxtimes V_{J,g-1,(\ul,\mu)}
	\end{equation}
	over $\F_q$.
	By the assumptions on $g$ and $\ul$, for each $\mu\in\Lambda$, the corresponding summand is non-zero; see \Cref{lemma:dimAtLeastOne,def:permissive}.
	\Cref{eq:fusionFieldOfDefinition} is a decomposition of $V_{J,g,\ul}$ into absolutely irreducible
	$\Mod(\Sigma_1^1)\times\PMod(\Sigma_{g-1,n}^1)$-modules.
	This decomposition is unique as the irreducible summands appear with no multiplicity.
	As $V_{J,g,\ul}$ can be defined over $\F_{q'}$, it is endowed with an action of $\Gal(\F_q/\F_{q'})$, and for any $g\in\Gal(\F_q/\F_{q'})$,
	the decomposition
	\begin{equation*}
		V_{J,g,\ul}=\bigoplus_{\mu\in\Lambda} g\cdot \left(V_{J,1,\mu}\boxtimes V_{J,g-1,(\ul,\mu)}\right)
	\end{equation*}
	is also a decomposition into absolutely irreducible summands.
	Hence, there exists a permutation $\sigma\in\mathfrak{S}(\Lambda)$ such that for each $\mu\in\Lambda$,
	$g\cdot \left(V_{J,1,\mu}\boxtimes V_{J,g-1,(\ul,\mu)}\right)$ and $V_{J,1,\sigma(\mu)}\boxtimes V_{J,g-1,(\ul,\sigma(\mu))}$ are isomorphic $\Mod(\Sigma_1^1)\times\PMod(\Sigma_{g-1,n}^1)$-modules.
	So, for any $\mu\in \Lambda$, $V_{1,\sigma(\mu)}$ and the Galois twist $g\cdot V_{1,\mu}$ are isomorphic.
	In particular $\dim V_{1,\sigma(\mu)}=\dim V_{1,\mu}$.
	By \Cref{lemma:dimTori}, this implies $\sigma(\mu)=\mu$. Hence, $\sigma=\id_\Lambda$.

	We conclude that the decomposition in \Cref{eq:fusionFieldOfDefinition} is stable under $\Gal(\F_q/\F_{q'})$ and thus defined over $\F_{q'}$.
	But this decomposition is also the eigenspace decomposition of a Dehn twist $T_\gamma\in\PMod(\Sigma_{g,n})$.
	So, the projective spectrum of the action of this Dehn twist must be defined over $\F_{q'}$. In particular, for any $u_1,u_2\in \oF_q$
	eigenvalues of the action of $T_\gamma$,
	$\frac{u_1}{u_2}\in \F_{q'}$. Now, the spectrum is $\{\zeta_p^{k^2}\mid k\in\lbrace 1,\ldots,\frac{p-1}{2}\rbrace\}$,
	so $\zeta_p^3=\frac{\zeta_p^{2^2}}{\zeta_p^{1^2}}\in\F_{q'}$,
	and hence also $\zeta_p\in \F_{q'}$ since $p\geq 5$.
	But $\F_q$ is generated by $\zeta_p$ over its prime field, so $\F_{q'}=\F_q$, as desired.
\end{proof}

We can now conclude the proof of \Cref{thm:surjectivity}:

\begin{proof}[Proof of \Cref{thm:surjectivity}]
	By \Cref{prop:excludingUnequalCharacteristic}, $G_{J,g,\ul}$
	is a finite group of Lie type in the same characteristic as $\F_q=\Z[\zeta_p]/J$.
	Then \Cref{prop:definingCharacteristicImage} implies the theorem.
\end{proof}

\subsection{Proof of surjectivity modulo non-prime ideals}

\label{sec:conclusionSurjectivity}

Let us now turn to the proof of \Cref{thm:surjectivityImproved}. We begin by proving the following weaker statement.

\begin{lemma}\label{lemma:surjectivityProduct}
	Let $p$, $g$, $\ul$ and $I\subset \Zrp$ be as in \Cref{thm:surjectivityImproved}.\footnote{In particular, we also allow $p=5$ here.
	We recall that surjectivity for $p=5$ and maximal ideals will be addressed in Section \ref{sec:case-p=5}.}
	Assume that $I$ is reduced, i.e. $I=I_1\dotsb I_k$
	with $I_1,\dotsc,I_k$ distinct prime ideals. Then
	\begin{equation*}
		\Modgn \lra \PSU(h_{g,\ul})(\Zrp/I)=\prod_{i=1}^{k} \PSU(h_{g,\ul})(\Zrp/I_i)
	\end{equation*}
	is surjective.
\end{lemma}

This lemma is the statement mentioned in \Cref{rmk:surjectivityProduct} and is weaker than the statement of
\Cref{thm:surjectivityImproved} for the same $I$ as it involves $\PSU(h_{g,\ul})$ rather than $\SU(h_{g,\ul})$.
Note that with $J_i$ such that $J_i\cap \Zrp=I_i$,
$\PSU(h_{g,\ul})(\Zrp/I_i)$ is $\PSU(V_{J_i,g,\ul})$ if $J_i=\overline{J_i}$ and $\PSL(V_{J_i,g,\ul})$ otherwise.

\begin{proof}
	Assume by contradiction that $\Modgn \lra \prod_{i=1}^{k} \PSU(h_{g,\ul})(\Zrp/I_i)$ is not surjective. We may assume that $k$ is minimal, i.e.
	that $\Modgn \ra \prod_{i=1}^{k-1} \PSU(h_{g,\ul})(\Zrp/I_i)$ is surjective.
	By \Cref{thm:surjectivity}, $k\geq 2$ and $\Modgn \ra \PSU(h_{g,\ul})(\Zrp/I_k)$ is surjective.
	Denote by $H$ the image of the map in \Cref{lemma:surjectivityProduct}. It is a subgroup of the product
	\begin{equation*}
		G_1\times G_2=\left(\prod_{i=1}^{k-1} \PSU(h_{g,\ul})(\Zrp/I_i)\right)\times \PSU(h_{g,\ul})(\Zrp/I_k)
	\end{equation*}
	and surjects onto both factors. By the diagonal Hall lemma (see \cite[Lemma 3.5]{Kup11}), we must have that
	$\PSU(h_{g,\ul})(\Zrp/I_i)\simeq \PSU(h_{g,\ul})(\Zrp/I_k)$ for some $1\leq i \leq k-1$
	and the image of the representation
	\begin{equation*}
		\Modgn \lra \PSU(h_{g,\ul})(\Zrp/I_i)\times \PSU(h_{g,\ul})(\Zrp/I_k)
	\end{equation*}
	factors through the diagonal. In particular, as we assumed $k$ minimal, $k=2$ and $i=1$.
	As $g\geq 3$ and $p\geq 5$, we have $\dim V_{g,\ul}\geq 5$. This guarantees that $\PSU(h_{g,\ul})(\Zrp/I_1)\simeq \PSU(h_{g,\ul})(\Zrp/I_2)$
	is not an
	exceptional isomorphism \cite[§3.5]{conwayAtlasFiniteGroups1985}.
	Hence, there is an isomorphism $\Zrp/I_1\simeq \Zrp/I_2$, i.e. $I_1$ and $I_2$ are in the same $\Gal(\Q(\zeta_p)/\Q)$-orbit.
	Choose lifts $J_1$ and $J_2$ of $I_1$ and $I_2$ to prime ideals of $\Zzp$. Then $\Zzp/J_1\simeq \Zzp/J_2$.
	Let us denote this common finite field by $\F_q$ and the group $\PSU(h_{g,\ul})(\Zrp/I_1)\simeq \PSU(h_{g,\ul})(\Zrp/I_2)$ by $G$.
	
	Now, the representations $\rho_{J_1,g,\ul}:\Modgn\ra G$ and $\rho_{J_2,g,\ul}:\Modgn\ra G$ are conjugate by an automorphism of $G$.
	With $d=\dim V_{g,\ul}$, $G$ is either $\PSL_d(\F_q)$ or $\PSU_d(\F_q)$. Hence, $\mathrm{Aut}(G)$ is generated by inner automorphisms,
	Galois action and, if $G=\PSL_d(\F_q)$, inverse-transposition \cite[§3.3]{conwayAtlasFiniteGroups1985}. So, $\rho_{J_1,g,\ul}$ and
	$\rho_{J_2,g,\ul}$ are either Galois conjugate or dual up to Galois conjugation. Changing $J_1$ to its complex conjugate $\overline{J_1}$
	if necessary,
	we may assume that $\rho_{J_1,g,\ul}$ and $\rho_{J_2,g,\ul}$ are Galois conjugate, say by $F^k$ where $F$ is the Frobenius of $\F_q$.
	
	Consider the following fusion decompositions along a {\scc} $\gamma$ separating a compact torus with one boundary component.
	\begin{equation}\label{eq:fusionSurjectivityProduct}
		V_{g,\ul}=\bigoplus_{\mu\in\Lambda} V_{1,\mu}\boxtimes V_{g-1,(\ul,\mu)}.
	\end{equation}
	All summands are non-zero by \Cref{cor:nonTriviality}.
	This is both the eigenspace decomposition for the action of the Dehn twist $T_\gamma$ and the multiplicity-free irreducible
	decomposition of $V_{g,\ul}$ as a $\Mod(\Sigma_1^1)\times\PMod(\Sigma_{g-1,n}^1)$-representation.
	Up to a scalar independent of $\mu$, the eigenvalue of $T_\gamma$ on the summand corresponding to $\mu\in\Lambda=\{0,2,\dotsc,p-3\}$
	is $\zeta_p^{(\mu+1)^2}$.
	The decomposition of \Cref{eq:fusionSurjectivityProduct} exists for both reductions $\rho_{J_1,g,\ul}$ and $\rho_{J_2,g,\ul}$.
	As they are conjugate by $F^k$,
	we must have
	\begin{equation}\label{eq:GaloisActionSurjectivityProduct}
		\forall \mu\in \Lambda,\; F^k(V_{J_1,1,\mu}\boxtimes V_{J_1,g-1,(\ul,\mu)})\simeq V_{J_2,1,\sigma(\mu)}\boxtimes V_{J_2,g-1,(\ul,\sigma(\mu))}
	\end{equation}
	for some permutation $\sigma:\Lambda\ra\Lambda$. For each $\mu$, we have $\dim V_{J_1,1,\mu}=\dim V_{J_2,1,\sigma(\mu)}$,
	and by \Cref{lemma:dimTori}, this implies that $\sigma$ is the identity. The isomorphism of \Cref{eq:GaloisActionSurjectivityProduct}
	is compatible with the action of $T_\gamma$.
	Hence, there exists $u\in \F_q^\times$ such that for each $\mu\in\{0,2,\dotsc,p-3\}$, $F^k(\xi_1^{(\mu+1)^2})=u\xi_2^{(\mu+1)^2}$
	where $\xi_i\in \F_q$ for $i=1,2$ denotes the image of $\zeta_p$ under $\Zzp\ra\Zzp/J_i\simeq \F_q$.
	
	Let $q'$ be the characteristic of $\F_q$. Setting $\mu=0$, we have $\xi_1^{q'^k}=u\xi_2$. Setting $\mu=1$, we have $\xi_1^{4q'^k}=u\xi_2^{4}$,
	hence we get $\xi_1^{3q'^k}=\xi_2^{3}$.
	As $p\geq 5$, $3$ is invertible modulo $p$ and hence $\xi_1^{q'^k}=\xi_2$.
	This means that the $2$ maps $\Z[\zeta_p]\ra\F_q$ with kernels $J_1$ and $J_2$ are conjugate by an automorphism of $\F_q$. So, $J_1=J_2$, contradiction.
\end{proof}

We will deduce \Cref{thm:surjectivityImproved} from \Cref{lemma:surjectivityProduct} using \Cref{lemma:Frattini} below.
The Frattini subgroup $\Phi(H)$ of a finite group $H$ is the intersection of the maximal subgroups of $H$.
It is always characteristic and in particular $\Phi(H)$ is trivial when $H$ is simple. It is not hard to check that
$\Phi(H_1\times H_2)=\Phi(H_1)\times\Phi(H_2)$ for any finite groups $H_1$, $H_2$ and that
for any subset $X$ of a finite group $H$, $\langle X\rangle=H$ if and only if $\langle X,\Phi(H)\rangle=H$.

We deduce the following lemma from \cite[Theorem 1.3]{vasiuSurjectivityCriteriaPadic2003}.
\begin{lemma}\label{lemma:Frattini}
	Let $I\subset \Zrp$ be a non-zero ideal prime to $p$ and $\sqrt{I}$ its radical (i.e. the largest square-free ideal dividing $I$).
	Then the Frattini subgroup of $\SU(h_{g,\ul})(\Zrp/I)$ is the kernel of
	\begin{equation*}
		\pi:\SU(h_{g,\ul})(\Zrp/I)\lra \PSU(h_{g,\ul})(\Zrp/\sqrt{I}).
	\end{equation*}
	In particular, the only subgroup of $\SU(h_{g,\ul})(\Zrp/I)$ which surjects onto $\PSU(h_{g,\ul})(\Zrp/\sqrt{I})$ is $\SU(h_{g,\ul})(\Zrp/I)$.
\end{lemma}
\begin{proof}
	In this proof, we will denote $\Zrp$ by $\Or$, $\SU(h_{g,\ul})$ by $\mathrm{G}$ and $\PSU(h_{g,\ul})$ by $\mathrm{PG}$.
	Decompose $I=I_1^{e_1}\dotsb I_k^{e_k}$ into a product of powers of distinct prime ideals $I_1,\dotsc,I_k$.
	Then $\Or/I= \Or/I_1^{e_1}\times\dotsb\times\Or/I_k^{e_k}$, $\sqrt{I}=I_1\dotsb I_k$ and $\Or/\sqrt{I}=\Or/I_1\times\dotsb\times\Or/I_k$.
	Thus $\mathrm{G}(\Or/I)$ is $\mathrm{G}(\Or/I_1^{e_1})\times\dotsb\times\mathrm{G}(\Or/I_k^{e_k})$.
	
	Let us first prove that $\Phi(\mathrm{G}(\Or/I_1^{e_1}))=\ker\pi_1$ with $\pi_1:\mathrm{G}(\Or/I_1^{e_1})\ra \mathrm{PG}(\Or/I_1)$.
	As $\mathrm{PG}(\Or/I_1)$ is simple, its Frattini subgroup is trivial. The Frattini subgroup of $\mathrm{G}(\Or/I_1^{e_1})$
	is thus contained in the kernel of $\pi_1$. Let us show that $\ker \pi_1$ is contained in $\Phi(\mathrm{G}(\Or/I_1^{e_1}))$.
	Assume by contradiction that there exists a maximal subgroup $H\subset \mathrm{G}(\Or/I_1^{e_1})$ which does not contain $\ker\pi_1$.
	Then $\langle H,\ker\pi_1\rangle=\mathrm{G}(\Or/I_1^{e_1})$, so $H$ surjects onto $\mathrm{PG}(\Or/I_1)$.
	Hence, as the kernel of $\mathrm{G}(\Or/I_1)\ra \mathrm{PG}(\Or/I_1)$ is the center,
	the image of $H$ in $\mathrm{G}(\Or/I_1)$ contains the derived subgroup of $\mathrm{G}(\Or/I_1)$, which is equal to $\mathrm{G}(\Or/I_1)$
	(see Schur multipliers \cite[§3.3]{conwayAtlasFiniteGroups1985}).
	Hence, $H$ surjects onto $\mathrm{G}(\Or/I_1)$. As $g\geq 3$ and $p\geq 5$,
	the rank of $\mathrm{G}$ is at least $5$. Moreover, $\mathrm{G}$ is simply connected. Hence, we can
	apply \cite[Theorem 1.3]{vasiuSurjectivityCriteriaPadic2003},
	which tells us that $H$ surjects onto $\mathrm{G}(\Or/I_1^{e_1})$. So, $H$ is not a proper subgroup, contradiction.
	Hence, $\Phi(\mathrm{G}(\Or/I_1^{e_1}))=\ker\pi_1$. Similarly for $i=2,\dotsc,k$. Now
	$\Phi(\mathrm{G}(\Or/I))=\Phi(\mathrm{G}(\Or/I_1^{e_1}))\times\dotsb\times\Phi(\mathrm{G}(\Or/I_k^{e_k}))$, so $\Phi(\mathrm{G}(\Or/I))$
	is $\ker \pi_1\times\dotsb\times\ker\pi_k=\ker\pi$, as desired.
	
	The second statement is deduced from the first one: if $H\subset \mathrm{G}(\Or/I)$ surjects onto $\mathrm{PG}(\Or/\sqrt{I})$,
	then $\langle H,\ker \pi\rangle=\mathrm{G}(\Or/I)$. But $\ker \pi=\Phi(\mathrm{G}(\Or/I))$, so $H=\mathrm{G}(\Or/I)$.
\end{proof}

We can now complete the proof of \Cref{thm:surjectivityImproved}.
Let $I\subset \Zrp$ be a non-zero ideal prime to $p$. By \Cref{lemma:surjectivityProduct},
$\Modtgn$ surjects onto $\PSU(h_{g,\ul})(\Zrp/\sqrt{I})$.
By \Cref{lemma:Frattini}, this implies that $\Modtgn$ surjects onto $\SU(h_{g,\ul})(\Zrp/I)$, as desired.

\subsection{A word on the images of the \texorpdfstring{representations $\bm{\rho_{p,g,\ul}}$}{SO(3) quantum representations} mod
\texorpdfstring{$\bm{p}$}{p}}
\label{sec:modp}

Let $p\geq 7$ be a prime and $g\geq 4$. Consider the Johnson subgroup $J_2(\Sigma_g)\subseteq \Mod(\Sigma_g)$ generated by
Dehn twists along separating {\scc}s. Separating curves act non-trivially modulo $J$ for any $J\subset \Z[\zeta_p]$
maximal ideal prime to $p$. As $J_2(\Sigma_g)$ is normal, by \Cref{thm:surjectivity},
this implies that $\rho_{J,g}(J_2(\Sigma_g))$ is
$\PSL(V_{J,g,\ul})$ if the degree of $\F_q=\Z[\zeta_p]/J$ over its prime field is odd, and $\PSU(V_{J,g,\ul})$ otherwise.
The situation is very different modulo $h=1-\zeta_p-1$: $\rho_{(h),g}$ is trivial on $J_2(\Sigma_g)$,
see \cite[Corollary 2.4(ii)]{detcherryKernel$mathrmSO3$WittenReshetikhinTuraevQuantum2025}.

However, as $\rho_{p,g}$ has dense image in $\PU_{d_{p,g}}$, so does ${\rho_{p,g}}_{\mid J_2(\Sigma_g)}$.
Hence, if one is able to prove that the adjoint trace field of ${\rho_{p,g}}_{\mid J_2(\Sigma_g)}$ is $\Q(\zeta_p+\zeta_p^{-1})$,
then strong approximation \cite[Corollary 10.6]{weisfeilerStrongApproximationZariskidense1984}
would imply that there exists $k\geq 1$ such that the closure
of the image of ${\rho_{p,g}}_{\mid J_2(\Sigma_g)}$ in $\PU(h_{p,g})(\Z[\zeta_p]_h)$
contains $\PU(h_{p,g})(\Z[\zeta_p]_h)\cap (I+h^k\mathrm{End}(V_{p,g}))$. Here $\Z[\zeta_p]_h$
is the $h$-adic completion.
We show the following.

\begin{proposition}
	Let $p\geq 7$ be a prime, $g\geq 4$ and $k$ such that the closure of
	the image of ${\rho_{p,g}}_{\mid J_2(\Sigma_g)}$ in $\PU(h_{p,g})(\Z[\zeta_p]_h)$
	contains $\PU(h_{p,g})(\Z[\zeta_p]_h)\cap (I+h^k\mathrm{End}(V_{p,g}))$.
	Then with $r=\mathrm{rk}_{\F_p}\:J_2(\Sigma_g)^\mathrm{ab}\otimes \F_p$,
	$$\frac{1}{1-1/r}r^k\geq \binom{\dim V_{p,g}}{2}.$$
	In particular $k$ goes to infinity as $p$ goes to infinity.
\end{proposition}

\begin{remark}
	Note that by \cite[Theorem A]{churchFiniteGenerationJohnson2021}, $J_2(\Sigma_g)$ is finitely generated for $g\geq 4$,
	and the rank $\mathrm{rk}_{\Q}\:J_2(\Sigma_g)^\mathrm{ab}\otimes \Q$ is known for $g\geq 6$
	by \cite[Theorem B]{dimcaAbelianizationJohnsonKernel2014}. This rank will match
	$\mathrm{rk}_{\F_p}\:J_2(\Sigma_g)^\mathrm{ab}\otimes \F_p$ for all but finitely many $p$.
\end{remark}

\begin{proof}
	Let $G$ be the quotient of $\PU(h_{p,g})(\Z[\zeta_p]_h)\cap (I+h\mathrm{End}(V_{p,g}))$
	by $\PU(h_{p,g})(\Z[\zeta_p]_h)\cap (I+h^{k+1}\mathrm{End}(V_{p,g}))$ and let $H$ be the image of $J_2(\Sigma_g)$ in $G$.
	Both $G$ and $H$ are finite nilpotent $p$-groups of nilpotency rank $\leq k$.
	Hence, $\log_p|H|\leq r+r^2+\dots+r^k=\frac{r^{k+1}-r}{r-1}\leq \frac{1}{1-1/r}r^k$. Now, by assumption, $H$ contains the image of
	$\PU(h_{p,g})(\Z[\zeta_p]_h)\cap (I+h^{k}\mathrm{End}(V_{p,g}))$ in $G$. This image can be identified
	with a space of antisymmetric matrices in $\mathrm{End}(V_{p,g})\otimes_{\Z[\zeta_p]} \Z[\zeta_p]/(h)$
	for the quadratic or skew-quadratic form induced by $h_{p,g}$ modulo $h$.
	By linear algebra, this space has dimension at least $\binom{\dim V_{p,g}}{2}$ over $\F_p$.
	Hence, $\log_p|H|\geq \binom{\dim V_{p,g}}{2}$ and $\frac{1}{1-1/r}r^k\geq \binom{\dim V_{p,g}}{2}$.
\end{proof}

\section{The case \texorpdfstring{$p=5$}{p=5}}
\label{sec:case-p=5}
In this section, we prove versions of our main theorem for the representations $\rho_{5,(2)^n}$ for any surface $\Sigma_{g,n}$ with $(g,n)\neq (1,0)$.
Note that when $p=5$, the set of colors is $\Lambda=\lbrace 0,2\rbrace$, so $2$ is the only non-zero color.
Because of this, we will denote the representations $\rho_{5,(2)^n}$ simply by $\rho_{5,n}$ in this section.
Similarly, we will simplify the notation $V_{g,(2)^n}$ (resp. $V_{J,g,(2)^n}$) as $V_{g,n}$ (resp. $V_{J,g,n}$) in this section.

The case $p=5$ was a recurring exception in many of the arguments of \Cref{sec:simplicity,sec:densityProof,sec:caracFinie}.
However, different techniques may be used in this case, thanks to the fact that the images
of Dehn twists by $\rho_{5,g,n}$ have only two eigenvalues.

In \cite{FLW}, it was proved that the representations $\rho_{5,n}$ have dense images in $\mathrm{PSU}(V_{g,n})$ for any surface $\Sigma_{g,n}$
with $(g,n)\neq (1,0)$.

In this section, we will show that for any maximal ideal $J\neq (\zeta_p-1)$, the representations $\rho_{5,J,n}$ are surjective onto $\PSU(V_{J,g,n})$
or $\PSL(V_{J,g,n})$, depending on whether $\overline{J}=J$ or not.

\subsection{Small surfaces cases}
\label{sec:4holedSphere}

In this section, we treat the case of the $4$-holed sphere. This will be the base case of an induction proof for general surfaces in the next section.
In fact, since this case is not covered by \Cref{thm:density,thm:surjectivity} even when $p\geq 7$, we will work more generally
with any prime $p\geq 5$ and a $4$-holed sphere with boundary components all colored by $p-3$. Note that $\dim V_{0,(p-3)^4}=2$ for any $p\geq 5$.
We have the following:

\begin{proposition}
	\label{prop:4holedSphere} For any prime $p\geq 5$, the representation $\rho_{0,(p-3)^4}$ of $\PMod(\Sigma_{0,4})$ has dense image in $\mathrm{PSU}_2$.
	
	Moreover, for any maximal ideal $J\subset \Z[\zeta_p]$ of norm $q$ coprime to $p$, the image of $\rho_{J,0,(p-3)^4}$ is $\mathrm{PSL}_2(\F_q)$ if $q$
	is not a square, and $\mathrm{PSU}_2(\F_q)$ if $q$ is a square.
\end{proposition}

We will need the following theorem, which describes the structure of (closed) subgroups of $\mathrm{PSU}_2\simeq \mathrm{SO}_3$ and $\mathrm{PSL}_2(\F_q)$.
Note that $\mathrm{PGL}_2(\F_{\sqrt{q}})$ may be viewed as a subgroup of $\mathrm{PSL}_2(\F_q)$ when $q$ is an odd square.
This follows from the fact that any element of $\F_{\sqrt{q}}^{\times}$ is a square in $\F_q^{\times}$.

\begin{theorem}
	\label{thm:subgroupPSL2}\begin{enumerate}
		\item Any proper closed subgroup of $\mathrm{PSU}_2$ is either conjugate into the normalizer of a maximal torus, or isomorphic to $A_4, S_4$
		or $A_5$.
		\item \cite{Dickson} Any proper subgroup of $\mathrm{PSL}_2(\F_q)$ with $q\geq 5$ odd is either conjugate into the normalizer of a maximal torus,
		into $\mathrm{PSL}_2(\F_{q'})$ with $q'|q$, $q'\neq q$, into $\PGL_2(\F_{\sqrt{q}})$ if $q$ is a square, into a Borel subgroup, or is isomorphic
		to $A_4,S_4$ or $A_5$.
		\item \cite{Dickson} Any proper subgroup of $\mathrm{PSL}_2(\F_{2^l})$ with $l\geq 2$ is either conjugate into the normalizer of a maximal torus,
		into $\mathrm{PSL}_2(\F_{2^k})$ with $k<l$, into a Borel subgroup or is isomorphic to $A_4$, $S_4$ or $A_5$.
		\end{enumerate}
\end{theorem}
		\begin{proof}
			Items (2) and (3) are part of Dickson's theorem \cite{Dickson} that classified the subgroups of $\mathrm{PSL}_2(\F_q)$.
			See also \cite{BHR} for a modern reference.
			As for (1), let $G$ be a proper closed subgroup of $\mathrm{PSU}_2$. If $G$ is finite, it is well known that $G$ is either abelian,
			dihedral or isomorphic to $A_4$, $S_4$ or $A_5$. If not, then the Lie algebra of $G$ must be a proper non-trivial subalgebra of
			$\mathfrak{su}_2$, hence it must have dimension $1$.
			Therefore, the connected component of the identity element in $G$ must be a maximal torus of $\mathrm{PSU}_2$, and $G$ must belong
			to its normalizer.
		\end{proof}

Before proceeding with the proof of Proposition \ref{prop:4holedSphere}, we make one final remark: the group $\mathrm{PSU}_2(\F_{q^2})$ is
isomorphic to $\mathrm{PSL}_2(\F_q)$ for any prime power $q$; see for example \cite[Lemma 3.1.1(ii)]{BHR}. Hence, Theorem \ref{thm:subgroupPSL2}
also informs us about the structure of subgroups of $\mathrm{PSU}_2(\F_{q^2})$.
		\begin{proof}[Proof of Proposition \ref{prop:4holedSphere}] By Masbaum \cite{M99}, the following two elements of $\mathrm{PSL}_2(\Z[\zeta_p])$
			are in the image of $\rho_{0,(p-3)^4}$:
			$$A=\begin{pmatrix}
				1 & 0 \\ 0 & \zeta^4
			\end{pmatrix} \ \textrm{and} \ B=\begin{pmatrix}
			\zeta^4 + \delta^{-2}(1-\zeta^4) & \delta^{-1}(1-\delta^{-2})(1-\zeta^4)
			\\ \delta^{-1}(\zeta^{-4}-1) & \zeta^{-4}+\delta^{-2}(1-\zeta^{-4})
			\end{pmatrix}$$
		
		In the above, we wrote $\zeta$ as shorthand for $\zeta_p$, which is a primitive $p$-th root of unity, and we set $\delta=-\zeta-\zeta^{-1}$.
		Note that $\delta$ is a unit in $\Z[\zeta_p]$. The elements $A$ and $B$ can be written as
		$\rho_{0,(p-3)^4}(t_c)$ and $\rho_{0,(p-3)^4}(t_c^{-1}t_d)$ where $c$ and $d$ are two separating simple closed curves
		intersecting geometrically twice, and $t_c,t_d$ are the associated Dehn twists.
		Furthermore, it is shown in \cite{M99} that $B$ has infinite order. If the closure of the subgroup $\langle A,B \rangle$
		were conjugate into the normalizer of a maximal torus, then this maximal torus would be the centralizer of $B$,
		and either $A$ commutes with $B$ or $ABA^{-1}=B^{-1}$. It is straightforward to check that this is not the case.
		
		Now, let $J$ be a maximal ideal of $\Z[\zeta_p]$ with norm $q$ coprime to $p$. Let $\overline{A}$ and
		$\overline{B}$ be the reductions mod $J$ of $A$ and $B$, and let $G=\langle \overline{A},\overline{B} \rangle$.
		Then $\overline{A}$ has order $p$ since $\zeta^4\neq 1\mod J$. Therefore, for $G$ to be conjugate into a Borel subgroup,
		one would need this Borel subgroup to be the stabilizer of one of the eigenlines of $\overline{A}$,
		and hence the upper right or upper left coefficient of $\overline{B}$ should vanish; this is not the case.
		Similarly, if $G$ were conjugate into the normalizer of a maximal torus, it would be a subgroup of a dihedral group.
		However, $\overline{A}\overline{B}$ does not commute with $\overline{A}$ (since $\overline{B}$ does not)
		and these elements both have order $p$ as they are images of Dehn twists, so $G$ is not a subgroup of a dihedral group.

		Next we show that $G$ is not included (up to conjugation) in $\mathrm{PGL}_2(\F_{q'})$, and hence
		also not included in $\mathrm{PSL}_2(\F_{q'})$, for any proper divisor $q'$ of $q$ if $q$ is not a square.
		Let $q=q_0^{2k+1}$ where $q_0$ is a prime. Note that $G$ contains an element of order $p$ (namely $\overline{A}$).
		However, since $2k+1$ is the order of $q_0$ mod $p$, we get that $p$ does not divide $|\mathrm{PGL}_2(\F_{q'})|=q'(q'^2-1)$.
		
		Similarly, if $q=q_0^{2k}$, we know that $\rho_{0,(p-3)^4}$ takes values in $\mathrm{PSU}_2(\F_q)\simeq \mathrm{PSL}_2(\F_{\sqrt{q}})$.
		Again we have that $p$ does not divide the order of $\mathrm{PGL}_2(\F_{q'})$ for any strict divisor $q'$ of $\sqrt{q}$.
		
		Finally, we check whether $G$ can be $A_4$, $S_4$ or $A_5$. It suffices to check that $G$ has an element of order $>5$.
		If $p\geq 7$, then the order of $\overline{A}$ is $p>5$. If $p=5$, a computation using SageMath shows that the norm of the
		lower left coefficient of $B^i$ is not divisible by any prime other than $2,3,5$ or $31$, for $1\leq i \leq 5$.
		Hence, $\overline{B}$ has order at least $6$ if the norm of $J$ is coprime to $2,3,5$ and $31$.
		A similar computation using the powers of $C=AB^2$ eliminates the case where the norm of $J$ is $31$.
		Finally, if $J$ divides $(2)$ (resp. $J$ divides $(3)$), then $\rho_{J,0,(2)^4}$
		takes values in $\mathrm{PSU}_2(\F_{16})\simeq \mathrm{PSL}_2(\F_4)\simeq A_5$
		(resp. $\mathrm{PSU}_2(\F_{81})\simeq \mathrm{PSL}_2(\F_9)\simeq A_6$),
		and a SageMath computation shows that $\rho_{J,0,(2)^4}$ is surjective in both cases.
\end{proof}

\subsection{Induction on the genus and number of boundary components}
\label{sec:induction}

In this section, we restrict to $p=5$. Therefore, we will write $V_{J,g,n}$ and $\rho_{J,g,n}$ for $V_{J,g,(2)^n}$ and $\rho_{J,g,(2)^n}$ in this section.
We prove the following:

\begin{theorem}
	\label{thm:casep=5} For any compact oriented surface $\Sigma_{g,n}$ with $(g,n)\neq (1,0)$, and for any maximal ideal
	$J\neq (1-\zeta_5)$ of $\Z[\zeta_5]$,
	$\rho_{J,g,n}$ is surjective onto $\mathrm{PSL}(V_{J,g,n})$ or $\mathrm{PSU}(V_{J,g,n})$,
	depending on whether $\overline{J}\neq J$ or $\overline{J}=J$ respectively.
\end{theorem}

\begin{remark}
	Applying \Cref{lemma:surjectivityProduct,lemma:Frattini}, we see that when
	$g\geq 3$, the result of \Cref{thm:casep=5} also applies to any non-zero ideal $J$.
	See \Cref{thm:surjectivityImproved} for how to properly state the result in that case.
\end{remark}

We will need to use the following two Lemmas.

\begin{lemma}\label{lemma:inductionLevel5}
Let $\mathbb{F}_q$ be a finite field and let $V_1$, $V_2$, $V_3$ be three non-zero vector spaces over $\mathbb{F}_q$.
Let $G \subseteq \SL(V_1 \oplus V_2 \oplus V_3)$ be a subgroup containing $\SL(V_1) \times \SL(V_2 \oplus V_3)$ and $\SL(V_1 \oplus V_2) \times \SL(V_3)$.
Then $G = \SL(V_1 \oplus V_2 \oplus V_3)$. When $q$ is an even power of a prime, the same result holds replacing $\SL$ by $\SU$.
\end{lemma}

\begin{proof}
Let us first prove the case of $\SL$.
Choose a basis $\mathbb{F}_q^d \simeq V_1 \oplus V_2 \oplus V_3$ formed by concatenating bases of $V_1, V_2$ and $V_3$.
Then $\SL(\mathbb{F}_q^d)$ is generated by the standard transvections $I_d + E_{i, i+1}$ and $I_d + E_{i+1, i}$, $1 \le i \le d-1$
\cite{humphriesGenerationSpecialLinear1986}.
But each such transvection is either in $\SL(V_1 \oplus V_2)$ or $\SL(V_2 \oplus V_3)$, hence the result.

Let us now turn to the case of $\SU$.
If necessary, we may change $(V_1,V_2,V_3)$ to $(V_3,V_2,V_1)$ without loss of generality to ensure that $\dim V_1\neq \dim V_2+\dim V_3$.
Then $\operatorname{Stab}_{V_1} = (\mathrm{U}(V_1) \times \mathrm{U}(V_2 \oplus V_3)) \cap \SU(V_1 \oplus V_2 \oplus V_3)$ is a maximal subgroup of
$\SU(V_1 \oplus V_2 \oplus V_3)$,
see \cite[Table 3.5.B]{KleidmanLiebeck} if the rank is $\ge 13$
and \cite[Tables 8.5,10,20,26,37,46,56,62,72,78]{BHR} if the rank is $< 13$.
Now, for any $\alpha \in \mathbb{F}_{q}^\times$ of norm $1$ in $\mathbb{F}_{\sqrt{q}}$, consider a matrix $P_{1, \alpha}$ in $\mathrm{U}(V_1)$
of determinant $\alpha$ and a matrix $P_{2, \alpha^{-1}}$ in $\mathrm{U}(V_2)$ of determinant $\alpha^{-1}$.
Then $M_\alpha = P_{1, \alpha} \oplus P_{2, \alpha^{-1}} \oplus \operatorname{id}_{V_3}$ is in $\SU(V_1 \oplus V_2) \times \SU(V_3)$,
and the matrices $M_\alpha$ together with $\SU(V_1) \times \SU(V_2 \oplus V_3)$ generate $\operatorname{Stab}_{V_1}$.
So, $\operatorname{Stab}_{V_1}\subseteq G$.
As $\SU(V_1 \oplus V_2) \times \SU(V_3) \not\subseteq \operatorname{Stab}_{V_1}$,
we have $\operatorname{Stab}_{V_1} \subsetneq G$ and hence $G = \SU(V_1 \oplus V_2 \oplus V_3)$ by maximality of $\operatorname{Stab}_{V_1}$.
\end{proof}

\begin{lemma}\label{lemma:inductionLevel5genus2}
Let $\mathbb{F}_q$ be a finite field and let $G \subseteq \SL_5(\mathbb{F}_q)$ be a subgroup acting irreducibly on $\mathbb{F}_q^5$
and containing $\SL_4(\mathbb{F}_q) \times \operatorname{id}_{\mathbb{F}_q}$. Then $G = \SL_5(\mathbb{F}_q)$.
When $q$ is an even power of a prime, the same holds replacing $\SL$ by $\SU$.
\end{lemma}

\begin{proof}
Maximal subgroups of $\SL_5(\mathbb{F}_q)$ are listed in \cite[Tables 8.18 and 8.19]{BHR}.
The only ones that contain $\SL_4(\mathbb{F}_q) \times \operatorname{id}_{\mathbb{F}_q}$
are the stabilizer of $\mathbb{F}_q^4 \subseteq \mathbb{F}_q^5$
and the stabilizer of $(\mathbb{F}_q^4)^* \subseteq (\mathbb{F}_q^5)^*$ (line $1$ in Table $8.18$ in the reference).
Neither of these acts irreducibly, hence $G$ is not included in either.
So, $G = \SL_5(\mathbb{F}_q)$.
To exclude the possibility that $\SL_4(\mathbb{F}_q)$ is contained in one of the non-geometric maximal subgroups
$\mathbb{Z}/5\mathbb{Z} \times \SL_2(\mathbb{F}_{11})$
and $\mathbb{Z}/5\mathbb{Z} \times \mathrm{U}_4(\mathbb{F}_2)$
in Table 8.19 in the reference,
note that as $q \ge 11$, the cardinality of $|\SL_4(\mathbb{F}_q)| \ge 10^{15}$ far exceeds their cardinality ($6600$ and $388800$ respectively).

The case of $\SU_4(\mathbb{F}_q) \subseteq G \subseteq \SU_5(\mathbb{F}_q)$ is treated similarly using \cite[Tables 8.20 and 8.21]{BHR}.
The only maximal subgroup containing $\SU_4(\mathbb{F}_q)$ is then
$(\mathrm{U}_4(\mathbb{F}_q) \times \mathrm{U}_1(\mathbb{F}_q)) \cap \SU_5(\mathbb{F}_q)$,
which does not act irreducibly on $\mathbb{F}_{q}^5$. The non-geometric maximal subgroups to exclude are the same as for $\SL_5(\mathbb{F}_q)$
and are dealt with by noting that $|\SU_4(\mathbb{F}_q)| \ge 10^{9}$ because $q \ge 16$.
\end{proof}

Using \Cref{prop:4holedSphere,lemma:inductionLevel5,lemma:inductionLevel5genus2}, Theorem \ref{thm:casep=5} can be proved by induction.
\begin{proof}[Proof of Theorem \ref{thm:casep=5}]
In this proof, we will denote $\mathrm{PSL}$ or $\mathrm{PSU}$ by $\mathbf{PG}$ and $\mathrm{SL}$ or $\mathrm{SU}$ by $\mathbf{G}$,
depending on whether $J$ is stable under complex conjugation.
We first consider the case of $\Sigma_{0,n}$, where $n\geq 0$. The cases $n=0,1,2,3$ are obvious since
$\dim V_{J,0,0}=\dim V_{J,0,2}=\dim V_{J,0,3}=1$ and $V_{J,0,1}=0$.
Moreover, the case $n=4$ is covered by Proposition \ref{prop:4holedSphere}.

Assume the proposition is true for all $\Sigma_{0,k}$ with $2\leq k\leq n$.
Pick a simple closed curve $\gamma$ in $\Sigma_{0,n+1}$ that cuts it into a pair of pants and a surface $\Sigma_{0,n-1}^1$.
We get an induced decomposition as $\PMod(\Sigma_{0,n-1}^1)$-representations
$$V_{J,0,{n+1}}\simeq V_{J,0,n}\oplus V_{J,0,{n-1}},$$
since $V_{J,0,(2,2,2)}$ and $V_{J,0,(2,2,0)}$ have dimension $1$.
The restriction of $\rho_{J,0,{n+1}}$ to $\PMod(\Sigma_{0,n-1}^1)$ is surjective on $\mathbf{PG}(V_{J,0,n})$
and $\mathbf{PG}(V_{J,0,{n-1}})$; in fact it surjects onto $\mathbf{PG}(V_{J,0,n})\times \mathbf{PG}(V_{J,0,{n-1}})$
by Hall's diagonal Lemma (see \cite[Lemma 3.5]{Kup11}), since $\dim V_{J,0,n}>\dim V_{J,0,{n-1}}$.
Now, choose another such curve $\gamma'$, disjoint from $\gamma$. Cutting along $\gamma$ and $\gamma'$, we get a decomposition:
$$V_{J,0,{n+1}}\simeq V_{J,0,n-1}\oplus V_{J,0,{n-2}}\oplus V_{J,0,{n-2}}\oplus V_{J,0,n-3}=U_1\oplus U_2\oplus U_3\oplus U_4,$$
with all the $U_i$ non-zero except if $n=4$, in which case $U_4=0$. By the above, we have that
the image of $\rho_{J,0,{n+1}}$ contains the images of both $\mathbf{G}(U_1\oplus U_2)\times\mathbf{G}(U_3\oplus U_4)$
and $\mathbf{G}(U_1\oplus U_3)\times\mathbf{G}(U_2\oplus U_4)$. Applying \Cref{lemma:inductionLevel5} twice
we have that $\rho_{J,0,{n+1}}$ is surjective onto $\mathbf{PG}(V_{J,0,{n+1}})$.

Next we induct on the genus of $\Sigma_{g,n}$. The case $g=1$, $n=1$ is obvious since $\dim V_{J,1,(2)}=1$.
Assume that the proposition is true for $\Sigma_{g,n}$ for all $n\geq 0$ and all $g\leq g_0$.
Pick a non-separating simple closed curve $\gamma$ in $\Sigma_{g_0+1,n}$ for some $n\geq 0$.
The curve $\gamma$ cuts the latter surface into a surface $\Sigma_{g_0,n}^2$. We get a decomposition of $\PMod(\Sigma_{g_0,n}^2)$-representations:
$$V_{J,g_0+1,n}\simeq V_{J,g_0,n}\oplus V_{J,g_0,{n+2}}.$$
As $(g_0,n)\neq (0,0)$ by assumption, $\dim V_{J,g_0,{n+2}}=\dim V_{J,g_0,n}+\dim V_{J,g_0,{n+1}}>\dim V_{J,g_0,n}$
and $\rho_{J,g_0+1,n}$ in restriction to $\PMod(\Sigma_{g_0,n}^2)$
surjects onto $\mathbf{PG}(V_{J,g_0,{n+2}})$ by the induction hypothesis.
If $(g_0,n)\neq (1,0)$, then the restriction also surjects onto $\mathbf{PG}(V_{J,g_0,{n}})$
and hence by Hall's diagonal Lemma, also on $\mathbf{PG}(V_{J,g_0,{n+2}})\times \mathbf{PG}(V_{J,g_0,{n}})$.
If $g_0\geq 1$, choose another non-separating curve $\gamma'$. Cutting along both $\gamma$ and $\gamma'$, we have a decomposition
$$V_{J,g_0+1,{n}}\simeq V_{J,g_0-1,n+4}\oplus V_{J,g_0-1,{n+2}}\oplus V_{J,g_0-1,n+2}\oplus V_{J,g_0-1,n}=U_1\oplus U_2\oplus U_3\oplus U_4,$$
and for $(g_0,n)\neq (1,0)$, the image of $\rho_{J,g_0+1,{n}}$ contains the images of both $\mathbf{G}(U_1\oplus U_2)\times\mathbf{G}(U_3\oplus U_4)$
and $\mathbf{G}(U_1\oplus U_3)\times\mathbf{G}(U_2\oplus U_4)$.
If $g_0=0$, then $n\geq 2$ and choose a simple closed curve $\gamma'$ in $\Sigma_{g_0+1,n}$ disjoint from $\gamma$ encircling $2$ of the marked points.
We then have a decomposition
$$V_{J,1,{n}}\simeq V_{J,0,n+1}\oplus V_{J,0,{n}}\oplus V_{J,0,n-1}\oplus V_{J,0,n-2}=U_1\oplus U_2\oplus U_3\oplus U_4,$$
with each $U_i$ non-zero unless $n=2$ (then $U_3=0$) or $n=3$ (then $U_4=0$),
and the image of $\rho_{J,g_0+1,{n}}$ contains the images of both $\mathbf{G}(U_1\oplus U_2)\times\mathbf{G}(U_3\oplus U_4)$
and $\mathbf{G}(U_1\oplus U_3)\times\mathbf{G}(U_2\oplus U_4)$.
In both cases $g_0=0$ and $g_0\geq 1$, applying \Cref{lemma:inductionLevel5} twice, we have that if $(g_0,n)\neq (1,0)$,
$\rho_{J,g_0+1,{n}}$ is surjective onto $\mathbf{PG}(V_{J,g_0+1,{n}})$.

In the case $(g_0,n)= (1,0)$, we have $\dim V_{J,2,0}=5$, $\dim U_1=2$, $\dim U_2=\dim U_3=\dim U_4=1$, and
the image of $\rho_{J,2,0}$ contains the images of both $\mathbf{G}(U_1\oplus U_2)\times\mathbf{G}(U_3)$
and $\mathbf{G}(U_1\oplus U_3)\times\mathbf{G}(U_2)$.
Applying \Cref{lemma:inductionLevel5}, we see that the image contains the image of $\mathbf{G}(U_1\oplus U_2\oplus U_3)$.
Applying \Cref{thm:irreducibility,lemma:inductionLevel5genus2}, we have that
$\rho_{J,2,0}$ surjects onto $\mathbf{PG}(V_{J,2,0})$.
\end{proof}

\section{Asymptotic faithfulness for specific sequences of boundary colors}
\label{sec:asympFaith}

It is well known that the $\mathrm{SO}(3)$-quantum representations $\rho_{p,g}$ of $\Mod(\Sigma_g)$ are asymptotically faithful;
this result was proved independently by Andersen \cite{A06} and Freedman, Walker, and Wang \cite{FWW02}.
Considering mapping classes with support $\Sigma_g^n \subset \Sigma_{g'}$ for some $g'$, one can readily deduce from asymptotic faithfulness for closed surfaces and fusion rules that the representations
$$\underset{\ul \in \Lambda^n}{\bigoplus}\rho_{p,g,\ul}$$
of $\Mod(\Sigma_{g}^{n})$ are also asymptotically faithful (modulo center). However,
for our purposes we wish to establish a stronger form of asymptotic faithfulness for surfaces with boundary.

Since in this section, the choice of parameter $p$ will vary, we will denote by $\Lambda_p$ the set of colors $\lbrace 0,2\ldots,p-3\rbrace$, and by $i_0(p)$ the permissive color introduced in Lemma \ref{lemma:permColor}.

For $\Sigma_{g}^{n}$ a compact oriented surface with genus $g$ and $n$ boundary components, for $\underline{x}\in [0,1]^n$,
and for $\Gamma$ a trivalent graph dual to some pair of pants decomposition of $\Sigma_{g}^{n}$,
we define a polytope $\Delta(\underline{x})$ as follows. Elements of $\Delta(\underline{x})$ consist of maps $\theta:E\longrightarrow [0,1]$ such that
$$\theta(i)+\theta(j)+\theta(k)\leq 2 \ \textrm{and} \ \theta(i)\leq \theta(j)+\theta(k)$$
if $(i,j,k)$ is a triple of edges adjacent to the same vertex, and $\theta(i)=x_i$ when $i$ is an external edge.

We note that $\Delta(\underline{x})$ may be viewed as a subset of $\R^{3g-3+n}$. One may think of $\Delta(\underline{x})$ as the polytope of "renormalized" $p$-admissible colorings of $\Gamma$, as $\frac{c(p)}{p}\in \Delta(\frac{1}{p}\ul(p))$ if $c(p)$ is a $p$-admissible coloring with boundary colors $\ul(p)$.
\begin{theorem}
	\label{thm:asymptFaith} Let $\Sigma_{g}^{n}$ be a compact oriented surface, and let $(\ul(p))_{p \geq 5, \ p \ \textrm{prime}}$
	be a sequence such that $\ul(p)\in \Lambda_p^n$ and $\frac{1}{p}\ul(p)\longrightarrow \underline{x}$.
	Assume furthermore that $\Delta(\underline{x})$ has non-empty interior. Then
	$$\underset{p\geq 5,\, p \, \textrm{prime}}{\bigcap} \ker \rho_{p,g,\ul(p)}=Z(\Mod(\Sigma_{g}^{n})).$$
\end{theorem}

\begin{remark}
	\label{rk:nonDependenceInDelta} While the polytope $\Delta(\underline{x})$ depends on the choice of $\Gamma$,
	whether it has non-empty interior depends only on $\underline{x}$.
	Indeed, one can show that this condition is equivalent to having $\dim V_{p,g,\ul(p)}\sim A p^{3g-3+n}$ for some constant $A$.
\end{remark}

\begin{remark}
	\label{rk:asympFaith} As will be clear in the proof of Theorem \ref{thm:asymptFaith},
	the same result holds if the intersection over all primes is replaced by the intersection over all primes $p\geq N$ for some integer $N$.
\end{remark}
Before proving Theorem \ref{thm:asymptFaith}, we mention the following corollary:
\begin{corollary}
	\label{cor:asymptFaith} For any compact oriented surface $\Sigma_{g}^{n}$,
	$$\underset{p\geq 5,\, p \, \textrm{prime}}{\bigcap} \ker \rho_{p,g,(i_0(p))^n}=Z(\Mod(\Sigma_{g}^{n})).$$
\end{corollary}
\begin{proof}
	Note that $\frac{i_0(p)}{p}\underset{p\rightarrow \infty}{\longrightarrow}\frac{1}{2}$;
	thus we need to show that $\Delta(\frac{1}{2},\ldots, \frac{1}{2})$ has non-empty interior.
	However, coloring each internal edge of the graph $\Gamma$ by $\frac{1}{2}$ gives a point in the interior of
	$\Delta(\frac{1}{2},\ldots, \frac{1}{2})\subset \R^{3g-3+n}$.
\end{proof}

Before proving Theorem \ref{thm:asymptFaith}, we briefly review the set-up of \cite{Det16}, whose results we will use.
Pick a pair of pants decomposition $\mathcal{P}=\lbrace \alpha_1,\ldots,\alpha_{3g-3+n}\rbrace$ of $\Sigma_{g}^{n}$.
Let $(\varphi_c)$ be the associated basis of $V_{p,g,\ul(p)}$. 
We consider the moduli space
$$\mathcal{M}(\Sigma_g^n)=\mathrm{Hom}(\pi_1(\Sigma_g^n),\mathrm{SU}_2)/\mathrm{SU}_2,$$
which has a natural structure of Poisson variety, given by the Atiyah-Bott-Goldman bracket. We also set 
$$\mathcal{M}(\Sigma_g^n,\underline{x})=\lbrace \rho \in \mathcal{M}(\Sigma_g^n) \ | \ \tr(\rho(\delta_i))=2\cos(\pi x_i)\rbrace,$$
where $\delta_i\in \pi_1(\Sigma_g^n)$ is in the conjugacy class corresponding to the $i$-th boundary component of $\Sigma_g^n$.
For any curve $\delta$ on $\Sigma_g^n$, we have an associated trace function 
$$t_{\delta}:\rho \longrightarrow -\tr(\rho(\delta))$$
on $\mathcal{M}(\Sigma_g^n)$. Hamiltonian flows by the angle functions $a_{\alpha_i}=\frac{1}{\pi}\mathrm{acos}(t_{\alpha_i}/2)$ give rise to an action of a torus $(\R/2\pi\Z)^{3g-3+n}$ on $\mathcal{M}(\Sigma_g^n)$, and a (surjective) momentum map 
$$\mu : \mathcal{M}(\Sigma_g^n,\underline{x})\longrightarrow \Delta(\underline{x}),$$
whose fibers are exactly the orbits of the torus action.

Recall that to any simple closed curve $\delta$ on $\Sigma_g^n$, one may associate a curve operator $T_{p,\delta}\in \mathrm{End}(V_{p,g,\ul})$. Pick a sequence $c(p)$ of $p$-admissible coloring of a dual trivalent graph $\Gamma$ to $\mathcal{P}$, with boundary colors $\ul(p)$. We assume that $\frac{c(p)}{p}$ converges to $\theta \in \mathrm{Int}(\Delta(\underline{x}))$.
Then, the content of \cite[Theorem 1.1 and 1.3]{Det16} is that for any $m\in \Z^{3g-3+n}$, and any simple closed curve $\alpha$ on $\Sigma_g^n$ the coefficient $\langle T_{p,\alpha}\varphi_{c(p)},\varphi_{c(p)+m}\rangle$ converges the $m$-th Fourier coefficient of $t_{\alpha}$ on $\mu^{-1}(\theta)$.

We will need some non-vanishing results for some of those Fourier coefficients, which can be deduced from the results of \cite{CM12}:

\begin{lemma}
	\label{lemma:nonVanishingFourier} Let $\Sigma_g^n$ be a compact oriented surface with boundary, let $\mathcal{P}=\lbrace \alpha_1,\ldots,\alpha_{3g-3+n}\rbrace $ be a pair of pants decomposition of $\Sigma_g^n$ and let $\underline{x}$ be such that $\Delta(\underline{x})$ has non-empty interior. Finally, let $\delta$ be a simple closed curve on $\Sigma_g^n$ and let $m_i=i(\alpha_i,\delta)$.
	Then the $\underline{m}$-th Fourier coefficient of $t_{\delta}$ on $\mu^{-1}(\theta)$ is non-zero for any $\theta \in \mathrm{Int}(\Delta(\underline{x})).$
\end{lemma}
\begin{proof}
	For closed surfaces, Lemma \ref{lemma:nonVanishingFourier} is exactly \cite[Lemma 4.4]{CM12}. We deduce the case with boundary from the closed case. Let $\widehat{\Sigma}=\Sigma_g^n \cup \overline{\Sigma_g^n}$ be the double of $\Sigma_g^n$, and let $\widehat{P}$ be the "double" pair of pants decomposition, obtained from the boundary curves of $\Sigma_g^n$, $\mathcal{P}$ and $\overline{\mathcal{P}}$. Let $\hat{\mu}$ be the associated momentum map on $\mathcal{M}(\widehat{\Sigma})$, with image $\widehat{\Delta}$.
	
	The simple closed curve $\delta$ lies on $\Sigma_{g,n}$, hence Hamiltonian flow by any of the angles functions except $a_{\alpha_i}$ leaves $t_{\delta}$ invariant. Thus, the $(\underline{m},0)$-Fourier coefficient of $t_{\delta}\in \mathcal{M}(\widehat{\Sigma})$ on $\hat{\mu}^{-1}((\underline{x},\theta,\theta))$ is equal to the $\underline{m}$-th Fourier coefficient of $t_{\delta}\in \mathcal{M}(\Sigma_g^n)$ on $\mu^{-1}(\theta)$.
	
	However, if $\theta \in \mathrm{Int}(\Delta(\underline{x})),$ then $(\underline{x},\theta,\theta) \in \mathrm{Int}(\widehat{\Delta})$. (Note that $\Delta(\underline{y})$ also has non-empty interior for $\underline{y}$ close to $\underline{x}$). Therefore, \cite[Lemma 4.4]{CM12} concludes.
\end{proof}

\begin{proof}[Proof of Theorem \ref{thm:asymptFaith}]
	The proof partly follows the strategy of \cite{FWW02}, but avoids the use of their Lemma 4.1 about ``graph-geodesic multicurves'',
	replacing it with more general estimates of some off-diagonal coefficients of curve operators.
	
	Let $f\in \Mod(\Sigma_{g}^{n})$ be non-central, then there exists a simple closed curve $\alpha$, not parallel to a boundary component,
	such that $f(\alpha)$ is non-isotopic to $\alpha$. Let $T_{p,\alpha},T_{p,f(\alpha)} \in \mathrm{End}(V_{p,g,\ul(p)})$
	be the associated curve operators; since $T_{p,f(\alpha)}=\rho_{p,g,\ul(p)}(f)\circ T_{p,\alpha}\circ \rho_{p,g,\ul(p)}(f)^{-1}$,
	it suffices to show that $T_{p,\alpha}\neq T_{p,f(\alpha)}$ for any large enough $p$. If $\alpha$ and $f(\alpha)$ are disjoint, we can pick a pair of pants decomposition containing $\alpha$ and $f(\alpha)$. The associated basis of $V_{p,g,\ul(p)}$ is a basis of common eigenvectors of $T_{p,\alpha}$ and $T_{p,f(\alpha)}$ but the eigenvalues do not match, hence $T_{p,\alpha}\neq T_{p,f(\alpha)}$ in that case. Therefore we assume that the geometric intersection number $i(\alpha,f(\alpha))$ is positive.
	
	Let $m_i=i(f(\alpha),\alpha_i)$. 
	By \cite[Theorems 1.1 and 1.3]{Det16}, for $\frac{c(p)}{p} \longrightarrow \theta \in \mathrm{Int}(\Delta(\underline{x}))$, the off-diagonal coefficients
	
	$$\langle T_{p,f(\alpha)}\varphi_{c(p)},\varphi_{c(p)+m}\rangle$$
	 converge to the $m$-th Fourier coefficient of $t_{f(\alpha)}$ on $\mu^{-1}(\theta)$, and Lemma \ref{lemma:nonVanishingFourier} implies that these coefficients are non-zero for large enough $p$.
	
	Hence, $T_{p,f(\alpha)}\neq T_{p,\alpha}$ for any large enough $p$, since $T_{p,\alpha}$ is diagonal in the basis $(\varphi_c)$.
\end{proof}

\section{Applications}
\label{sec:applications}

In this section, we collect a few applications of our main results.

\subsection{Subnormal cores of some subgroups of mapping class groups}
\label{sec:subnormalCore}

In this section, we prove Corollary \ref{cor:subnormalCore} that was stated in the introduction. Let $G=\Mod(\Sigma_{g}^{n})$
where $g\geq 2$ and $(g,n)\neq (2,1)$.%
\footnote{The condition $(g,n)\neq (2,1)$ comes from the fact that our density result \Cref{thm:density}
does not apply to that case as $V_{0,i_0}=0$.}
Recall that we can consider the following subgroups of $\Mod(\Sigma_{g}^{n})$:
\begin{itemize}
	\item[(i)]When $n=0$, the handlebody groups $\mathcal{H}_g$, which are the subgroups consisting of mapping classes that extend to
	a fixed handlebody of genus $g$ with boundary $\Sigma_g$;
	\item[(ii)]The centralizers $C_{\Mod(\Sigma_{g}^{n})}(f)$, where $f$ is a finite-order non-central mapping class in $ \Mod(\Sigma_{g}^{n})$;
	\item[(iii)]The normalizers $N_{\Mod(\Sigma_{g}^{n})}(\Gamma)$, where $\Gamma$ is a finite non-central subgroup of $\Mod(\Sigma_{g}^{n})$.
\end{itemize}

We note that there exist non-trivial subnormal subgroups of $\Mod(\Sigma_{g}^{n})$ with trivial normal cores. Indeed, by a result of Dahmani,
Guirardel and Osin \cite{DGO17}, the normal closures of high powers of pseudo-Anosov maps are free subgroups (with infinite rank),
and furthermore, by \cite[Theorem 1.1]{Ols17}, there exist normal non-trivial subgroups of any free group of rank at least $2$
with trivial normal core. The authors thank Andrew Putman for pointing out this argument to them.

Corollary \ref{cor:subnormalCore} will follow from the following lemma:

\begin{lemma}
	\label{lemma:subnormal} Let $g\geq 2$ and $(g,n)\neq (2,1)$, and let $H$ be a subgroup of $\Mod(\Sigma_{g}^{n})$ with non-trivial subnormal core.
	Then for $p$ prime large enough, $\rho_{p,g,(i_0)^n}(H)$ is dense in $\mathrm{PSU}(V_{p,g,(i_0)^n})$.
\end{lemma}

\begin{proof}
	It suffices to prove the lemma for $H$ subnormal and non-trivial. Note that $\rho_{p,g,(i_0)^n}(\Mod(\Sigma_{g}^{n}))$
	is dense in $\mathrm{PSU}(V_{p,g,(i_0)^n})$ by Theorem \ref{thm:density}.
	
	By Theorem \ref{thm:asymptFaith}, $\rho_{p,g,(i_0)^n}(H)$ is non-trivial for $p$ large enough. Therefore, $\overline{\rho_{p,g,(i_0)^n}(H)}$
	is a closed non-trivial and subnormal subgroup of $\mathrm{PSU}(V_{p,g,(i_0)^n})$ for $p$ large enough. Since $\mathrm{PSU}(V_{p,g,(i_0)^n})$
	is a simple Lie group, we conclude that $\overline{\rho_{p,g,(i_0)^n}(H)}=\mathrm{PSU}(V_{p,g,(i_0)^n})$ for $p$ large enough.
\end{proof}

\begin{proof}[Proof of Corollary \ref{cor:subnormalCore}]
	It suffices to show that these subgroups do not have dense image by the representations $\rho_{p,g,(i_0)^n}$ for $p$ large.
	For the handlebody groups $\mathcal{H}_g$ we note that their images by $\rho_{p,g}$ all fix the vacuum vector,
	which is the vector $RT_p(H_g,\emptyset)$ associated by the Reshetikhin-Turaev TQFT to the empty handlebody.
	
	For the subgroups $C_{\Mod(\Sigma_{g}^{n})}(f)$ where $f$ has finite order, since $\rho_{p,g}(f)$ is non-trivial for $p>>1$,
	the closure of $\rho_{p,g}(C_{\Mod(\Sigma_{g}^{n})}(f))$ stabilizes a non-trivial decomposition
	$V_{p,g,(i_0)^n}=\underset{\lambda}{\bigoplus}V_{\lambda}$ into eigenspaces of $\rho_{p,g,(i_0)^n}(f)$.
	
	Similarly, when $\Gamma$ is a non-central finite subgroup of $\Mod(\Sigma_{g}^{n})$, by asymptotic faithfulness,
	for $p$ large enough we get a non-trivial decomposition
	 $$V_{p,g,(i_0)^n}=\underset{\chi}{\bigoplus}V_{\chi}$$
	 into $\Gamma$-isotypic subrepresentations, and any element of $\overline{\rho_{p,g,(i_0)^n}(N_{\Mod(\Sigma_{g}^{n})}(\Gamma))}$
	 leaves the decomposition invariant (perhaps permuting subspaces).
\end{proof}
\subsection{Mapping class groups are residually finite simple}
\label{sec:resFinite}

We recall that a group $G$ is said to be residually finite simple if there exists a sequence of epimorphisms
$\left(\rho_i:G \twoheadrightarrow G_i\right)_{i\in I}$, where the groups $G_i$ are non-abelian finite simple groups,
and such that $\underset{i\in I}{\bigcap}\ker \rho_i=\lbrace 1_G \rbrace$.
\begin{theorem}\label{thm:resSimple}
	Let $g\geq 1$ and $n\geq 0$ such that $(g,n)\neq (1,0),(1,2),(1,3),(2,0)$ or $(2,1)$.
	Then the pure mapping class group $\PMod(\Sigma_{g,n})$ of a closed surface of genus $g$ with $n$ marked points is residually finite simple.
	
	The same is true for any subgroup of $\PMod(\Sigma_{g,n})$ containing a non-trivial subnormal subgroup.
\end{theorem}
We note that in the case $(g,n)=(1,0)$ or $(2,0)$, the same is true of the quotient of $\Mod(\Sigma_g)$ by its center.
The cases $(1,2)$, $(1,3)$ and $(2,1)$ are excluded because our proof or tensor-indecomposability (\Cref{thm:tensor-indec}) does not apply,
and hence neither does our proof of simplicity of the image (\Cref{thm:simpleGroup}).
The case of surfaces without marked points was proved in \cite{MR12},
and the above theorem answers a question of Masbaum and Reid (see \cite[§6]{MR16}).

\begin{proof}
	Note that $\PMod(\Sigma_{g,n})=\Mod(\Sigma_{g}^{n})/Z(\Mod(\Sigma_{g}^{n}))$ if $(g,n)\neq (1,0)$ or $(2,0)$.
	By Corollary \ref{cor:asymptFaith}, we have
	$$\underset{p,J}{\bigcap} \ker \rho_{J,g,(i_0)^n}=Z(\Mod(\Sigma_{g}^{n})),$$
	where the above intersection is for all primes $p\geq 7$ and all maximal ideals $J\neq (1-\zeta_p)$.
	However, by \Cref{thm:simpleGroup}, the image of $\rho_{J,g,(i_0)^n}$ is a non-abelian finite simple group.
	Note that by Lemma \ref{lemma:permissive1}, $(i_0)^n$ is permissive for any $n\geq 2$.
	
	Therefore, $\PMod(\Sigma_{g,n})$ is residually finite simple.
	
	To extend to the case of a subgroup $H$ containing a non-trivial subnormal subgroup $K$, notice that $\rho_{J,g,(i_0)^n}(K)$
	is non-trivial for any non-trivial subgroup $H$ and any $p,J$ large enough by Theorem \ref{thm:asymptFaith};
	since $K$ is subnormal, since $\rho_{J,g,(i_0)^n}$ is surjective on $\PMod(\Sigma_{g,n})$, and since the image is simple,
	we conclude that $\rho_{J,g,(i_0)^n}(K)=\rho_{J,g,(i_0)^n}(H)=\rho_{J,g,(i_0)^n}(\PMod(\Sigma_{g,n}))$ is simple for $p,J$ large enough.
	We conclude again by asymptotic faithfulness that $H$ is residually finite simple (see Remark \ref{rk:asympFaith}).
\end{proof}

\subsection{Realizability of congruence classes of invariants of links and embedded graphs in \texorpdfstring{$\bm{3}$}{3}-manifolds and
applications to embeddings between \texorpdfstring{$\bm{3}$}{3}-manifolds}
\label{sec:embeddings}

For $\Gamma$ a banded trivalent graph (not necessarily connected) and $p\geq 5$ a prime number, let $\Delta_{\Gamma}(p)$
be the set of maps $\lambda:E(\Gamma)\longrightarrow \Lambda_p=\lbrace 0,2,\ldots,p-3\rbrace$,
such that the triple of colors $(\lambda(a),\lambda(b),\lambda(c))$ is $p$-admissible for any triple $(a,b,c)$ of edges adjacent to the same vertex.
\begin{theorem}\label{thm:realizability}
	Let $p\geq 5$ be a prime number, let $M$ be a closed compact oriented $3$-manifold and let $\Gamma$ be a banded trivalent graph.
	Let also $J\neq (1-\zeta_p)$ be a maximal ideal in $\Z[\zeta_p]$. Then for any map
	$$Z:\Delta_{\Gamma}(p)\setminus \lbrace 0 \rbrace \longrightarrow \Z[\zeta_p]/J$$
	there exists an embedding of $\Gamma$ in $M$, such that for any $\ul \in \Delta_{\Gamma}(p)\setminus\{0\}$,
	$$Z(\ul)\equiv RT_p(M,\Gamma,\ul)\mod{J}.$$
\end{theorem}

One significance of Theorem \ref{thm:realizability} is that, thanks to the theory of Frohman-Kania-Bartoszynska ideals \cite{FKB},
it allows us to construct complements of handlebodies in a given $3$-manifold $M$ that do not embed in another $3$-manifold $M'$,
if $M$ and $M'$ satisfy some mild conditions:
\begin{corollary}\label{cor:nonEmbedding}
	Let $M$ and $M'$ be two closed oriented $3$-manifolds, and let $H$ be a disjoint union of handlebodies (of genus at least $1$).
	Assume that for some prime number $p\geq 5$ and some maximal ideal $J$ in $\Z[\zeta_p]$ different from $(1-\zeta_p)$,
	one has $RT_p(M)\in J$ but $RT_p(M')\notin J$. Then there exists an embedding of $H$ in $M$ such that $M\setminus H$ does not embed in $M'$.
\end{corollary}
\begin{proof}
	Recall from \cite{FKB} that, for a given prime number $p\geq 5$, the Frohman-Kania-Bartoszynska ideal $I_p(N)$ of a
	$3$-manifold with boundary is the $\Z[\zeta_p]$-ideal spanned by the invariants $RT_p(M)$ where $M$ runs over all closed
	$3$-manifolds in which $N$ embeds. For a maximal ideal $J\subset \Z[\zeta_p]$ different from $(1-\zeta_p)$
	and such that the vector $RT_p(N) \in RT_p(\partial N)$ has coefficients in $J$ in a BHMV basis, one has $I_p(N)\subset J$.
	
	Now, let $M$ and $M'$ satisfy the hypotheses of the corollary.
	Pick a trivalent graph $\Gamma$ embedded in $M$ as in the conclusion of Theorem \ref{thm:realizability} with $Z$ identically $0$,
	such that a regular neighborhood of $\Gamma$ is homeomorphic to $H$. Then setting $N:=M\setminus H$,
	the coefficients of $RT_p(N)$ in a BHMV basis are the $RT_p(M,\Gamma,\ul)$ for $\ul \in \Delta_{\Gamma}(p)$.
	They are all in $J$ by construction, since $RT_p(M,\Gamma,0)=RT_p(M)\in J$. Thus $I_p(N)\subset J$,
	but since $RT_p(M')\notin J$, we conclude that $N$ does not embed in $M'$.
\end{proof}

The proof of Theorem \ref{thm:realizability} will rely on the following Lemma:

\begin{lemma}
	\label{lemma:braidSurjectivity} Let $p\geq 5$ be a prime number,
	let $J\neq (1-\zeta_p)$ be a maximal ideal in $\Z[\zeta_p]$ and
	let $\Sigma_{g,2n}$ be a surface of genus $g\geq 3$ with $2n$ marked points, $n\geq 1$.
	Let us view the pure braid group $P_{2n}(\Sigma_g)$ as a subgroup of $\PMod(\Sigma_{g,2n})$. Then
	$$\underset{\ul\in \Lambda^n\setminus \lbrace 0 \rbrace}{\prod}\rho_{J,g,(\ul,\ul)}|_{P_{2n}(\Sigma_g)}\colon%
	P_{2n}(\Sigma_g)\longrightarrow \underset{\ul\in \Lambda^n\setminus \lbrace 0 \rbrace}{\prod} \mathbf{PG}_{d_{p,g,(\ul,\ul)}}(\Z[\zeta_p]/J)$$
	is surjective, where $\mathbf{PG}=\mathrm{PSU}$ or $\mathrm{PSL}$, depending on whether $\overline{J}=J$ or not.
\end{lemma}

\begin{proof}
	Let $1\leq i\leq n$. Let us first explain how to read out $\ul$ from eigenspaces of some elements in the image of
	$\rho_{J,g,(\ul,\ul)}|_{P_{2n}(\Sigma_g)}$.
	Let $\gamma$ be a simple closed curve in $\Sigma_{g,2n}$ that separates a disk containing only
	the two marked points $i$ and $n+i$, which are both labeled with $\lambda_i$.
	Then, cutting along $\gamma$ and denoting by $\ul'$ the tuple $\ul$ with the $i$-th component removed, we have
	the following decomposition into eigenspaces for the action of the Dehn twist $T_\gamma$:
	\begin{equation}\label{eq:dimSum}
	V_{J,g,(\ul,\ul)}=\underset{(\lambda_i,\lambda_i,\mu) \ p\textrm{-admissible}}{\bigoplus} V_{J,g,(\ul',\ul',\mu)}.
	\end{equation}
	with each term non-zero by \Cref{lemma:dimAtLeastOne} since $g\geq 1$.
	However, the set $S_{\lambda_i}:=\lbrace \mu \ | \ (\lambda_i,\lambda_i,\mu) \ p\textrm{-admissible}\rbrace$ is
	$$S_{\lambda_i}=\lbrace 0, 2,\ldots ,2\lambda_i \rbrace$$
	if $\lambda_i<i_0$, and
	$$S_{\lambda_i}=\lbrace 0, 2,\ldots, (2p-4)-2\lambda_i\rbrace$$
	if $i>i_0$.
	Note that $S_{\lambda_i}$ contains an odd number of elements in the former case,
	and an even number in the latter case. Finally, $S_{i_0}=\Lambda$ by Lemma \ref{lemma:permColor}.
	We deduce that $\lambda_i\mapsto |S_{\lambda_i}|$ is injective. Hence
	$\lambda_i$ can be read by counting the number of eigenspaces of $T_\gamma$.
	As $T_\gamma$ is in $P_{2n}(\Sigma_g)$, we have that $\ul$ can be read
	from $\rho_{J,g,(\ul,\ul)}|_{P_{2n}(\Sigma_g)}$ by
	counting eigenspaces.

	Assume by contradiction that the statement of the lemma is false.
	Let $L\subset \Lambda^n\setminus\{0\}$ be a subset with minimal cardinality such that
	${\prod}_{\ul\in L}\rho_{J,g,(\ul,\ul)}|_{P_{2n}(\Sigma_g)}$ is not surjective.
	Set $L=L'\sqcup\{\ul\}$. By \Cref{thm:surjectivity}, as $P_{2n}(\Sigma_g)$ is normal in $\PMod(\Sigma_{g,2n})$,
	$\rho_{J,g,(\ul,\ul)}(P_{2n}(\Sigma_g))$ is either
	$\mathbf{PG}_{d_{p,g,(\ul,\ul)}}(\Z[\zeta_p]/J)$ or trivial.
	As $\ul\neq 0$, at least one of the twists $T_\gamma\in P_{2n}(\Sigma_g)$ of the previous paragraph
	has at least $2$ eigenspaces.
	Hence, $\rho_{J,g,(\ul,\ul)}(P_{2n}(\Sigma_g))$ is non-trivial and $\rho_{J,g,(\ul,\ul)}|_{P_{2n}(\Sigma_g)}$ is surjective.
	Thus $L'\neq \emptyset$. Let $H$ denote the image of
	$\underset{\ul\in L}{\prod}\rho_{J,g,(\ul,\ul)}|_{P_{2n}(\Sigma_g)}$.
	It is a subgroup of
	$$G_1\times G_2=\left(\underset{\ul'\in L'}{\prod}\mathbf{PG}_{d_{p,g,(\ul',\ul')}}(\Z[\zeta_p]/J)
	\right)\times \mathbf{PG}_{d_{p,g,(\ul,\ul)}}(\Z[\zeta_p]/J)$$
	which surjects onto both factors, by minimality of $|L|$.
	As in the proof of \Cref{lemma:surjectivityProduct}, we argue, using Hall's diagonal Lemma,
	that $H=G_1\times G_2$.
	To do this, we need to show that no two distinct maps
	$\rho_{J,g,(\ul,\ul)}|_{P_{2n}(\Sigma_g)}$, $\rho_{J,g,(\ul',\ul')}|_{P_{2n}(\Sigma_g)}$
	are related by post-composition with an isomorphism
	$\phi\colon\mathbf{PG}_{d_{p,g,(\ul',\ul')}}(\Z[\zeta_p]/J)\simeq \mathbf{PG}_{d_{p,g,(\ul,\ul)}}(\Z[\zeta_p]/J)$.
	Assume otherwise. Then there exists $\sigma\in\mathrm{Aut}(\Z[\zeta_p]/J)$ such that $\phi$
	is induced by an isomorphism $V_{J,g,(\ul',\ul')}\simeq \sigma(V_{J,g,(\ul,\ul)})$
	or an isomorphism $V_{J,g,(\ul',\ul')}\simeq \sigma(V_{J,g,(\ul,\ul)})^*$
	\cite[§3.3]{conwayAtlasFiniteGroups1985}. In particular, for any $k\in P_{2n}(\Sigma_g)$,
	$\rho_{J,g,(\ul',\ul')}(k)$ and $\rho_{J,g,(\ul,\ul)}(k)$ have the same number of eigenvalues.
	By the above derivation, this implies $\ul'=\ul$, contradiction.
\end{proof}

\begin{proof}[Proof of Theorem \ref{thm:realizability}]
	Let $\Gamma$ be the banded trivalent graph. Pick a Heegaard splitting $M=H_1\underset{\Sigma}{\bigcup} H_2$ of genus at least $3$.
	Pick a $3$-ball $B$ that intersects $\Sigma$ in a single disk $D$.
	We embed $\Gamma$ in $B$, then isotope $\Gamma$ so that $\Gamma\cap H_2$ is a neighborhood in $\Gamma$ of the vertices of $\Gamma$,
	while $\Gamma \cap H_1$ is a neighborhood in $\Gamma$ of the midpoints of edges.
	(If $\Gamma$ has closed loops among its components, we want one arc of each loop in each of $H_1$ and $H_2$).
	
	With this setting, we can arrange that $(H_i,\Gamma\cap H_i,\ul)$ is a non-zero vector in $RT_p(\Sigma,\Gamma\cap \Sigma,\ul) \mod J$
	for $i=1$ or $2$ and any $\ul \in \Delta_{\Gamma}(p)$.
	For $i=2$, we can arrange that the different connected components of $\Gamma\cap H_2$ are embedded in disjoint balls in the ball $B$.
	The vectors $(H_i,\Gamma\cap H_i,\ul)$ are then primitive over $\Z[\zeta_p,\frac 1 p]$,
	as they may be seen as part of a BHMV basis of $RT_p(\Sigma,\Gamma\cap \Sigma,\ul)$.
	For $i=1$, we may change the embedding of $\Gamma\cap H_1$ in $B\cap H_1$ so that
	the closure of $\Gamma\cap H_1\subset B\cap H_1$ onto its mirror is a disjoint union of unlinks in the ball
	$(B\cap H_1)\cup_D \overline{(B\cap H_1)}$. This implies that the norms $\langle(H_1,\Gamma\cap H_1,\ul),(H_1,\Gamma\cap H_1,\ul)\rangle$
	are invertible in $\Z[\zeta_p,\frac 1 p]$, as they are, up to units,
	products of invariants of unlinks. Hence, the $(H_1,\Gamma\cap H_1,\ul)$ are primitive.
	
	We now modify the embedding of $\Gamma$ in $M$ by cutting along $\Sigma$ and regluing using a suitably chosen element $\beta$ of $P_n(\Sigma)$.
	The ambient manifold is still $M$ after this operation, since elements of $P_n(\Sigma)$ become trivial after forgetting the action on punctures.
	We choose the element $\phi$ so that
	$$\langle RT_p(H_1,\Gamma\cap H_1,\ul),\rho_{J,g,\ul}(\beta)(RT_p(H_2,\Gamma\cap H_2,\ul))\rangle \equiv Z(\ul) \mod J$$
	for any $\ul \in \Delta_{\Gamma}(p)\setminus \lbrace 0 \rbrace$.
	
	The existence of such an element is guaranteed by Lemma \ref{lemma:braidSurjectivity}, the non-vanishing of the vectors $RT_p(H_i,\Gamma\cap H_i,\ul)$
	and the non-degeneracy of the form $\langle ,\rangle$ mod $J$.
\end{proof}
\subsection{Characteristic subgroups of \texorpdfstring{$\bm{\pi_1(\Sigma_g)}$}{the fundamental group of a genus g surface} with trivial outer action
of \texorpdfstring{$\bm{\Mod(\Sigma_g)}$}{the mapping class group}}
\label{sec:congSubProp}

By the Dehn-Nielsen-Baer theorem, the mapping class group $\Mod(\Sigma_g)$ of a closed surface of genus $g$ is isomorphic to the group
$\mathrm{Out}(\pi_1(\Sigma_g))$. Note that from the Birman exact sequence
$$1\longrightarrow \pi_1(\Sigma_g)\longrightarrow \Mod(\Sigma_{g,1})\simeq \mathrm{Aut}^+(\pi_1(\Sigma_g))\longrightarrow%
\Mod(\Sigma_g)\simeq \mathrm{Out}^+(\pi_1(\Sigma_g)) \longrightarrow 1$$
we get that any characteristic subgroup $\Gamma$ of $\pi_1(\Sigma_g)$ gives rise to maps
$$\Mod(\Sigma_{g,1})\longrightarrow \mathrm{Aut}(\pi_1(\Sigma_g)/\Gamma)$$
and
$$\Mod(\Sigma_g)\longrightarrow \mathrm{Out}(\pi_1(\Sigma_g)/\Gamma).$$
A famous question of Ivanov \cite[§1]{Iva}, known as the \textit{Congruence Subgroup Property} for mapping class groups of surfaces, asks the following:

\begin{question}
	\label{ques:congSubProp} Let $g\geq 3$. Is it true that any finite index subgroup of $\Mod(\Sigma_g)$ contains the kernel of a map
	$$\Mod(\Sigma_g)\longrightarrow \mathrm{Out}(\pi_1(\Sigma_g)/\Gamma)$$
	where $\Gamma$ is some finite index characteristic subgroup of $\pi_1(\Sigma_g)$.
\end{question}

In contrast to Question \ref{ques:congSubProp}, a strange consequence of our main results is the following:

\begin{theorem}
	\label{thm:congSubProp} Let $g\geq 3$. There exists a sequence $(\Gamma_n)_{n\geq 0}$ of finite index characteristic subgroups of $\pi_1(\Sigma_g)$
	such that $\underset{n\geq 0}{\cap}\Gamma_n=\lbrace 1\rbrace$ and such that the associated maps
	$$\Mod(\Sigma_g)\longrightarrow \mathrm{Out}(\pi_1(\Sigma_g)/\Gamma_n)$$
	are all trivial.
\end{theorem}

\begin{proof}
	Let us consider the representations $\rho_{J,g,(i_0)}$ of $\Mod(\Sigma_{g,1})$ for $p\geq 7$ prime and $J\neq (1-\zeta_p) \subset \Z[\zeta_p]$
	a maximal ideal.
	By Theorem \ref{thm:surjectivity}, these maps are surjective onto $G=\mathrm{PSL}(V_{J,i_0})$ or $\mathrm{PSU}(V_{J,i_0})$.
	Thanks to \Cref{thm:asymptFaith} and simplicity of $G$, for $p$ and $J$ large enough,
	the restriction to the subgroup $\pi_1(\Sigma_g)$, which is normal in $\Mod(\Sigma_{g,1})$, is also surjective onto $G$.
	Moreover, from the Skein construction of quantum representations, one can see that
	any orientation-reversing diffeomorphism of $\Sigma_{g,1}$ conjugates $\ker\rho_{J,g,(i_0)}$ into $\ker\rho_{\overline{J},g,(i_0)}$,
	which equals $\ker\rho_{J,g,(i_0)}$ since $\rho_{J,g,(i_0)}$ and $\rho_{\overline{J},g,(i_0)}$ are dual.
	As the group of (potentially orientation reversing)
	diffeomorphisms of $\Sigma_{g,1}$ up to isotopy is isomorphic to $\mathrm{Aut}(\pi_1(\Sigma_g))$,
	the subgroups $\Gamma_{J}:=\ker \rho_{J,g,(i_0)}|_{\pi_1(\Sigma_g)}$ are characteristic subgroups of $\pi_1(\Sigma_g)$, and moreover,
	the natural maps
	$$\Mod(\Sigma_{g,1})\longrightarrow \mathrm{Aut}(\pi_1(\Sigma_g)/\Gamma_J)$$
	take values in the inner automorphisms of $\pi_1(\Sigma_g)/\Gamma_J\simeq G$. Therefore, the natural maps
	$$\Mod(\Sigma_{g})\longrightarrow \mathrm{Out}(\pi_1(\Sigma_g)/\Gamma_J)$$
	are all trivial. The fact that $\underset{J \ large}{\cap}\Gamma_J=\lbrace 1 \rbrace$ follows once more from Theorem \ref{thm:asymptFaith}.
\end{proof}

We remark that the above construction does not provide counterexamples to Question \ref{ques:congSubProp}:
this construction can also be made for the groups $\PMod(\Sigma_{1,n})$ (\Cref{thm:resSimple}),
but these groups have the congruence property \cite[§2.1]{asadaFaithfulnessMonodromyRepresentations2001}.
It would be interesting to verify the Congruence Subgroup Property for the subgroups $\ker \rho_{J,g}$ of $\Mod(\Sigma_g)$.
\bibliographystyle{hamsalpha}
\bibliography{biblio}
\end{document}

%% file: tikzstyles.tikzstyles
\tikzstyle{green}=[text={black!30!green}]
\tikzstyle{blue}=[text=blue]
\tikzstyle{red}=[text=red]
\tikzstyle{puncture}=[fill=white, draw=red, shape=circle, minimum size=1pt]
\tikzstyle{blackpuncture}=[fill=white, draw=black, shape=circle, minimum size=1pt]
\tikzstyle{cyan}=[text=cyan]

\tikzstyle{markings}=[-, draw=red, line width=1pt, line cap=round]
\tikzstyle{overbraid}=[-, draw=white, fill=none, line width=6pt]
\tikzstyle{thick}=[-, line width=2pt, draw=blue]
\tikzstyle{dashedline}=[-, dashed]
\tikzstyle{dottedline}=[-, dash pattern=on 0.75pt off 0.75pt, line width=0.75pt]
\tikzstyle{thin red}=[-, line width=0.25pt, draw=black]
\tikzstyle{tangle}=[-, draw=blue, line width=1pt, fill={blue!20}]
\tikzstyle{scc}=[-, draw={black!30!green}, fill={blue!20}, line width=1pt]
\tikzstyle{inner boundary}=[-, fill=white]
\tikzstyle{outer boundary}=[-, fill={red!20}]
\tikzstyle{lowerboundery}=[-, line width=1.5pt, line cap=round, draw=red]
\tikzstyle{upperboundery}=[-, line width=1.5pt, line cap=round, draw=blue]
\tikzstyle{dottedcycle}=[-, draw=blue, dash pattern=on 0.5pt off 1pt on 4pt off 1pt, decoration={markings, mark=at position 0.5 with {\arrow{>}}}, postaction=decorate]
\tikzstyle{cycle}=[-, draw=blue, decoration={markings, mark=at position 0.5 with {\arrow{>}}}, postaction=decorate]
\tikzstyle{path}=[-, draw=cyan, line width=0.25pt]
\tikzstyle{arrowpath}=[-, draw=cyan, line width=0.25pt, decoration={markings, mark=at position 0.5 with {\arrow{>}}}, postaction=decorate]
\tikzstyle{orientedpath}=[-, line width=0.25pt, decoration={markings, mark=at position 0.5 with {\arrow{<}}}, postaction=decorate]
\tikzstyle{inner square}=[-, fill={blue!20}]
\tikzstyle{outer square}=[-, fill={red!20}]
\tikzstyle{blueline}=[-, draw=blue]
\tikzstyle{greenline}=[-, draw=green]
\tikzstyle{bluesquare}=[-, draw=blue, fill={blue!20}]